\documentclass[12pt]{amsart}

\usepackage[utf8]{inputenc}
\usepackage[T1]{fontenc} 
\usepackage{graphicx} 
\usepackage{amsmath,amsthm}
\usepackage{amssymb}
\usepackage{standalone}
\usepackage{tikz}
\usetikzlibrary{calc}
\usepackage{pgfplots}
\usepackage{pst-plot}
\usepackage{algorithm}
\usepackage{algpseudocode}
\usepackage{verbatim}
\usepackage{float}
\usepackage{mathtools}
\usepackage{float,lipsum}
\usepackage{bbm}
\usepackage[english]{babel}
\usepackage[
backend=biber,
style=alphabetic,
]{biblatex}

\newtheorem{theorem}{Theorem}
\newtheorem{lemma}{Lemma}

\newtheorem{proposition}{Proposition}
\newtheorem{assumption}{Assumption}
\newtheorem{definition}{Definition}
\theoremstyle{remark}
\newtheorem{remark}{Remark}
\usepackage{breqn}

\usepackage{blindtext}
\usepackage{hyperref}
\DeclareFontFamily{U}{mathx}{}
\DeclareFontShape{U}{mathx}{m}{n}{<-> mathx10}{}
\DeclareSymbolFont{mathx}{U}{mathx}{m}{n}
\DeclareMathAccent{\widehat}{0}{mathx}{"70}
\DeclareMathAccent{\widecheck}{0}{mathx}{"71}

\newcommand{\erf}{\operatorname{erf}}
\newcommand{\erfc}{\operatorname{erfc}}
\newcommand{\be}{\begin{equation}}
\newcommand{\ee}{\end{equation}}

\newcommand{\defeq}{\vcentcolon=}

\title[Spectral Interference]{On Spectral Interference of the short-time Fourier transform and its nonlinear variations}

\author[Chand]{Shrikant Chand}
\address{Department of Mathematics, Duke University, Durham, NC, USA}
\email{shrikant.chand@duke.edu}

\author[Nolen]{James Nolen}
\address{Department of Mathematics, Duke University, Durham, NC, USA}
\email{james.nolen@duke.edu}

\author[Wu]{Hau-Tieng Wu}
\address{Department of Mathematics, Courant Institute of Mathematical Sciences, 
New York University, New York, NY, USA}
\email{hw3635@nyu.edu}

\begin{document}

\maketitle

\begin{abstract}
Spectral interference, the frequency counterpart of the \textit{beating phenomenon} in the time domain, can severely distort time–frequency representations (TFRs) in physical applications. We study this phenomenon for the short-time Fourier transform (STFT) with a Gaussian window and for nonlinear refinements based on the reassignment method, with an emphasis on the synchrosqueezing transform (SST). Working with a two-component harmonic model, we quantify when STFT can (and cannot) resolve two nearby frequencies: a sharp transition occurs at a critical gap that scales inversely to kernel bandwidth and depends explicitly on the amplitude ratio. Below this threshold, the spectrogram ridges undergo bifurcation and form repeating \textit{time–frequency bubbles}, which we describe asymptotically and, in the balanced-amplitude case, approximate closely by ellipses. We then analyze the STFT phase, showing a canonical winding behavior, and relate the complex-valued SST reassignment map to a holomorphic structure via the Bargmann transform. In the two-component setting the reassignment rule admits an explicit Möbius-geometry description, sending frequency lines to circular arcs in the complex plane. Finally, viewing SST and reassignment through a measure mapping perspective, we derive small-kernel asymptotics that explain when reassignment sharpens energy and when it produces distorted or misleading TFRs; we also introduce a generalized synchrosqueezing framework that isolates the role of STFT weighting and clarifies how alternative choices can mitigate interference in certain regimes.
\end{abstract}

\section{Introduction}

Time series are a ubiquitous data type in modern science and engineering, arising in applications as diverse as physiology, geophysics, finance, and engineering. A common modeling viewpoint is that a time series is formed by superimposing several oscillatory components, each carrying dynamical information from a different source, possibly contaminated by noise. When the signal is composed of multiple oscillatory components with time-varying amplitudes and/or frequencies, time-frequency (TF) analysis tools such as the short-time Fourier transform (STFT) \cite{grochenig} and its nonlinear refinements, e.g., synchrosqueezing transform (SST) \cite{sst}, can be considered to explore the "local spectrum" of the signal. However, when distinct oscillatory components have close, or even crossing, instantaneous frequencies (IFs), the resulting TF representation (TFR) exhibits complicated patterns, which we refer to as {\em spectral interference}.

We shall emphasize that to our knowledge, spectral interference does not have a universally accepted definition, and its meaning varies widely depending on the data type and application domain. We specifically focus on time series literature here. Consider the two-component signal $f(t)=\cos(2\pi t)+a\cos(2\pi \phi(t))$, where $\phi$ is smooth with $\phi'(t)>0$ and $a>0$. If $\phi''(t)>0$ and $\phi'(t_0)=1$, there is a clear crossover of frequencies at time $t_0$, which we may regard as spectral interference {\em on the model level}. This type of model can be analyzed using the chirplet transform \cite{mann1995chirplet} and is out of the scope of this paper. 
When IFs do not cross but approach each other, the signal in the time domain exhibits a characteristic behavior known as the \textit{beating phenomenon} \cite{1or2emd}, which is the time-domain counterpart of spectral interference. To be more specific, when $a=1$ and $\phi(t)= (1+\Delta)t$, a simple trigonometric identity yields 
\begin{equation}\label{equation beating in time domain}
f(t) = 2\cos(\pi \Delta t)\cos(2\pi(1+\Delta/2)t)\,, 
\end{equation}
where it is common to call $\cos(\pi \Delta t)$ the ``beating envelope''.

From an algorithmic perspective, when $\Delta>0$ is small, the two oscillatory components cannot be distinguished by the Fourier transform if the sampling length $T > 0$ is insufficiently long, nor by TF analysis methods, such as the STFT or SST, if the analysis window is too short. We may regard this behavior as spectral interference {\em on the algorithm level}. There exist many efforts exploring how different algorithms behave under such setup \cite{genton2007statistical,1or2emd,1or2sst,lostanlen2021one,wu2024frequency,cicone2024one}.
When the sampling rate is low \cite{meyer1992wavelets}, or even nonuniform \cite{lin2015interpolation}, aliasing may appear, which can be considered as spectral interference {\em on the sampling level}. Our focus in this work is to understand spectral interference on the algorithm level, particularly as it arises in TF analysis methods.

In this paper, we are interested in spectral interference on the algorithm level when two oscillatory components' IF get close. Specifically, we are interested in spectral interference generated by the application of TF analysis algorithms, particularly STFT and SST. 
The uncertainty principle dictates that no choice of window allows STFT to simultaneously produce sharp localization in both time and frequency. As a result, when two oscillatory components have nontrivial instantaneous frequencies that are close over some time interval, the spectral contents of these two components may overlap, leading to spectral interference.
In practice, this manifests as blurred or merged ridges in the spectrogram, spurious side lobes, and oscillatory patterns that depend sensitively on the choice of window and sampling parameters. 
Spectral interference is complicated when it appears near the boundary. Near boundary regions, the TFR often exhibits nontrivial transition behavior due to incomplete data, which is usually called the \textit{boundary effect}.   Since the boundary effect is out of the scope of spectral interference, we do not discuss it in this paper and postpone it to our future research.

We analyze the beating phenomenon when applying STFT and SST under a simple yet representative two-component model. For STFT with a Gaussian window, we show that the associated spectrogram--the squared magnitude of STFT--exhibits a specific ridge-like structure near the interference region. We derive asymptotic expansions about the ``bubble'' formed in the ridge, or local maximum, that make this structure explicit and quantify how it depends on the frequency gap and model parameters. 
Next, we turn to the STFT phase and the associated reassignment rule, which provides complementary information to the spectrogram and its ridges, and is the foundation of SST. We uncover a holomorphic structure in the complex valued synchrosqueezing reassignment map and characterize it using M\"obius geometry from complex analysis. 
Finally, we investigate how spectral interference affects SST via reassignment. Since SST is built by nonlinearly reallocating STFT energy according to reassignment, our analysis reveals when and how spectral interference leads to sharpened, distorted, or even misleading representations. All together, these results provide a unified picture of spectral interference at the algorithm level for both linear and nonlinear TF analysis.

The remainder of the paper is organized in the following way: Section \ref{sec:math_model} defines the 2 component model we study and clarifies the particular form of spectral interference we are interested in; Section \ref{sec:linear_methods} discusses the STFT and how the separation between frequencies plays an explicit role in spectral interference, as well as defining ridges and mathematically describing the ellipses that occur in STFT ridges for the two-component model; Section \ref{sec:phase} discusses and compares both phase based and synchrosqueezing reassignment and uses M\"obius geometry to paint a clearer picture of how reassignment works; Section \ref{sec:nonlinear_methods} describes spectral interference for the SST and shows how stationary phase methods can be used to understand leading order behavior of the SST, while also proposing a generalized synchrosqueezing transform that can be used to sharpen a broader class of TF representations.

\section{Mathematical Model}\label{sec:math_model}

Recall functions satisfying the \textit{adaptive harmonic model} \cite{sst}.
\begin{definition}[Adaptive Harmonic Model]\label{Definition: AHM}
    A function satisfying the adaptive harmonic model (AHM) \cite{ahm} is one of the form
    \[
    f(t) = \sum_{k=0}^{K-1} A_k(t) \cos(2\pi \phi_k(t)),
    \]
    where $K\in \mathbb{N}$, 
    \begin{gather*}
        A_k\in C^1(\mathbb{R})\cap L_{\infty}(\mathbb{R}), \qquad \phi_k\in C^2(\mathbb{R}),\\
        0<\inf_{t\in\mathbb{R}} \phi_k'(t) \leq\sup_{t\in\mathbb{R}}\phi_k'(t)<\infty, \qquad M_j''\defeq \sup_x |\phi_j''(x)|<\infty,\\
        |A_k'(t)|,\, |\phi_k''(t)|\leq \epsilon|\phi_k'(t)|, \quad \forall t\in\mathbb{R}\,,
    \end{gather*}
and the frequency separation condition is satisfied; that is, $\phi_k'(t) - \phi'_{k-1}\geq\Delta>0$ for $k=1,\ldots,K-1$. 
\end{definition}

We call $A_k(t)$, $\phi_k(t)$, and $\phi'_k(t)$ the amplitude modulation (AM), the phase, and the instantaneous frequency (IF) of the $k$-th intrinsic mode type (IMT) function, $A_k(t) \cos(2\pi \phi_k(t))$, respectively, and $\Delta$ the spectral gap. When the spectral gap $\Delta\to 0$, two components have close frequencies at some times and experience the beating phenomenon, and the STFT encounters spectral interference.
In this paper, we are interested in understanding the behavior of spectral interference when the STFT is applied to analyze the above class of functions with $K=2$; that is, we focus on the 2-component case as a simplified model. Furthermore, since spectral interference is a local phenomenon, we further simplify the analysis by considering $\phi_k(t)$ to be linear and $A_k(t)$ to be constant, since locally near some fixed $t_0$, Taylor expansion tells us that the above components satisfy $A_k(t)=A_k(t_0) + O(\varepsilon|t-t_0|)$ and $\phi_k'(t)=\phi'_k(t_0)+O(\varepsilon|t-t_0|)$. Thus, we first focus on studying the {\em two complex harmonic model}--that is,
\be\label{eqn:f_def}
f(t)=f_0(t)+f_1(t)\,,
\ee
where $f_0(t)\defeq e^{2\pi i \xi_0 t}$,  $f_1(t)\defeq ae^{2\pi i \xi_1 t}$, $a>0$ is the amplitude of the second component, and $\Delta\defeq \xi_1-\xi_0>0$ is the spectral gap between the two components $f_0$ and $f_1$. In later sections, we show how our analysis extends to the general AHM case via Taylor expansion. 
We should note that one could directly study the general AHM with $K=2$, though the required tedious local Taylor approximations may mask the main findings that follow. Notice that we could choose to include a phase deviation term (e.g. $f_1(t)=ae^{2\pi i \xi_1 t+\theta}$), however the analysis that follows can be easily adapted by a time shift, so we assume without loss of generality that $\theta=0$. We also denote the equally weighted average frequency as $\bar{\xi}\defeq (\xi_0+\xi_1)/2$. 

We provide a quantification of the beating phenomenon by showing that when $\Delta$ is smaller than some critical $\Delta_*$ that depends on the chosen window length, \textit{time-frequency bubbles} appear at regular locations in the TF domain \cite{delprat1997, 1or2ridges}. At times when such bubbles begin to appear, distinct curves (or \textit{ridges}) associated with IFs in the TFR merge together and it is challenging to recover these two components. We quantify when and how such bubbles form when applying both linear type (e.g. STFT) and non-linear type (e.g. SST) TF analysis algorithms. In what follows, all proofs may be found in the Appendix, unless stated otherwise.

\section{Linear Time-Frequency Methods}\label{sec:linear_methods}
Let $h$ be a Schwartz function and $f$ a tempered distribution. The \textit{modified STFT} of $f$ is defined by the windowed Fourier transform 
\be
V_f^{(h)}(t,\eta) = \int_{-\infty}^\infty f(x) h(x-t) e^{-2\pi i \eta (x-t)} dx.\label{stft_def}
\ee
This is a modification of the usual STFT by the phase factor $e^{2\pi i \eta t}$ to reduce later notation. Note that when the window $h$ is Gaussian, this is often referred to as the \textit{Gabor Transform}.

In the two component model,
by linearity of the map $f\mapsto V_f$, we have
\begin{align}
    V_f^{(h)}(t,\eta) &= V_{f_0}^{(h)}(t,\eta) + V_{f_1}^{(h)}(t,\eta)\notag\\
    &= e^{2\pi i \xi_0 t}\left(\hat h(\eta-\xi_0)+ae^{2\pi i \Delta t}\hat h(\eta-\xi_1)\right).\label{stft}
\end{align}

An immediate observation from \eqref{stft} is that, ignoring the global phase factor $e^{2\pi i \xi_0 t}$, the STFT of $f$ is comprised of the complex-valued weighted sum of two Fourier transforms of the window $h$ with centers $\xi_0$ and $\xi_1$.
We focus on the case where $h$ is a Gaussian window and for the rest of the paper, we invoke the following assumption.
\begin{assumption}
    Let $h(x)=\frac{1}{\sigma\sqrt{\pi}} e^{-\frac{x^2}{\sigma^2}}$ be a Gaussian with fixed temporal bandwidth $\frac{\sigma}{\sqrt{2}}.$
\end{assumption}
With this assumption, we see that \eqref{stft} becomes
\be\label{eqn:stft_simplified}
    V_f^{(h)}(t,\eta) =e^{2\pi i \xi_0 t}\left(e^{-\pi^2\sigma^2 (\eta-\xi_0)^2}+ae^{2\pi i \Delta t}e^{-\pi^2\sigma^2 (\eta-\xi_1)^2}\right),
\ee
since $\hat{h}(\omega-\mu)=e^{-\pi^2\sigma^2 (\omega-\mu)^2}$ is a Gaussian centered at $\mu$ with standard deviation $\sim 1/\sigma$. This tells us that the phase term $e^{2\pi i \Delta t}$ influences whether $|V_f^{(h)}(t,\eta)|$ is the modulus of the sum of two Gaussians (for times like $t=1/\Delta$), the difference of two Gaussians (for times like $t=1.5/\Delta$), or somewhere in between the two.

\subsection{When $\sigma\Delta$ is large}
It is important to note that while $f\mapsto V_f$ is a linear operator, the map $f\mapsto |V_f|$ is nonlinear. When $\sigma\Delta$ is large, this nonlinearity of the two harmonic model becomes weak, which is stated in the following proposition.

\begin{proposition}\label{prop:stft_large_Delta}
For all $t$ and $\eta$ and the Gaussian window, 
\[
\sup_{t,\eta}\left|\left|V_{f}^{(h)}(t,\eta)\right|-\left(\left|V_{f_0}^{(h)}(t,\eta)\right|+\left|V_{f_1}^{(h)}(t,\eta)\right|\right)\right|\leq 2ae^{-\pi^2\sigma^2(\Delta/2)^2}\xrightarrow[\sigma\Delta\to\infty]{}0,
\]
so as $\sigma\Delta$ grows the modulus $\left|V_f^{(h)}(t,\eta)\right|$ becomes uniformly close to $\left|V_{f_0}^{(h)}(t,\eta)\right|+\left|V_{f_1}^{(h)}(t,\eta)\right|$, independent of $t$.
\end{proposition}
The proof of this proposition is provided in Appendix \ref{sec:app_proofs}. This proposition reflects the uncertainty principle; that is, no matter how close two frequencies are, if we choose $\sigma$ sufficiently large, we can separate two harmonic components. On the other hand, if two frequencies are far different, a small $\sigma$ can still separate two harmonic components.

\subsection{When $\sigma\Delta$ is small}
When $\sigma\Delta$ is small, there is significant overlap between $\hat h(\eta-\xi_0)$ and $\hat h(\eta-\xi_1)$.  When the phase term is $e^{2\pi i \Delta t}=1$, the two Gaussian terms in \eqref{eqn:stft_simplified} combine constructively. However, when $e^{2\pi i \Delta t}=-1$, the two Gaussian terms combine destructively. This phenomenon is similar to the concept of wave interference in physics where two coherent waves with misaligned phases are combined such that the resulting amplitude may have greater amplitude (constructive interference) or lower amplitude (destructive interference) if the two waves are in phase or out of phase, respectively. In our setting, the two harmonic components are non-coherent if viewed as two waves, but have a constant phase discrepancy. This means that instead of having either constructive or destructive interference based on alignment, every time in our signal can be described by either constructive or destructive interference. 
\begin{definition}[Constructive and destructive times]\label{times}
For $k\in\mathbb{Z}$, define
$t_k^+\defeq\frac{k}{\Delta}$ as the {\em constructive time} and $t_k^-\defeq\frac{k+\frac12}{\Delta}$ as the {\em destructive time}. 

\end{definition}

At the constructive times $t_k^+$ one has $e^{2\pi i \Delta t_k^+}=1$ and at the destructive times $t_k^-$ one has $e^{2\pi i \Delta t_k^-}=-1$.
The above definition characterizes two extreme times at which constructive and destructive interference occur. Note that at other time $t$ such constructive or destructive behavior still exists but not extremely. Also note that the difference between a pair of constructive and destructive times reflects the frequency of the beating envelop \eqref{equation beating in time domain}. In particular, we can write the square of the modulus of \eqref{eqn:stft_simplified} as 

\begin{align}
    \left|V_f^{(h)}(t,\eta)\right|^2 
    &= \left| \hat{h}(\eta-\xi_0)+ae^{2\pi i \Delta t}\hat{h}(\eta-\xi_1)\right|^2\notag\\
    &= \left(\hat{h}(\eta-\xi_0)\right)^2+a^2\left(\hat{h}(\eta-\xi_1)\right)^2\notag \\
    &\qquad+ a\hat{h}(\eta-\xi_0)\hat{h}(\eta-\xi_1)\left(e^{2\pi i \Delta t}+e^{-2\pi i \Delta t}\right)\notag\\
    &=e^{-2\pi^2\sigma^2(\eta-\xi_0)^2}+ a^2 e^{-2\pi^2\sigma^2(\eta-\xi_1)^2}\label{eqn:int_by_phase}\\
    &\qquad + \underbrace{2 a e^{-\pi^2\sigma^2\left((\eta-\xi_0)^2+(\eta-\xi_1)^2\right)}}_{>0} \cos(2\pi\Delta t).\notag
\end{align}
We can see that the sign of the last term in \eqref{eqn:int_by_phase} is controlled by the term $\cos(2\pi\Delta t)\in[-1,1]$. In particular for $k\in\mathbb{Z}$, when $t\in(\frac{k+1/4}{\Delta},\frac{k+3/4}{\Delta})$, then $|V_f|$ exhibits destructive interference since $\cos(2\pi\Delta t)<0$ and similarly for $t\in (\frac{k-1/4}{\Delta},\frac{k+1/4}{\Delta})$, $|V_f|$ exhibits constructive interference. Notice also that at times $t=\frac{k\pm1/4}{\Delta}$ we have $|V_f|^2=|V_{f_0}|^2+|V_{f_1}|^2$, which seems to be composed of two Gaussian but might have only one local maximum in the $\eta$ axis when $\Delta$ is small. With the above definition of constructive and destructive times, we can characterize the precise behavior at the constructive time as in the following lemma.

\begin{figure}
    \centering
    \includegraphics[width=\linewidth]{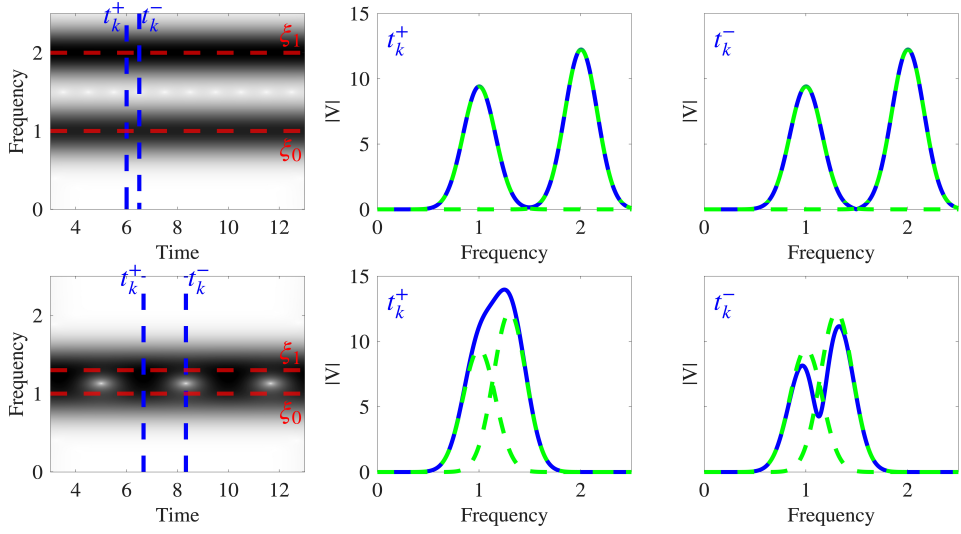}
    \caption{All with $\sigma = \sqrt{2}, a=1.3$, and $\xi_0=1$. (Top row) Enhanced spectrogram $|V_f^{(h)}|$ with larger gap $\xi_1=2$ and cross sections in blue at constructive and destructive times and $|V_{f_0}^{(h)}(t,\eta)|$ and $|V_{f_1}^{(h)}(t,\eta)|$ in green; (Bottom row) Same as above with smaller gap $\xi_1=1.3$. }
\label{fig:diff_Delta_STFT}
\end{figure}

\begin{lemma}[STFT at constructive times]\label{lem:stft_const}
    At constructive times $t_k^+$, the map $\eta\mapsto|V_f^{(h)}(t_{k}^{+},\eta)|$ has one local maximum (i.e. constructive interference between $\hat h(\eta-\xi_0)$ and $\hat h(\eta-\xi_1)$) when $\Delta\leq\Delta_{\textup{critical, STFT}}$ and two local maxima when $\Delta>\Delta_{\textup{critical, STFT}}$, where 
    \[
    \Delta_{\textup{critical, STFT}}=\frac{1+s}{\pi\sigma\sqrt{2s}}
    \]
    and $s$ solves
    \be\label{eqn:stft_s_condition}
    \ln\frac{s}{a}=\frac{1}{2}\left(s-\frac{1}{s}\right).
    \ee
\end{lemma}
The proof of this lemma is provided in Appendix \ref{sec:app_proofs}.
\begin{remark}\label{rem:critical_delta_a1}
    When $a=1$, the two components $f_0$ and $f_1$ have balanced amplitudes and Lemma \ref{lem:stft_const} tells us that the critical gap between two modes for the STFT to begin separating modes is 
    \[
    \Delta_{\text{critical, STFT}}(a=1)=\frac{\sqrt{2}}{\pi\sigma}
    \]
    since $s=1$ is a solution to \eqref{eqn:stft_s_condition}.
    
    \label{rem:critical_delta_a1} 
\end{remark}
\begin{remark}
    Notice that a similar result was recently proven in \cite{1or2ridges}. In the aforementioned paper, the result for the case $a=1$ is recovered identically. Though, in our statement, the relationship between the critical separation and the amplitude of the second component as a function of $s$ is made precise through a transcendental relation. Furthermore, it can be shown using basic calculus that $a=1$ provides the best critical separation using the above form.
\end{remark}

The constructive times are the worst case times for spectral interference since the last term in \eqref{eqn:int_by_phase} is maximized. Intuitively, if we can separate two harmonic components at the constructive times, we shall be able to separate them at other times. We show that this intuition is accurate by the following lemma.
\begin{lemma}\label{lem:stft_const_worst_case}
    When $\Delta>\Delta_{\textup{critical, STFT}}$, for every $t\in\mathbb{R}$ the map $\eta\mapsto |V_f^{(h)}(t,\eta)|$ has exactly two local maxima and one local minimum.
\end{lemma}
The proof of this lemma is provided in Appendix \ref{sec:app_proofs}. Similarly, we can characterize the behavior at the destructive times in the following lemma.
\begin{lemma}[STFT at destructive times]\label{lem:stft_dest}
At destructive times $t_k^{-}$, the function $\eta\mapsto |V_f^{(h)}(t_k^{-},\eta)|$ has one global minimum $|V(t_k^-, \eta_{\textup{AVG}})| = 0$ at the frequency
\[
\eta_{\textup{AVG}}\defeq\bar{\xi}-\frac{\ln a}{2\pi^2\sigma^2\Delta},
\]
and two local maxima, at frequencies $\eta_-$ and $\eta_+$, with 
\begin{align}
\xi_0 -  \Delta \frac{ae^{-C\Delta^2}}{1+ae^{-C\Delta^2}}\ a \le \eta_- < \xi_0 , \quad \quad \xi_1 < \eta_+  \le \xi_1 + \Delta \frac{e^{-C\Delta^2}}{a+e^{-C\Delta^2}}\ .\label{eqn:right_loc_eta}
\end{align}
\end{lemma}

The proof of this lemma is provided in Appendix \ref{sec:app_proofs}. This lemma says that while at destructive times we obtain two local maxima, they are biased estimator of the true frequencies. A natural extension of this local maxima structure is further discussed by the notion of ridges, which we present in the next section.


\subsubsection{Ridges and bubble formation}
In the application of TF analysis, it is common practice to identify and trace the local maxima that form continuous chains on the spectrogram, $|V_f(t,\xi)|^2$, which encode dynamical information of the underlying system. 
In what follows, we adopt the following standard definition of ridges:
\be
    \mathcal{R}_f=\{(t,\eta): \partial_\eta |V_f(t,\eta)|^2=0 \;\;\text{and}\;\; \partial_{\eta\eta} |V_f(t,\eta)|^2<0 \}\,.\label{eqn:ridge_def}
\ee
It has been well established \cite{delprat1992, carmonahwang1999,colominas2021decomposing} that under mild regularity conditions on the oscillatory components, specifically, when their IFs are sufficiently separated, and provided that the window bandwidth is appropriate (neither so long that the IF varies significantly over the window nor so short that the resulting spectral support exceeds the spectral gap), ridge extraction provides a reliable estimator of the IF for each IMT-component. While we do not study noisy signals in this work, we refer readers to \cite{liu2024analyzing} for ridge analysis when the signal is contaminated by noise and the references therein for existing ridge extraction algorithms.

We now examine how the ridges of our signal $f$ behave for small $\Delta$. See Figures \ref{fig:ellipsess2H} and \ref{fig:ellipses2H} for its complicated structure, even under the simple two harmonic model. Interestingly, when the amplitudes of the two components are balanced and we define ridges in such a way, then the ridge set can be shown to approximately contain a repeating ellipse structure as formally described in the following theorem.

\begin{theorem}\label{thm:ellipse}
    Let $a=1$ and $\Delta\leq \frac{\sqrt{2}}{\pi\sigma}$. For $k\in\mathbb{Z}$, define the times
    \[
    t^{L,k}_*=\frac{k}{\Delta}+\frac{1}{2\pi\Delta}\cos^{-1}\left(\pi^2\sigma^2\Delta^2-1\right), \hspace{2mm} t^{R,k}_*=\frac{k+1}{\Delta}-\frac{1}{2\pi\Delta}\cos^{-1}\left(\pi^2\sigma^2\Delta^2-1\right).
    \]
    If $t\in (t^{L,k}_*, t^{R,k}_*)$ for some $k\in\mathbb{Z}$, then the function $\eta \mapsto \gamma(t,\eta)\defeq |V_f^{(h)}(t,\eta)|$ has two local maxima.  If $t\in (t^{R,k-1}_*, t^{L,k}_*)$, then $\eta \mapsto \gamma(t,\eta)$ has one local maximum. Thus, a bifurcation occurs at times $t^{L,k}_*$ and $t^{R,k}_*$.  For $t \in (t^{L,k}_*,t^{R,k}_*)$, $\int_{E_k} \left|\frac{\partial}{\partial_\eta}\gamma(t,\eta)^2\right| dr=O(\Delta^2)$ holds for points $(t,\eta)$ on the ellipse $E_k$ defined by 
    \[
    2\pi^2\sigma^2\left(\eta-\bar{\xi}\right)^2+\frac{4\pi^2\Delta^2(t-\frac{k+1/2}{\Delta})^2}{\arccos^{2}\left(1-\pi^{2}\sigma^2 \Delta^{2}\right)}=1
    \]
    and $dr$ denotes the arc-length measure on $E_k$.
\end{theorem}
\begin{proof}[Proof of Theorem \ref{thm:ellipse}]
    We are interested in the case where constructive interference occurs. By Lemma \ref{lem:stft_const}, this occurs when $\Delta\leq \frac{\sqrt{2}}{\pi\sigma}$. Starting from \eqref{eqn:stft_simplified} we can first introduce the change of coordinates $x=\eta-\bar{\xi}$ and let $C\defeq \pi^2\sigma^2$ to write $$\gamma(t,x)=\left|e^{-C(x+\Delta/2)^2}+ae^{2\pi i\Delta t} e^{-C(x-\Delta/2)^2}\right|.$$ To find the extrema of $\gamma(t,\eta)$, we can equivalently find extrema of $\gamma(t,x)^2$ (since $\gamma$ is non-negative and its maxima at the same locations as maxima of $\gamma^2$). A transition from two peaks to one peak occurs at the bifurcation times when the two distinct extrema merge into a single extremum. At these critical times $t_*$, the first derivative vanishes and the second derivative is 0. Interestingly, for the $a=1$ case, a nice simplification occurs, and since we know where the bifurcation occurs in frequency (at the averaged frequency $\bar{\xi}$), we can get the bifurcation time in a simple way. For convenience, we denote
    \[
    A_1(x)\defeq e^{-2C(x+\frac{\Delta}{2})^2}, \qquad A_2(x)\defeq a^2e^{-2C(x-\frac{\Delta}{2})^2},
    \]
    \[
    A_3(x)=2ae^{-2Cx^2-C\frac{\Delta^2}{2}}\cos(2\pi\Delta t)=2\cos(2\pi\Delta t)\sqrt{A_1(x)A_2(x)}.
    \]
    We can then write 
    \begin{align}
        F(t,x)\defeq \gamma(t,x)^2&= e^{-2C(x+\frac{\Delta}{2})^2}+a^2e^{-2C(x-\frac{\Delta}{2})^2}+2a\cos(2\pi\Delta t) e^{-C\left(2x^2+\frac{\Delta^2}{2}\right)}\label{eqn:F_general}\\
        &=A_1(x)+A_2(x)+A_3(x).\notag
    \end{align}
    We have 
    
    \[
    \frac{\partial A_1}{\partial x}=-4C\left(x+\frac{\Delta}{2}\right)A_1,\qquad\frac{\partial^2 A_1}{\partial x^2}=-4CA_1+16C^2\left(x+\frac{\Delta}{2}\right)^2 A_1,
    \]
    \[
    \frac{\partial A_2}{\partial x}=-4C\left(x-\frac{\Delta}{2}\right)A_2,\qquad\frac{\partial^2 A_2}{\partial x^2}=-4CA_2+16C^2\left(x-\frac{\Delta}{2}\right)^2 A_2,
    \]
    \[
    \frac{\partial A_3}{\partial x}=-4Cx A_3,\qquad\frac{\partial^2 A_3}{\partial x^2}=-4CA_3+16C^2x^2 A_3,
    \]
    so 
    \[
    \frac{\partial F}{\partial x}=-4C\left[\left(x+\frac{\Delta}{2}\right)A_1+\left(x-\frac{\Delta}{2}\right)A_2+xA_3\right],
    \]
    and setting equal to zero (for $a=1$) and factoring out the common exponential factor from $A_1,A_2,$ and $A_3$, we get
    \[
    (x+\frac{\Delta}{2})e^{-2C\Delta x} + (x-\frac{\Delta}{2})e^{2C\Delta x} + 2x \cos(2\pi\Delta t)=0,
    \]
    or equivalently,
    \begin{equation}
        2x(\cos(2\pi\Delta t)+\cosh(2C\Delta x))-\Delta\sinh(2C\Delta x)=0.\label{eqn:ell_hyp}
    \end{equation}
    Define $H(x)$ to be the left-hand side of \eqref{eqn:ell_hyp}. For $x>0$, a direct computation gives
    \[
    H''(x) = 2C\Delta \sinh(2C\Delta x) (4-2C\Delta^2) + 8xC^2\Delta^2\cosh(2C\Delta x)>0,
    \]
    since $C\Delta^2\leq 2$ by assumption. Thus $H$ is strictly convex on $(0,\infty)$, so $H$ has at most one positive zero. Because $F$ (hence $\gamma$) is even in $x$, it follows that $\gamma(t,\eta)$ has either one local maximum (at $x=0$) or two symmetric local maxima. We can also find
    \begin{equation}\label{eqn:ellipse_2deriv}
        \frac{\partial^2 F}{\partial x^2}=-4C(A_1+A_2+A_3)+16C^2\left[\left(x+\frac{\Delta}{2}\right)^2A_1+\left(x-\frac{\Delta}{2}\right)^2A_2+x^2A_3\right].
    \end{equation}
    Evaluating \eqref{eqn:ellipse_2deriv} at $x=0$ (with $a=1$) yields
    \[
    \frac{\partial^2 F}{\partial x^2}(t,0)=8C e^{-C\Delta^2/2}\left(C\Delta^2-(1+\cos(2\pi\Delta t))\right).
    \]
    Hence $x=0$ is a local maximum if and only if $1+\cos(2\pi \Delta t)>C\Delta^2$ and a local minimum if and only if $1+\cos(2\pi\Delta t)<C\Delta^2$. The transition occurs when $\cos(2\pi \Delta t)=C\Delta^2-1$, which gives the bifurcation times 
    \[
    t^{L,k}_*=\frac{k}{\Delta}+\frac{1}{2\pi\Delta}\cos^{-1}\left(C\Delta^2-1\right)
    \]
    and
    \[
    t^{R,k}_*=\frac{k+1}{\Delta}-\frac{1}{2\pi\Delta}\cos^{-1}\left(C\Delta^2-1\right)
    \]
    for $k\in\mathbb{Z}$ which are the left and right bifurcation times (centered around each destructive time $t_k^-=\frac{k+1/2}{\Delta}$) respectively. By the convexity argument above, this implies that for $t\in(t^{R,k-1}_*,t^{L,k}_*)$ there is exactly one local maximum (at $x=0$), while for $t\in(t^{L,k}_*,t^{R,k}_*)$ there are exactly two local maxima. We expect (based on numerical observations) that when $\Delta$ is small, there exists some ellipse $E$ such that the parametrization of this ellipse solves \eqref{eqn:ell_hyp}. If we apply the shift $t\mapsto s+t_k^-$ and apply approximations for small $\Delta$, then  \eqref{eqn:ell_hyp} becomes 
    \begin{align*}
        0&=2x(\cos(2\pi\Delta t)+\cosh(2C\Delta x))-\Delta\sinh(2C\Delta x)\\
        &=2x\left(\cos\left(2\pi\Delta \left(s+\frac{k+1/2}{\Delta}\right)\right)+\cosh(2C\Delta x)\right)-\Delta\sinh(2C\Delta x)\\
        &= 2x(\cosh(2C\Delta x)-\cos(2\pi\Delta s))-\Delta\sinh(2C\Delta x)\\
        &= 2x\Delta^2(2C^2 x^2 + 2\pi^2 s^2 - C) +O(\Delta^4)
    \end{align*}
    so that up to $O(\Delta^2)$ (and for $x\neq 0$), \eqref{eqn:ell_hyp} becomes
    \begin{equation}
        2Cx^2+\frac{2\pi^2}{C}s^2=1.
    \end{equation}
    The above equation tells us that up to $O(\Delta^2)$, we can approximate the frequency dimension radius of the ellipse as $1/\sqrt{2C}=1/(\sqrt{2}\pi\sigma)$. Using these two radii, we can infer the equation of the ellipse for small $\Delta$ as 
    \begin{equation}
        2Cx^2+\frac{4\pi^2\Delta^2(t-t_k^-)^2}{\arccos^{2}\left(1-C\Delta^{2}\right)}=1\label{eqn:ellipse_eq}.
    \end{equation} 
    To prove that \eqref{eqn:ellipse_eq} (call this ellipse $E_k$) in fact approximates the solution subset of \eqref{eqn:ell_hyp}, we can show that
    \[
    \int_{E_k} \left| \frac{\partial F}{\partial x}\right| dr\to 0
    \]
    as $\Delta \to 0$ for some valid parametrization of $E_k$. Translating back to $s$, we can let $$E_k=\{(x,s): \frac{x^2}{a^2}+\frac{s^2}{b^2}=1, \text{ where } a=\frac{1}{\sqrt{2C}}, b=\frac{1}{2\pi\Delta}\arccos(1-C\Delta^{2})\}.$$ Defining the parametrization of this ellipse as $x(u)=a\cos(u)$ and $t(u)=b\sin(u)$ for $u\in [0,2\pi)$, then we can write 
    \begin{align*}
        \int_{E_k} \left| \frac{\partial F}{\partial x}\right| dr&= \int_0^{2\pi} \left| \frac{\partial F}{\partial x}(x(u),t(u))\right| J(u) du\,,
    \end{align*}
    where
    \begin{align*}
J(u)=\sqrt{x'(u)^2+t'(u)^2}
        = \sqrt{a^2 \sin^2(u)+b^2\cos^2(u)}\,.
    \end{align*}
    We can find 
    \begin{align*}
        \left| \frac{\partial F}{\partial x}(x(u),t(u))\right|&= \Delta^2\left|4C^2a^3 \cos^3 (u) + 4\pi^2 a b^2 \cos(u)\sin^2(u)-2Ca \cos(u)\right|\\
        &\leq 4C^2\Delta^2a^3 \left|\cos^3 (u)\right| + 4\pi^2 \Delta^2a b^2 \left|\cos(u)\sin^2(u)\right|\\
        &\quad+2C\Delta^2a \left|\cos(u)\right|\\
        &\leq \Delta^2\left(4C^2 a^3 +4\pi^2 ab^2 + 2Ca\right)+O(\Delta^4).
    \end{align*}
    We must also compute
    \begin{align*}
J(u)= \sqrt{a^2 \sin^2(u)+b^2\cos^2(u)}\leq a|\sin(u)|+b|\cos(u)|.
    \end{align*}
    Putting everything together, we have 
    \begin{align}
        \int_{E_k} \left| \frac{\partial F}{\partial x}\right| dr&= \int_0^{2\pi} \left| \frac{\partial F}{\partial x}(x(u),t(u))\right| J(u) du\notag\\
        &\leq \int_0^{2\pi} \Delta^2 \left[\left(4C^2 a^3 +4\pi^2 ab^2 + 2Ca\right)+O(\Delta^4)\right] J(u) du\notag\\
        &=\Delta^2 \left[\left(4C^2 a^3 +4\pi^2 ab^2 + 2Ca\right)+O(\Delta^4)\right]\int_0^{2\pi}  J(u) du\notag\\
        &\leq \Delta^2 \left[\left(4C^2 a^3 +4\pi^2 ab^2 + 2Ca\right)+O(\Delta^4)\right]\times \notag\\
        &\quad\int_0^{2\pi}  a|\sin(u)|+b|\cos(u)| du\notag\\
        &= \Delta^2 (4a+4b)\left(4C^2 a^3 +4\pi^2 ab^2 + 2Ca\right)+O(\Delta^4).\label{eqn:ell_err}
    \end{align}
    Notice that for small $\Delta$ we have $b=\frac{1}{2\pi\Delta}\arccos(1-C\Delta^{2})=\frac{\sigma}{\sqrt{2}}+O(\Delta^2)$. Using this and recalling that $a=\frac{1}{\sqrt{2C}}$, \eqref{eqn:ell_err} tells us that
    \begin{align*}
        \int_{E_k} \left| \frac{\partial F}{\partial x}\right| dr&\leq \Delta^2 (4a+4b)\left(4C^2 a^3 +4\pi^2 ab^2 + 2Ca\right)+O(\Delta^4)\\
        &= 12\Delta^2(1+\pi\sigma^2)+O(\Delta^4)
    \end{align*}
    which vanishes as $\Delta\to 0$.
\end{proof}

\begin{remark}
Notice that this ellipse only exists when the argument of the $\arccos$ is between $-1$ and $1$. This only occurs when $\Delta<1/\sqrt{C}=\sqrt{2}/(\pi\sigma)$, which agrees with the critical threshold we found in Remark \ref{rem:critical_delta_a1}. When this condition does not hold, it can be shown that the best approximation is hyperbolic. In this case, the spectral gap is large enough for the STFT to separate two distinct components for all time, though the visible artifacts of spectral interference can be described by a hyperbolic shape in the TF domain rather than the elliptic time frequency bubbles that appear when the spectral gap is small.
\end{remark}

See Figure \ref{fig:ellipsess2H} for an illustration of this Theorem. We refer to the ellipse-shaped ridges indicated by green dashed curve as ``bubbles'', and the theorem says that these bubbles emerge when $\Delta$ gets sufficiently small. The formation of bubbles arises from the nonlinear interaction between amplitude and phase of $V_f$, which, in turn, reflects the rich structure of spectral interference, even in this simple two harmonic model.

From Theorem \ref{thm:ellipse}, we find that for $a=1$, the width and height of bubbles is constant (with respect to $\Delta$) to first order. The height and width are given by $\frac{1}{\sqrt{2}\pi\sigma}+O(\Delta^2)$ and $\frac{\sigma}{2}+O(\Delta^2)$, respectively, where we simply approximated the radii $r_a$ and $r_b$ from the proof of Theorem \ref{thm:ellipse} for small $\Delta$. Surprisingly, when $\Delta$ decreases, the width and height of bubbles remains constant, while the times at which they occur spread out with width inverse to $\Delta$. Note that when a bubble appears, it is at the times when the STFT is capable of distinguishing two components. The theorem confirms the intuition that the closer the two frequencies are, the more difficult it is for the STFT to distinguish them. As $a$ varies away from $1$, the ridge is continuously perturbed away from being ellipse-shaped. This can be seen in Figure \ref{fig:ellipses2H}.

\begin{figure}
    \centering
    \includegraphics[width=\linewidth]{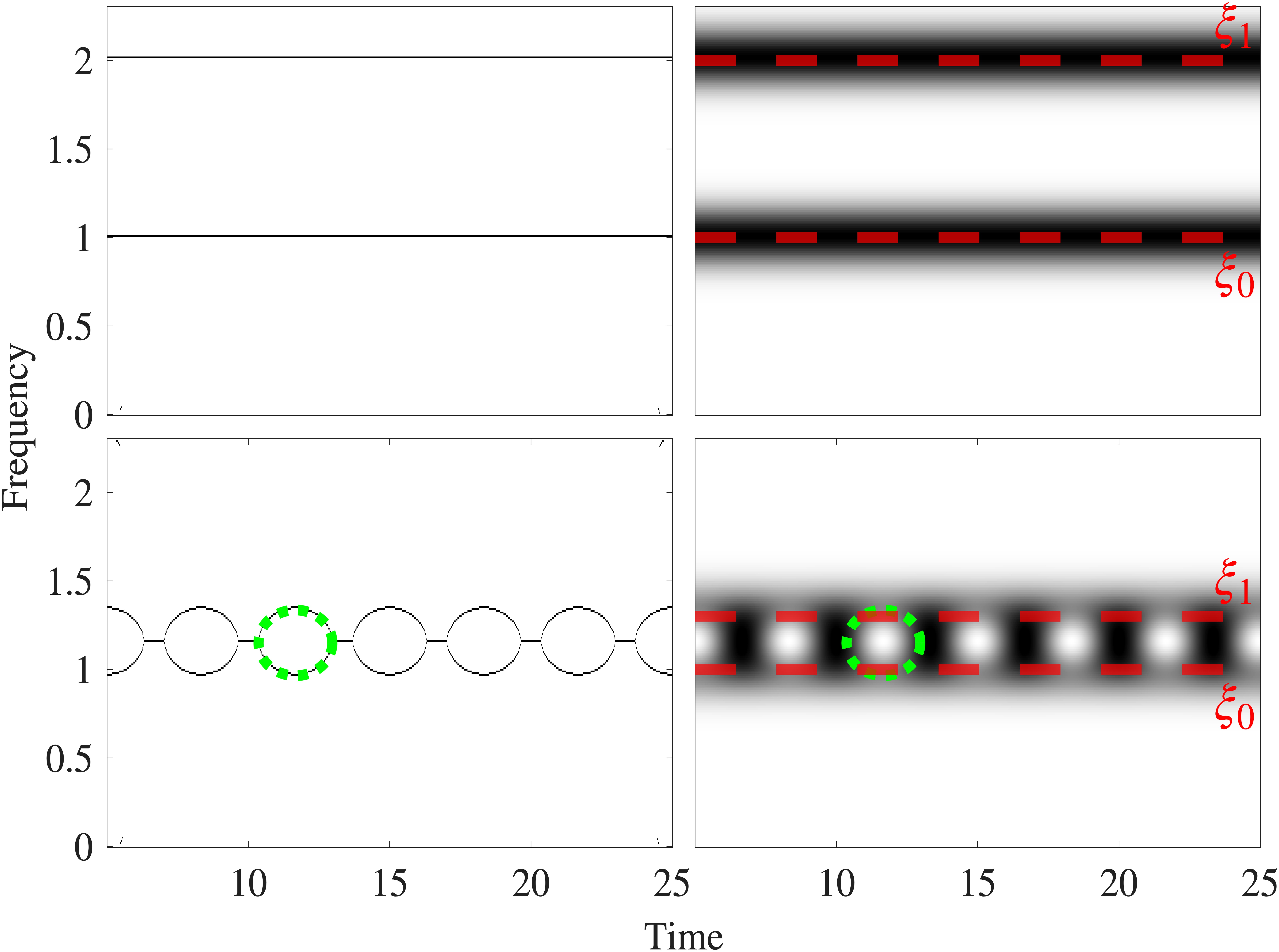}
    \caption{All with $\sigma=\sqrt{2}, a=1$ and $\xi_0=1$. Ridges of the two component signal as defined in \eqref{eqn:ridge_def}. (Top left) With $\xi_1=2$; (Bottom left) With $\xi_1=1.3$, green dashed curve is the ellipse from Theorem \ref{thm:ellipse}; (Right) Corresponding spectrograms.}
    \label{fig:ellipsess2H}
\end{figure}

\begin{figure}[hbt!]
    \centering
    \includegraphics[width=1\linewidth]{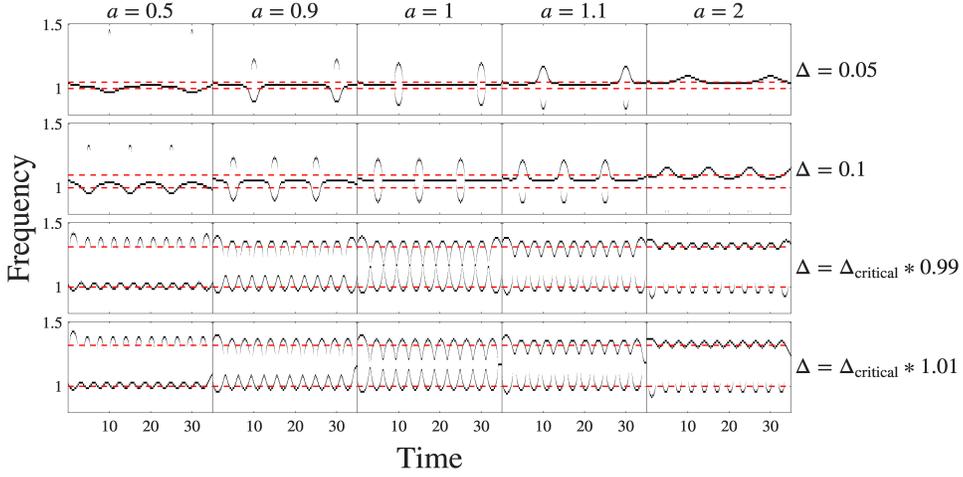}
    \caption{All with $\sigma=\sqrt{2}$ and $\xi_0=1$. Ridges of the two component signal as defined in \eqref{eqn:ridge_def} with varying $a,\Delta$; The bottom three rows show when $\Delta\approx\Delta_{\text{critical, STFT}}(a=1)$ is near the critical point.}
    \label{fig:ellipses2H}
\end{figure}

\newpage

\subsection{Generalizing to the 2-component AHM}\label{Section STFT approximation of AHM}
The previous results can be generalized to the 2-component AHM model in the following way.
Suppose
    \[
    F(x)=A_0(x) e^{2\pi i \phi_0(x)} + A_1(x) e^{2\pi i \phi_1(x)}
    \]
    satisfies the AHM model in Definition \ref{Definition: AHM}.
    Also fix a time $t_*$ and set $\xi_j = \phi_j'(t_*)$, $a_j\defeq A_j(t_*)$, and $a=a_1/a_0$, and define the two-component harmonic model
    \[
    f(x)=e^{2\pi i \xi_0 x} + ae^{2\pi i \xi_1 x}.
    \]
    Assume without loss of generality (phase removal by a time shift), that the constant phase offset satisfies $\phi_j(t_*)=0$. Then, it has been well known \cite{sst} that for all $t,\eta$,
    \begin{align*}
        |V_F(t,\eta)-a_0 V_f(t-t_*,\eta)| &\leq \varepsilon \sum_{j=0}^1\bigg[|\phi_j'(t_*)|\left(|t-t_*|+\frac{\sigma}{\sqrt{\pi}}\right)\\
        &\qquad\qquad+\frac{M_j''}{2}\left((t-t_*)^2+\frac{\sigma^2}{2}\right)\bigg]\\
        & \quad+2\pi\varepsilon\sum_{j=0}^1a_j\bigg[\frac{|\phi_j'(t_*)|}{2}\left((t-t_*)^2+\frac{\sigma^2}{2}\right)\\
        & \qquad\qquad\qquad+\frac{M_j''}{6}\bigg(|t-t_*|^3+\frac{3\sigma}{\sqrt{\pi}}|t-t_*|^2\\
        &\qquad\qquad\qquad+\frac{3}{2}\sigma^2 |t-t_*|+\frac{\sigma^2}{\sqrt{\pi}}\bigg)\bigg].
    \end{align*}
The proof of the above statement is provided in Appendix \ref{sec:app_proofs}. This approximation says that locally around $t_*$, STFT of $F(t)$ is close to that of $f(t)$ up to an error of order $\varepsilon$. This deviation says that the ridges of $F(t)$ is deviated from those of $f(t)$ by an order of $\varepsilon$. See Figure \ref{fig:ellipsess} for one case of the AHM, where the instantaneous frequency of the second component changes slowly over time, while still demonstrating the ellipse structure.

\begin{figure}
    \centering
    \includegraphics[width=\linewidth]{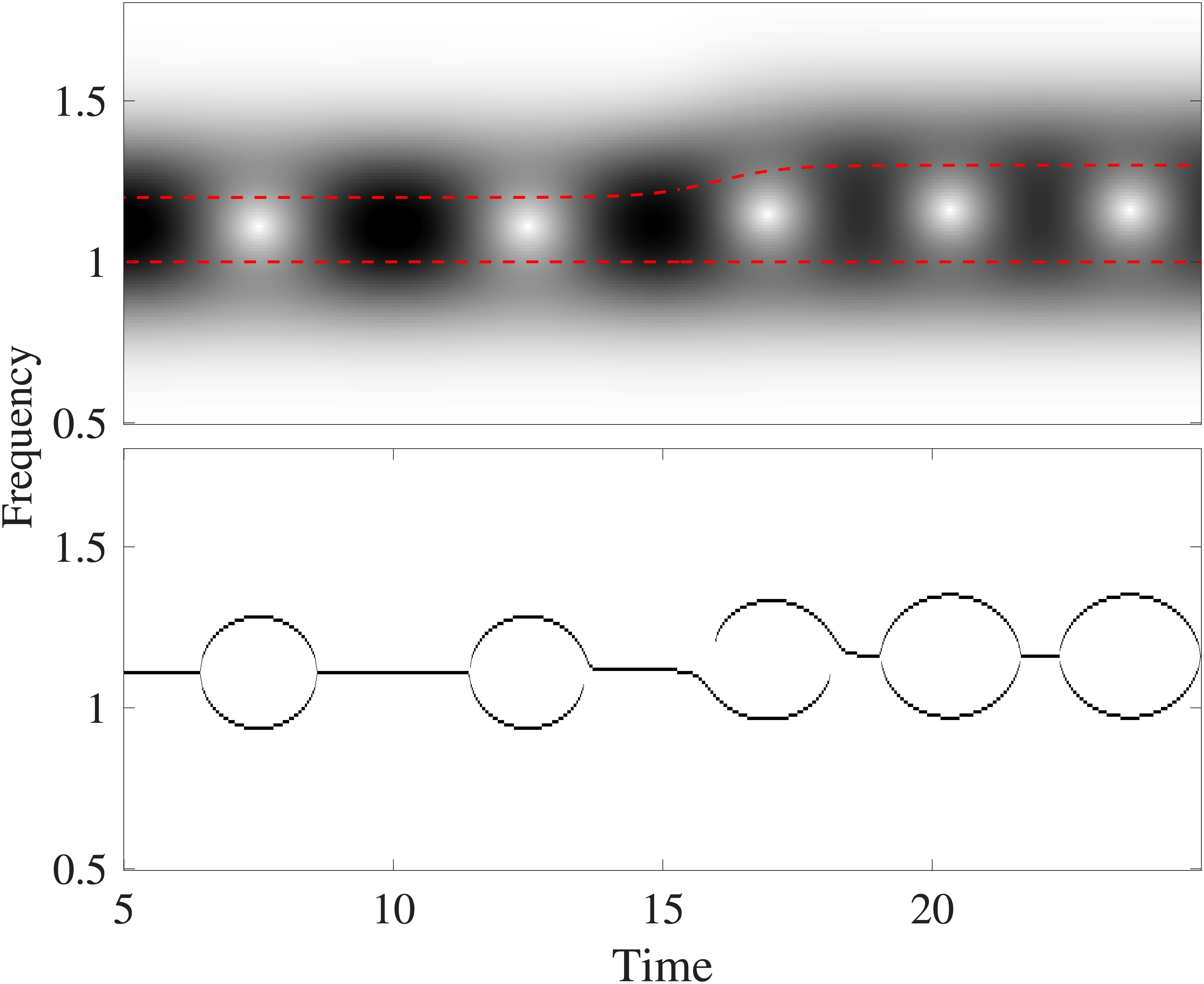}
    \caption{With $\sigma=\sqrt{2}, a=1$, $\xi_0=1$, and $\xi_1$ starting at 1.2 and slowly changing to 1.3. (Top) STFT and true IFs overlayed in red dashed line; (Bottom) Ridges of the two component signal as defined in \eqref{eqn:ridge_def}. Notice the bifurcation point becomes discontinuous as $\xi_1$ changes.}
    \label{fig:ellipsess}
\end{figure}

\section{Phase and Reassignment Methods}\label{sec:phase}
While STFT in general is a complex-valued function defined on the TF domain, in practice it is typically the magnitude information of STFT, $|V_f^{(h)}(t,\eta)|$, that is analyzed. For example, the ridge discussed in the previous section is solely based on the magnitude. On the other hand, phase information is often ignored. From a data analysis perspective, however, the phase also encodes meaningful information that we shall explore. Recall that the complication of spectral interference we saw in the above ridge analysis is a consequence of phase dynamics. Write 
\be\label{eqn:phase_form_stft}
V_f^{(h)}(t,\eta)=|V_f^{(h)}(t,\eta)|e^{i\phi(t,\eta)}\,,
\ee
where $\phi(t,\eta)\in [0,2\pi)$ is the phase.
The rich structure of the phase of the STFT can be seen for example in Figure \ref{fig:phase}, where the amplitude is shown in parallel for a comparison.  

\subsection{Phase structure near zeros}
Note that the phase $\phi$ is ill-defined when $V_f^{(h)} = 0$. We shall answer questions like how many zeros can we have, and how does the phase behave near zeros.
Take the $a=0$ case as a special example. In this case, $V_f^{(h)}(t,\eta) =e^{2\pi i \xi_0 t}e^{-\pi^2\sigma^2 (\eta-\xi_0)^2}$, which is nonzero everywhere. However, when $a>0$, we start to see zeros caused by the interference between two components, particularly when $\Delta$ is small. 
To explore the phase structure near zeros, note that as a complex-valued function on the TF domain $V_f$ is bounded, so by Liouville's theorem, $V_f$ is not holomorphic.
First, recall how the STFT is related to a holomorphic function on the TF domain.
\begin{definition}[Bargmann Transform, \cite{grochenig}]
    The \textit{Bargmann transform} of a function $f$ on $\mathbb{R}^d$ is the function $Bf$ on $\mathbb{C}^d$ defined by
    \[
    Bf(z)=2^{d/4} \int_{\mathbb{R}^d} f(x) \exp\left(-\pi x^2+2\pi x z-\frac{\pi}{2}z^2\right)dx.
    \]
\end{definition}
Since we choose a Gaussian window, we can use Proposition 3.4.1 from \cite{grochenig} with $z=t/\sigma-i\pi\sigma\eta$ to write
\begin{equation}\label{stft vs bt}
V_f^{(h)}(t,\eta)=e^{-\frac{1}{2}\left[(t/\sigma)^2+(\pi\sigma\eta)^2\right]}e^{-i\pi t\eta}B_\sigma f\left(\frac{t}{\sigma}-i\pi\sigma\eta\right)
\end{equation}
where
\[
B_\sigma f(z) \defeq \frac{1}{\sigma\sqrt{\pi}}\int_{\mathbb{R}} f(x) \exp\left(-\frac{x^2}{\sigma^2}+\frac{2x}{\sigma}z-\frac{1}{2}z^2\right)dx
\]
is a rescaled Bargmann transform. Because $B$, and thus $B_\sigma$, is entire in $z$, the map $z\mapsto e^{\frac{1}{2}[(t/\sigma)^2+(\pi\sigma\eta)^2]}e^{i\pi t\eta} V_f^{(h)}(t,\eta)$ with $t=\sigma \Re(z),\eta=-\Im(z)/(\pi\sigma)$ is entire. Consequently, since $e^{\frac{1}{2}[(t/\sigma)^2+(\pi\sigma\eta)^2]}e^{i\pi t\eta}$ is nowhere-vanishing, $V_f^{(h)}$ can be written as a nowhere-vanishing smooth function times an entire function of $z$; the zeros of $V_f^{(h)}$ on the TF domain correspond exactly to the zeros of $B_\sigma f$ and hence are isolated with finite multiplicities. We shall mention that the peculiar behavior of zeros of STFT has been considered in the literature \cite{flandrin_spect_zeros} to study noise.

Next, we observe that the phase function is discontinuous in Figure \ref{fig:phase} at the zeros of the STFT (at times $t=t_k^-$ and frequency $\eta=\eta_{\text{AVG}}$).  In the following proposition we show that in a small neighborhood around a zero, the phase winds exactly once.

\begin{figure}
    \centering
    \includegraphics[width=\linewidth]{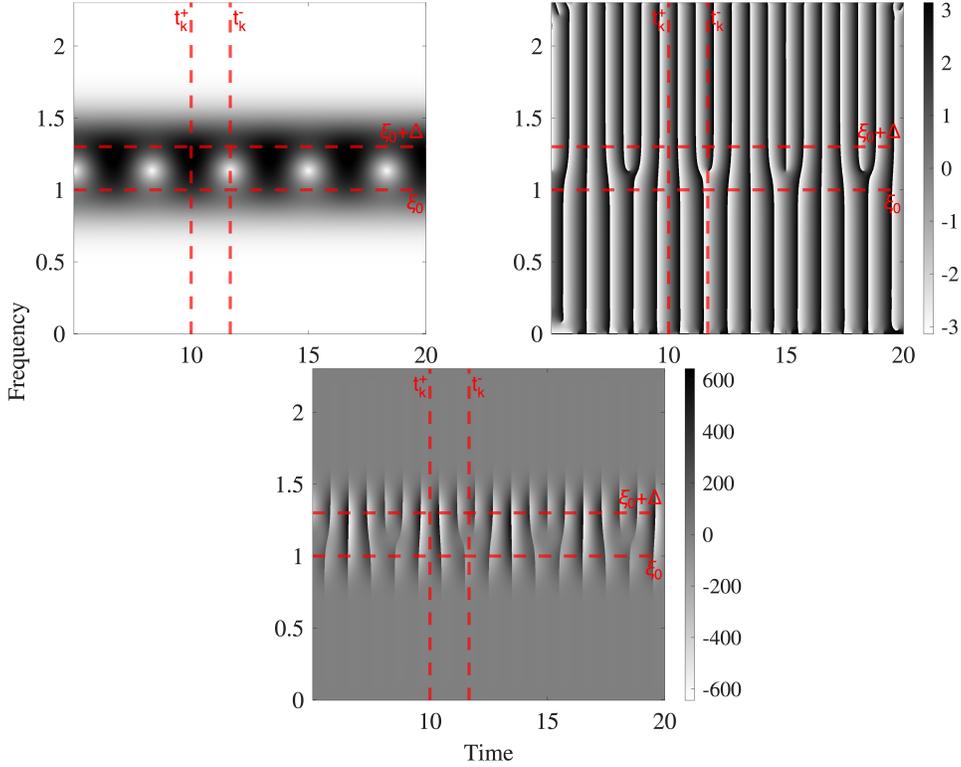}
    \caption{The same signal as described in the bottom row of Figure \ref{fig:diff_Delta_STFT} (with $\sigma=\sqrt{2}, a=1.3$, $\xi_0=1$, and $\xi_1=1.3$); (Top left) The amplitude of the STFT $|V|$; (Top right) The phase $\phi(t,\eta)$ (in radians), showing a phase transition at both constructive and destructive times; (Bottom) The product $|V|\times\phi(t,\eta)$ where the phase no longer winds around zero points of $V$.}
    \label{fig:phase}
\end{figure}

\begin{proposition}\label{prop:winding}
    Suppose $(t_0,\eta_0)$ is a zero of $V_f^{(h)}(t,\eta)$. There exists $r>0$ such that for every $0<\rho<r$, the ellipse
    \be\label{eqn:zeros_ellipse}
    \mathcal{E}_\rho:= \{(t,\eta)|
    \,(t-t_0)^2 + \pi^2\sigma^4 (\eta-\eta_0)^2=(\sigma\rho)^2\}
    \ee
    contains no zeros of $V_f^{(h)}$ on its boundary and contains exactly one interior zero. Furthermore, the curve $V_f^{(h)}|_{\mathcal{E}_\rho}$ winds about $0$ exactly once. 
\end{proposition}

The proof of the above proposition is provided in Appendix \ref{sec:app_proofs}. The above proposition describes how the phase changes near a zero point, $(t_0,\eta_0)$.  
In practice, the phase can be weighted by the magnitude by pointwisely multiplying the phase and magnitude, as shown in the lower plot in Figure \ref{fig:phase}, to enhance  visualization. In particular, by Taylor expansion, one can see that the amplitude decays linearly near zero points. Therefore, even though the phase exhibits discontinuous behavior when approaching zero points, multiplying by amplitude masks this behavior and enhances visualization.

\subsection{Phase-based reassignment}
Motivated by \cite{kodera_reassignment, flandrin_reassignment}, it has become well known that the phase of the STFT can be applied to {\em reassign} STFT coefficients on the TF domain and produce a more concentrated TFR. Among many variations \cite{sst,oberlin2015second,pham2017high}, we start with considering the \textit{phase reassignment rule} defined as
\begin{align}
    \hat{\eta}_p(t,\eta)&\defeq \frac{1}{2\pi}\frac{\partial \phi}{\partial t}\,,\label{eqn:phase_freq_reassign}
\end{align}
where $\phi(t,\eta):=\arg V_f^{(h)}(t,\eta)$ and the principal branch is taken for the argument, and the \textit{synchrosqueezing reassignment rule} considered in \cite{sst} defined as
\be\label{eqn:synchrosqueezed_reassignment}
    \hat{\eta}_s(t,\eta) \defeq \begin{cases}
    \dfrac{1}{2\pi i}\dfrac{\partial_t V_f^{(h)}(t,\eta)}{V_f^{(h)}(t,\eta)} & \text{if } V_f^{(h)}(t,\eta)\neq 0\\
    -\infty & \text{else}.
    \end{cases}
\ee 

Note that these two reassignment rules are closely related. To see this, we can write
\begin{align*}
    \hat{\eta}_s(t,\eta)&=\frac{1}{2\pi i}\frac{\partial_t V_f^{(h)}(t,\eta)}{V_f^{(h)}(t,\eta)}&&\text{when }\left|V_f^{(h)}(t,\eta)\right|\neq 0\notag\\
    &=\frac{1}{2\pi i}\partial_t \log(V_f^{(h)}(t,\eta))\notag\\
    &=\frac{1}{2\pi i}\left(\partial_t \log\left|V_f^{(h)}(t,\eta)\right|+i\partial_t \phi(t,\eta)\right) && \text{using }\eqref{eqn:phase_form_stft}\notag\\
    &= \hat{\eta}_p(t,\eta) +\frac{1}{2\pi i}\partial_t \log\left|V_f^{(h)}(t,\eta)\right|\\
    &= \hat{\eta}_p(t,\eta) +\frac{i}{2\pi}\Re\left(\frac{V_f^{(Dh)}(t,\eta)}{V_f^{(h)}(t,\eta)}\right)\\
    &= \hat{\eta}_p(t,\eta) + i \Im\left(\hat{\eta}_s(t,\eta)\right)\\
    &\approx \hat{\eta}_p(t,\eta)&& \text{if } |V| \text{ varies slowly}\notag
\end{align*}
when $V_f^{(h)}(t,\eta)\neq 0$, where $Dh$ denotes the first derivative of the window $h$ and we use the relationship  $\partial_t V_f^{(h)}(t,\eta)= - V_f^{(Dh)}(t,\eta)+2\pi i \eta V_f^{(h)}(t,\eta)$. From this, we see that the phase reassignment operator is the real part of the synchrosqueezing reassignment operator: 
\[
\hat{\eta}_p(t,\eta)=\Re\left(\hat{\eta}_s(t,\eta)\right),
\]
and the synchrosqueezing reassignment operator differs from the phase reassignment operator by the following purely imaginary term derived from the amplitude of $V$
\begin{align}
    \frac{1}{2\pi i}&\partial_t \log\left|V_f^{(h)}(t,\eta)\right|\label{eqn:extra_term}\\
    &=i\frac{ a\Delta e^{-\pi^2\sigma^2\left((\eta-\xi_0)^2+(\eta-\xi_1)^2\right)} \sin\left(2\pi\Delta t\right)}{e^{-2\pi^2\sigma^2(\eta-\xi_0)^2} + a^2 e^{-2\pi^2\sigma^2(\eta-\xi_1)^2} + 2ae^{-\pi^2\sigma^2\left((\eta-\xi_0)^2+(\eta-\xi_1)^2\right)} \cos\left(2\pi\Delta t\right)}\,.\notag
\end{align}
Notice that this term is zero at constructive times $t_k^+$. Thus, based on Lemma \ref{lem:stft_const_worst_case}, both types of reassignment will encounter spectral interference at the same critical frequency gap, but we will later see that with a ridge as defined in \eqref{eqn:ridge_def}, the synchrosqueezing reassignment rule is {\em empirically} observed to admit less constructive interference than the phase reassignment rule.
Figure \ref{fig:Mobius_mapping} shows both the real and imaginary parts of the phase and synchrosqueezed reassignment operators. Observe that the constant frequency lines in the left panel are warped in the complex plane shown in the right panels.
\begin{figure}
    \centering
    \includegraphics[width=\linewidth]{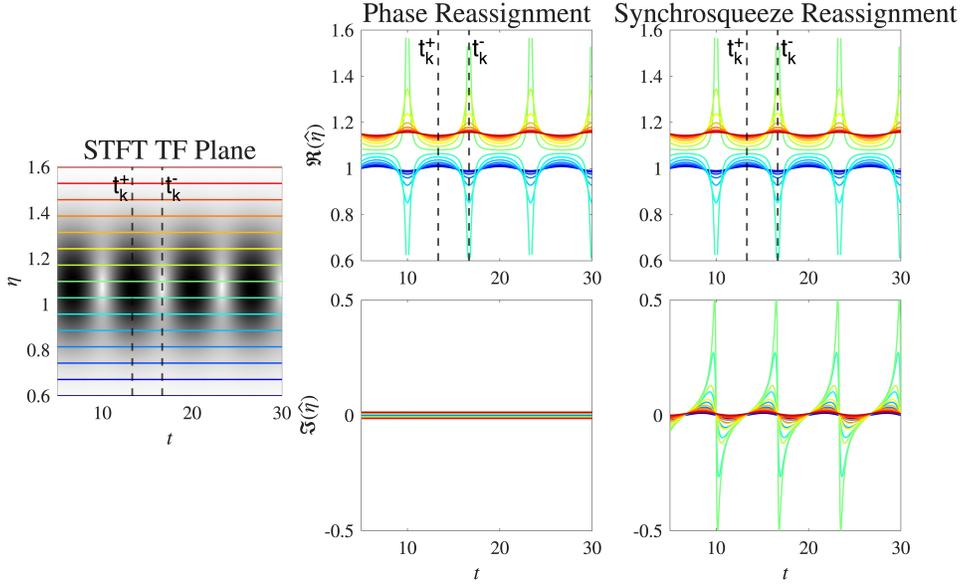}
    \caption{All with $\sigma=\sqrt{2}, a=1$, $\xi_0=1$, and $\xi_1=1.15$. (Left) The STFT of $f$ with colored horizontal lines representing constant lines in the time frequency plane $(t,\eta)$; (Middle) the image of the constant colored lines in the (top) real and (bottom) imaginary parts of the phase reassigned time frequency plane showing how frequencies are pushed towards the true ridges at $1$ and $1.15$; (Right) the image for the synchrosqueezing reassignment. Notice that the real parts of both types of reassignment are the exact same, but the imaginary parts differ.}
    \label{fig:Mobius_mapping}
\end{figure}

Furthermore, in the two component model we study, the synchrosqueezing reassignment operator takes the form of a well known rational function on the complex plane as defined below.

\begin{definition}[M\"obius transformation] A M\"obius transformation $f$ is a mapping of the form $f:\hat{\mathbb{C}}\to\hat{\mathbb{C}}$, where $\hat{\mathbb{C}}$ is the one point compactification of $\mathbb{C}$ and
$$f(z)=\frac{az+b}{cz+d}$$ for $z\in\mathbb{C}$ and $a,b,c,d\in\mathbb{C}$ such that $ad-bc\neq 0$. We set $f\left(-\frac{d}{c}\right)=\infty$ if $c\neq 0$ and $f(\infty)=\frac{a}{c}$ if $c\neq 0$, else $f(\infty)=\infty$.
\end{definition}
We can write the synchrosqueezing reassignment rule of our two harmonic signal explicitly as 
\begin{align}
    \hat{\eta}_s(t,\eta)
    =\frac{\xi_0\hat{h}(\eta-\xi_0) + a \xi_1 e^{2\pi i \Delta t} \hat{h}(\eta-\xi_1)}{\hat{h}(\eta-\xi_0) + ae^{2\pi i \Delta t} \hat{h}(\eta-\xi_1)}
    =\frac{\xi_0 + \xi_1 ae^{2\pi i \Delta t} \frac{\hat{h}(\eta -\xi_1)}{\hat{h}(\eta - \xi_0)}}{1 + ae^{2\pi i \Delta t} \frac{\hat{h}(\eta -\xi_1)}{\hat{h}(\eta - \xi_0)}}\,.\notag
\end{align}
The rational function structure of $\hat{\eta}_s(t,\eta)$ suggests that it can be written as the composition
\[
\hat{\eta}_s(t,\eta)=M(q(t,\eta))\,,
\]
where $M$ is the M\"obius transformation
\[
M(z) = \frac{\xi_0 + \xi_1z}{1+z}
\]
and
\begin{align*}
    q(t,\eta)&= ae^{2\pi i \Delta t} \frac{\hat{h}(\eta -\xi_1)}{\hat{h}(\eta - \xi_0)}
    =ae^{2\pi i \Delta t}e^{-\pi^2 \sigma^2 ((\eta -\xi_1)^2 - (\eta - \xi_0)^2)}
    =ae^{2\pi i \Delta t}e^{2\pi^2 \sigma^2 \Delta(\eta-\bar{\xi})}
\end{align*}
is a complex-valued function on the TF domain. Thinking of $M:\hat{\mathbb{C}}\to \hat{\mathbb{C}}$ as a map on the extended complex plane, notice that $M(0)=\xi_0$ and $M(\infty)=\xi_1$. Thus points $q$ with small modulus $|q|\ll 1$ are mapped close to $\xi_0$ while points with large modulus $|q|\gg 1$ are mapped close to $\xi_1$. The next lemma makes this "attraction" to $\xi_0$ quantitative when $|q|$ is small.
\begin{lemma}\label{lem:M_map_bound}
If $|a| e^{ \pi^2 \sigma^2 \Delta (\eta - \bar \xi)}\leq 1/2$, then 
\[
|\hat{\eta}_s(t,\eta) - \xi_0| = |M(q(t,\eta))-\xi_0| \leq 2 \Delta |a| e^{ \pi^2 \sigma^2 \Delta (\eta - \bar \xi)}, \quad \forall \;t \in \mathbb{R}
\]
\end{lemma}
\begin{proof}[Proof of Lemma \ref{lem:M_map_bound}]
    If $z \neq -1$, we write
    \begin{align*}
    M(z) = \frac{\xi_0 + \xi_1 z}{1 + z} =\xi_0+\frac{z}{1+z} = \xi_1 - \frac{\Delta}{1 + z}.
    \end{align*}
    Let $z = re^{i \theta}$ with $|r| \neq 1$, and define $\bar z = \frac{1}{1-r^2} \in \mathbb{R}$. Then
    \begin{align*}
    \frac{1}{1 + z} - \bar z =\,& \frac{1}{1 + re^{i \theta}} - \frac{1}{1-r^2}  \\
    = \,&-\left( \frac{r}{1 - r^2}\right) \left( \frac{r + e^{i \theta}}{1 + re^{i \theta} }\right)    -\left( \frac{r}{1 - r^2}\right) \left( \frac{re^{-i \theta/2} + e^{i \theta/2}}{e^{-i \theta/2} + re^{i \theta/2} }\right)  \\
    =\,& -\left( \frac{r}{1 - r^2}\right) \frac{v}{\bar v}, \quad \quad v = re^{-i \theta/2} + e^{i \theta/2}.
    \end{align*}
    Therefore,
    \[
    \left| \frac{1}{1 + z} - \bar z\right| =  \frac{|r|}{|1 - r^2|},
    \]
    which is independent of $\theta$. Returning to the definition of $M$, we find that
    \begin{align*}
    \left| M(z) - (\xi_1 - \Delta \bar z) \right | & = \Delta \left| \frac{1}{ 1 +z} - \bar z \right|  = \Delta \frac{|r|}{|1 - r^2|}
    \end{align*}
    Now, letting $y = \Delta \sigma^2 (\eta - \bar \xi)$, and $s = \Delta t$, we have 
    \[
    q\left(\frac{s}{\Delta}, \bar{\xi}+\frac{y}{2\Delta \sigma^2}\right)=ae^{2\pi i s}e^{\pi^2 y}\,.
    \]
    Therefore,
    \begin{align*}
    \hat{\eta}_s(t,\eta) & = \hat{\eta}_s(s/\Delta, \bar \xi + y/(\sigma^2 \Delta)) =   M(q(s/\Delta,\bar \xi + y/(\sigma^2 \Delta))) = M ( a e^{2 \pi i s }e^{\pi^2 y})
    \end{align*}
    Observe from the above conclusion that
    \begin{align}
    |M(z) - \xi_0| = \left| (\xi_1 - \xi_0) \frac{z}{1 + z} \right| = \frac{\Delta |z|}{|1 + z|} \,.\label{eqn:Mxi0simple}
    \end{align}
    If $|a| e^{ \pi^2 \sigma^2 \Delta (\eta - \bar \xi)} \leq 1/2$, then in the new coordinates, we have $| ae^{\pi^2 y}| \leq 1/2$.  Therefore, applying \eqref{eqn:Mxi0simple}, we have
    \[
    |M ( a e^{\pi^2 y}e^{2 \pi i s } ) - \xi_0| \leq 2 \Delta |a| e^{ \pi^2 y}.
    \]
\end{proof}
From Lemma $\ref{lem:M_map_bound}$, if $\sigma\gg 1/\Delta$, then the factor $e^{\pi^2 \sigma^2 \Delta (\eta-\bar{\xi})}$ is very small as soon as $\eta<\bar{\xi}$, even when $a$ is large. For instance, setting $\xi_s=\xi_0+\Delta/4$ (which lies closer to $\xi_0$ than to $\xi_1$), we have
\[
|a|e^{\pi^2\sigma^2\Delta(\xi_s-\xi_0)}=|a|e^{-\pi^2\sigma^2\Delta/4},
\]
so the vertical strip $(-\infty, \xi_s]$ in the TF plane is mapped by $\hat{\eta}_s$ into a small neighbhorhood of $\xi_0$.

To describe the global geometry of the reassignment, it is convenient to write
\[
\gamma(r,\theta):=M(re^{i\theta}), \qquad r\geq 0, \theta\in [0,2\pi),
\]
so that for fixed $(t,\eta)$,
\[
\hat{\eta}_s(t,\eta)=\gamma(r(\eta),\theta(t)),\qquad r(\eta)=|q(t,\eta)|, \theta(t)=2\pi\Delta t.
\]
As $t$ varies, the phase $e^{i\theta(t)}$ runs around the unit circle, so $t\mapsto q(t,\eta)$ traces the circle $\{z\in\mathbb{C}: |z| = r(\eta)\}$ centered at $0$. In this sense, the parameter $r=r(\eta)$ is an "interpolation parameter": for each fixed $\theta$, the curve $r\mapsto \gamma(r,\theta)$ moves from $\xi_0$ (at $r=0$) toward $\xi_1$ (as $r\to\infty)$ along a circular arc in the complex plane. The following lemma makes this precise.

\begin{lemma}\label{lem:M\"obius_arc}
Define $\delta = \frac{\xi_1 - \xi_0}{2} = \frac{\Delta}{2} > 0$. For any fixed $\theta \in (0,\pi)$, consider
\[
\gamma(r,\theta) = M(re^{i\theta})
= \frac{\xi_0 + \xi_1 r e^{i\theta}}{1 + r e^{i\theta}}, \qquad r \geq 0.
\]
Then the image of $[0,\infty)$ under $r \mapsto \gamma(r,\theta)$ is an arc of the unique circle in the Riemann sphere that passes through the three points $\xi_0, \xi_1, \bar{\xi} + i\delta \tan\left(\frac{\theta}{2}\right)$.
\end{lemma}

\begin{proof}[Proof of Lemma \ref{lem:M\"obius_arc}]
    Using $\xi_0 = \bar{\xi} - \delta$ and $\xi_1 = \bar{\xi} + \delta$, we can rewrite $M$ as
    \[
    M(z)
    = \frac{\xi_0 + \xi_1 z}{1+z}
    = \frac{(\bar{\xi} - \delta) + (\bar{\xi} + \delta)z}{1+z}
    = \bar{\xi} + \delta\frac{z - 1}{1 + z}.
    \]
    Therefore
    \[
    \gamma(r,\theta)
    = M(re^{i\theta})
    = \bar{\xi} + \delta\frac{re^{i\theta}-1}{1+re^{i\theta}}
    = \bar{\xi} - \delta\frac{1-re^{i\theta}}{1 + re^{i\theta}}.
    \]
    
    We can evaluate $\gamma$ at three specific values of $r$. First, $\gamma(0,\theta) = M(0) = \xi_0$ and $\lim_{r\to\infty} \gamma(r,\theta) = \lim_{r\to\infty} M(re^{i\theta}) = \xi_1$,
    since $M(\infty)=\xi_1$. For $r=1$ we obtain
    \[
    \gamma(1,\theta)
    = \bar{\xi}-\delta\frac{1 - e^{i\theta}}{1 + e^{i\theta}}.
    \]
    Using
    \[
    \frac{1 - e^{i\theta}}{1 + e^{i\theta}}
    = \frac{e^{i\theta/2}\left(e^{-i\theta/2} - e^{i\theta/2}\right)}{e^{i\theta/2}\left(e^{-i\theta/2} + e^{i\theta/2}\right)}
    = \frac{e^{-i\theta/2} - e^{i\theta/2}}{e^{-i\theta/2} + e^{i\theta/2}}
    = \frac{-2i\sin(\theta/2)}{2\cos(\theta/2)}
    = -i\tan\left(\frac{\theta}{2}\right),
    \]
    we get
    \[
    \gamma(1,\theta)
    = \bar{\xi} - \delta\left(-i\tan\left(\frac{\theta}{2}\right)\right)
    = \bar{\xi} + i\delta \tan\left(\frac{\theta}{2}\right).
    \]
    Now, consider the ray $L_\theta = \{re^{i\theta}: r \geq 0\}\subset\mathbb{C}$ and its completion $\widetilde{L}_\theta = L_\theta \cup \{\infty\}$ which is a generalized half-circle (the union of a straight line and the point at infinity) in the extended complex plane $\hat{\mathbb{C}}$. M\"obius transformations map generalized circles to generalized circles, so $C_\theta \defeq M(\widetilde{L}_\theta)$ is a generalized half-circle in $\hat{\mathbb{C}}$. We have shown that the three points $\xi_0=\gamma(0,\theta), \gamma(1,\theta) = \bar{\xi}+i\delta \tan\left(\frac{\theta}{2}\right),
    \xi_1=\lim_{r\to\infty} \gamma(r,\theta)$ lie on $C_\theta$. Since $\theta\in(0,\pi)$ implies $\Im \gamma(1,\theta)\neq 0$, these three points are not collinear and therefore determine the unique Euclidean circle $C_\theta$. Finally, as $r$ increases from $0$ to $\infty$, the points $re^{i\theta}$ trace the ray $L_\theta$, and by continuity of $M$ the image
    \[
    \{\gamma(r,\theta): r\geq 0\} = M(L_\theta)
    \]
    is a continuous curve contained in $C_\theta$ that starts at $\xi_0$ and tends to $\xi_1$. Thus, it is an arc of the unique circle through $\xi_0$, $\bar{\xi} + i\delta\tan(\theta/2)$, and $\xi_1$.
\end{proof}
See Figure \ref{fig:mob} for an illustration of the mapping described in Lemma \ref{lem:M\"obius_arc}. Finally, we can show that the synchrosqueezed reassignment rule for the AHM behaves closely to its locally linearized version, which we have carefully studied. The following proposition makes the error in such an approximation explicit.

\begin{figure}
    \centering
    \includegraphics[width=\linewidth]{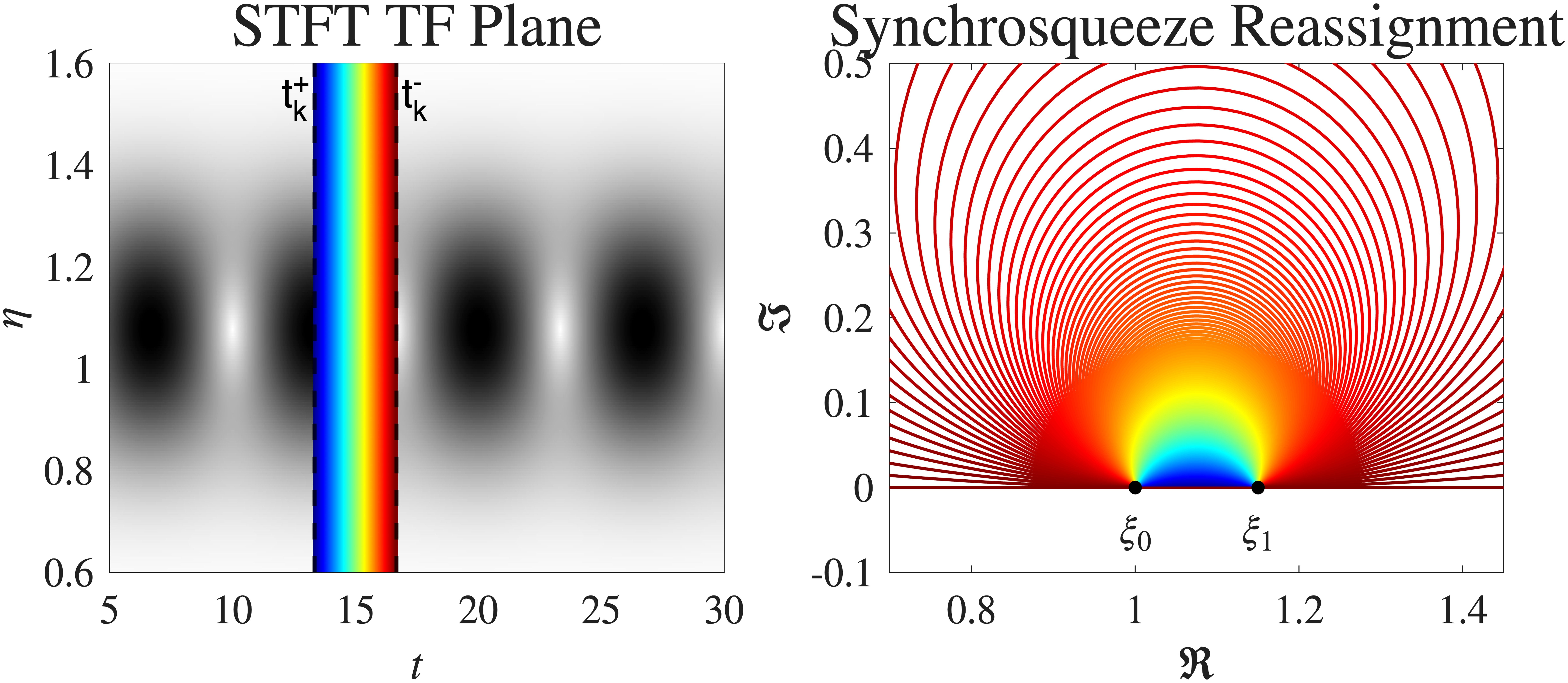}
    \caption{The same signal as described in Figure \ref{fig:Mobius_mapping} (with $\sigma=\sqrt{2}, a=1$, $\xi_0=1$, and $\xi_1=1.15$) demonstrating the circular arcs from Lemma \ref{lem:M\"obius_arc}: (Left) The STFT with colored vertical lines; (Right) The complex plane showing the circular arcs that are the images of the colored lines under the synchrosqueezed reassignment.}
    \label{fig:mob}
\end{figure}

\begin{proposition}\label{prop:AHM_omega}
    Let
    \[
    F(x)=A_0(x) e^{2\pi i \phi_0(x)} + A_1(x) e^{2\pi i \phi_1(x)}
    \]
    satisfy the AHM model in Definition \ref{Definition: AHM}.   
    Fix a time $t_*$ and set $\xi_j = \phi_j'(t_*)$, $a_j\defeq A_j(t_*)$, and $a=a_1/a_0$, and define the two-component harmonic model
    \[
    f(x)=e^{2\pi i \xi_0 x} + ae^{2\pi i \xi_1 x}.
    \]
    Assume without loss of generality (phase removal by a time shift), that the constant phase offset satisfies $\phi_j(t_*)=0$. Fix $T>0$ and $0<\beta<\frac{1}{2}$.
    Then there exists constants $\varepsilon_0>0$ and $C_{\hat{\eta}}=C_{\hat{\eta}}\left(T,\sigma,\xi_0,\xi_1,a_0,a_1,M_0'',M_1''\right)$ such that for all $0<\varepsilon\leq \varepsilon_0$, and all $t,\eta$ with $|t-t_*|\leq T$ and $|V_f^{(h)}(t-t_*,\eta)|\geq \varepsilon^\beta$ we have
    \[
    \left|{\hat{\eta}}_s^F(t,\eta)-{\hat{\eta}}_s^f(t-t_*,\eta)\right|\leq C_{\hat{\eta}}\varepsilon^{1-2\beta}
    \]
    where ${\hat{\eta}}_s^F$ and ${\hat{\eta}}_s^f$ are the synchrosqueezing reassignment rules associated with $F$ and $f$ respectively, and $C_{\hat{\eta}}$ is of the form
    \[
    C_{\hat{\eta}}=\frac{C_{Dh} + (1+|a|)\frac{\sqrt{2}}{\sigma} e^{-1/2} C_h}{\pi |a_0|}
    \]
    with $C_h,C_{Dh}$ depending on 
    the windows $h$ and $Dh$.
\end{proposition}
The proof of the above proposition is provided in Appendix \ref{sec:app_proofs}. 

\begin{remark}
    Since $\hat{\eta}_p=\Re(\hat{\eta}_s)$, we can immediately see that
    \[
    \left|{\hat{\eta}}_p^F(t,\eta)-{\hat{\eta}}_p^f(t-t_*,\eta)\right|\leq\left|{\hat{\eta}}_s^F(t,\eta)-{\hat{\eta}}_s^f(t-t_*,\eta)\right|\leq C_{\hat{\eta}}\varepsilon^{1-2\beta},
    \]
    following the above proposition.
\end{remark}

\section{Nonlinear Time-Frequency Methods}\label{sec:nonlinear_methods}

One of the central objectives of TF analysis is to construct a TFR that encodes dynamic information (e.g. AM and IF) thereby enabling users to extract the underlying signal's dynamics effectively. However, linear-type TF analysis algorithms, like the STFT or CWT, typically produce blurred TFRs due to the uncertainty principle \cite{grochenig}. This limitation motivated the development of nonlinear variations of the STFT or CWT, ranging from reassignment methods \cite{kodera_reassignment,flandrin_reassignment} to SST \cite{sst,oberlin2015second}. One main feature of these methods is the utilization of information encoded in the TFR determined by the STFT or CWT to sharpen the TFR so that users can more easily obtain desired dynamics of the underlying signal. In this section, we focus on studying SST and its variations, which utilize the reassignment method that is based on phase information encoded in the TFR determined by STFT. We quantify the spectral inference behavior of SST and consider a variant of SST to further explore how spectral interference behaves. Consider the \textit{generalized synchrosqueezed transform} as follows.

\begin{definition}[Generalized synchrosqueezed transform]
    Fix a window $h$ and a smooth mollifier kernel $g_\alpha(z):\mathbb{C}\to\mathbb{R}_+$ such that $g_\alpha\to \delta$ weakly as $\alpha\to 0$ and whose restriction to $\mathbb{R}$ has unit $L^1$ norm (i.e. $\|g_\alpha\|_{L^1(\mathbb{R})}=1$). Further define the reassignment rule $\hat{\eta}_s(t,\eta)$ as in \eqref{eqn:synchrosqueezed_reassignment} and let $G_f^{(h)}(t,\eta)$ be any complex-valued function on the TF plane where  $\forall t\in\mathbb{R}: G_f^{(h)}(t,\cdot)\in L^1(\mathbb{R})$. Then the \textit{squeezed transform} of $f$ is defined as
    \[
    S_{G,f}^{(h)}(t,\xi) = \int_{-\infty}^{\infty} G_f^{(h)}(t,\eta) g_\alpha \left(\hat{\eta}_s(t,\eta)-\xi\right)d\eta.
    \]
\end{definition}
We focus ourselves on two types of $G$ for the proceeding discussion: $G=V$ and $G=\mathbbm{1}_{\{\eta\in [-R,R]\}}$ for some $R\gg 1$.
When $G=V$, we recover the \textit{standard STFT-based synchrosqueezing transform (SST)} \cite{sst} as 
\begin{equation}\label{eqn:sst_def}
S_{V,f}^{(h)}(t,\xi) = \int_{-\infty}^\infty V_f^{(h)}(t,\eta) g_\alpha \left(\hat{\eta}_s(t,\eta)-\xi\right)d\eta\,.
\end{equation}

To study the impact of spectral interference on the generalized SST with different $G$, we make the following assumption for the rest of this paper.
\begin{assumption}\label{ass:g_alpha}
    Let $g_\alpha(z) = \frac{1}{\sqrt{\pi \alpha}} e^{-\frac{|z|^2}{\alpha}}$, where $z\in \mathbb{C}$, be a Gaussian with temporal bandwidth $\sqrt{\alpha}$ where $\alpha\ll 1$ is also fixed.
\end{assumption}
The case $G=\mathbbm{1}_{\{\eta\in [-R,R]\}}$ arises from a natural question -- how do the magnitude and phase of STFT respectively influence the reassignment result? The dependence of $G=\mathbbm{1}_{\{\eta\in [-R,R]\}}$ on $R$ is introduced for a technical reason. Note that the integral is well defined when $R$ is finite while $\alpha>0$. 
However, when $G=1$, the integral blows up. Indeed, when $\alpha>0$ and $\xi=\xi_0-\delta$ is fixed where $\delta>0$, Lemma \ref{lem:M_map_bound} tells us that $g_\alpha({\hat{\eta}}_{s,f}(t,\eta)-\xi)\geq \exp(-2\Delta\epsilon/\alpha)/\sqrt{\pi\alpha}$ for all $\eta<C$ for some constant $C$, so $S_{1,f}^{(h)}(t,\xi_0-\delta)=\infty$. We will set $R=R(\alpha)$ so that $R\to \infty$ when $\alpha\to 0$ in the following analysis. 
We shall also mention that when $G=\mathbbm{1}_{\{\eta\in [-R,R]\}}$, the generalized SST is generally non-invertible since the magnitude and phase of STFT are missing, and therefore cannot be used to decompose the signal into its individual oscillatory components.

\begin{remark}
Recall that when $G=|V|$ or $G=|V|^2$, the above formulation reduces to the original reassignment method \cite{flandrin_reassignment,chassande2003time}, in which the reassignment is performed solely along the frequency axis. Unlike SST, the original signal cannot be reconstructed from the resulting TFR since the phase information is missing. Since the analysis of this method is similar, we omit its discussion for brevity.
\end{remark}

\begin{figure}[hbt!]
    \centering    \includegraphics[width=\linewidth]{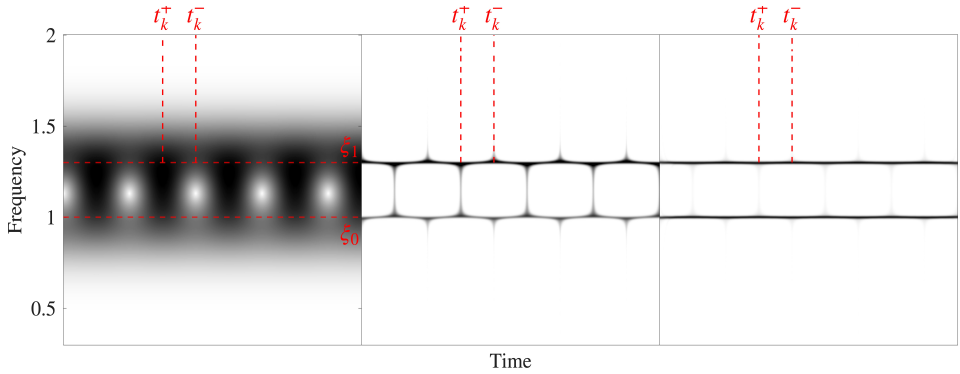}
    \caption{The same signal as described in the bottom row of Figure \ref{fig:diff_Delta_STFT} (all with $\sigma=\sqrt{2}, a=1.3$, $\xi_0=1$, and $\xi_1=1.3$) with both STFT and various squeezed transforms; (Left) STFT $|V|$; (Middle) Standard synchrosqueezing $|S^{(h)}_{V,f}|$; (Right) $S^{(h)}_{\mathbbm{1}_{\{\eta\in [-R,R]\}},f}$.}
    \label{fig:squeeze_comps}
\end{figure}

As we show later, the special case when $G:=\mathbbm{1}_{\{\eta\in [-R,R]\}}$ shows the effects of spectral interference are eliminated in our special two harmonic setup, which sheds light on the role of the STFT in inducing spectral interference.
See Figure \ref{fig:squeeze_comps} for a comparison with a clear different in terms of spectral interference. We notice that $S^{(h)}_{\mathbbm{1}_{\{\eta\in [-R,R]\}},f}$ seems to eliminate the effects of spectral interference, while spectral interference can be easily detected visually.

\subsection{Analysis of generalized SST}\label{subsec:51}

Now, we carry out a series of analysis to study the behavior of spectral interference in the generalized SST. 
To this end, it is convenient to regard the SST as the regularization of a pushforward of a certain complex measure, under the reassignment map. Fix $t$ and again let $G_f^{(h)}(t,\eta)$ be any complex-valued function on the TF plane where  $\forall t\in\mathbb{R}: G_f^{(h)}(t,\cdot)\in L^1(\mathbb{R})$. Define a complex measure on the $\eta$-axis by
\[
\mu_t(d\eta)=G_f^{(h)}(t,\eta)\,d\eta.
\]
Pushing this measure forward by the reassignment map $\eta\mapsto \hat{\eta}_s(t,\eta)$ obtains a measure $\nu_t$ on $\mathbb{C}$,
\[
\nu_t=(\hat{\eta}_s(t,\cdot))_\#\mu_t.
\]
Then $S_{G,f}^{(h)}(t,\cdot)$ is the convolution of the complex measure $\nu_t$ with the kernel $g_\alpha$:
\[
S_{G,f}^{(h)}(t,\xi)=\int_{\mathbb{C}} g_\alpha(\xi-z)\,\nu_t(dz).
\]
In the two-component model, the range of $\hat{\eta}_s(t,\cdot)$ is a $1$-dimensional curve in $\mathbb{C}$. At the constructive and destructive times considered below, we will find that this curve lies on the real axis,
\[
\hat{\eta}_s(t_k^+,\mathbb{R})=(\xi_0,\xi_1),
\qquad
\hat{\eta}_s(t_k^-,\mathbb{R})=(-\infty,\xi_0)\cup(\xi_1,\infty).
\]
For other times, $\nu_t$ is concentrated on a circle in the complex plane as described by Lemma \ref{lem:M\"obius_arc}. When $\nu_t$ is absolutely continuous with respect to $1$-dimensional Hausdorff measure $\mathcal{H}^1$ along its support curve, we write
\[
\nu_t(dz)=\Theta_{G,t}(z)\,\mathcal{H}^1(dz),
\]
and call $\Theta_{G,t}$ the pushforward density. In the real-axis cases above, $\mathcal{H}^1$ is Lebesgue measure on the corresponding subset of $\mathbb{R}$. If $\xi\in\mathbb{R}$ is such that the equation $\hat{\eta}_s(t,\eta)=\xi$ has finitely many regular solutions $\eta_j(t,\xi)$, then
\be\label{eqn:Theta_density}
\Theta_{G,t}(\xi)=\sum_{\hat{\eta}_s(t,\eta_j)=\xi}\frac{G_f^{(h)}(t,\eta_j)}{\big|\partial_\eta \hat{\eta}_s(t,\eta_j)\big|},
\ee
on the support set, and $\Theta_{G,t}(\xi)=0$ off the support. We can make this measure mapping perspective explicit at the constructive and destructive times in the following theorems.

First, consider the case when $G=\mathbbm{1}_{\{\eta\in [-R,R]\}}$ at the constructive and destructive times. We can use asymptotic approximations in terms of the small parameter $\alpha\ll 1$ in $g_\alpha$, to see that $S^{(h)}_{\mathbbm{1}_{\{\eta\in [-R,R]\}},f}$ is actually able to separate ridges for any {\em arbitrarily} small separation between two components.

\begin{theorem}[Asymptotics for $S_{\mathbbm{1}_{\{\eta\in [-R,R]\}},f}^{(h)}(t,\xi)$]\label{thm:asym_laplace_G}
Assume Assumption \ref{ass:g_alpha} holds.
    Take $S_{G,f}^{(h)}$ as defined above, and fix $G=\mathbbm{1}_{\{\eta\in [-R,R]\}}$. Fix a compact $\Xi \subset(\xi_0, \xi_1)$ and let $c_\Xi\defeq \inf_{\xi\in\Xi}\min\{(\xi-\xi_0)^2,(\xi-\xi_1)^2\}>0$ where $R=R(\alpha)=o\left(e^{c_\Xi/\alpha}\right)$ and $R\to \infty$ as $\alpha\to 0$. For the constructive and destructive times, the following piecewise asymptotics hold as $\alpha$ is sufficiently small,
    \[
    S_{\mathbbm{1}_{[-R,R]},f}^{(h)}(t_k^+,\xi)=
    \begin{cases}
    \Theta_{1,t_k^+}(\xi)+O(\alpha), & \xi\in(\xi_0,\xi_1),\\
    O\left(\exp\left(-\frac{M(\xi)}{\alpha}\right)\right), & \xi\in(-\infty,\xi_0)\cup(\xi_1,\infty),
    \end{cases}
    \]
    \[
    S_{\mathbbm{1}_{ [-R,R]},f}^{(h)}(t_k^-,\xi)=
    \begin{cases}
    O\left(\exp\left(-\frac{M(\xi)}{\alpha}\right)\right), & \xi\in(\xi_0,\xi_1),\\
    \Theta_{1,t_k^-}(\xi)+O(\alpha), & \xi\in(-\infty,\xi_0)\cup(\xi_1,\infty),
    \end{cases}
    \]
    where $\Theta_{1,t}$ is defined by \eqref{eqn:Theta_density} and admits the explicit forms
    \[
    \Theta_{1,t_k^+}(\xi)=\frac{1}{2\pi^2\sigma^2}\frac{1}{|\xi-\xi_0|\,|\xi-\xi_1|}\quad\text{for }\xi\in(\xi_0,\xi_1),
    \]
    \[
    \Theta_{1,t_k^-}(\xi)=\frac{1}{2\pi^2\sigma^2}\frac{1}{|\xi-\xi_0|\,|\xi-\xi_1|}\quad\text{for }\xi\in(-\infty,\xi_0)\cup(\xi_1,\infty),
    \]
    and $\Theta_{1,t_k^\pm}(\xi)=0$ outside the specified intervals and $M(\xi)=\min\{(\xi-\xi_0)^2,(\xi-\xi_1)^2\}$.
\end{theorem}
\begin{proof}[Proof of Theorem \ref{thm:asym_laplace_G}]
    Fix $t=t_k^+$ and define $\zeta(\eta)\defeq \hat{\eta}_s(t_k^+,\eta)$. At $t_k^+$ we have $\zeta(\mathbb{R})=(\xi_0,\xi_1)$ and $\zeta$ is strictly increasing, so for each $\xi\in(\xi_0,\xi_1)$ there is a unique $\eta_*=\eta_*(\xi)$ such that $\zeta(\eta_*)=\xi$. By definition,
    \[
    S_{1,f}^{(h)}(t_k^+,\xi)=\frac{1}{\sqrt{\pi\alpha}}\int_{-R(\alpha)}^{R(\alpha)}\exp\left(-\frac{(\zeta(\eta)-\xi)^2}{\alpha}\right)d\eta.
    \]
    Fix a compact $\Xi\subset(\xi_0,\xi_1)$. Since $R(\alpha)\to\infty$, there exists $\alpha_0>0$ such that for all $\alpha\in(0,\alpha_0]$ and all $\xi\in\Xi$ we have $\eta_*(\xi)\in(-R(\alpha),R(\alpha))$. Define
    \[
    g(\eta)\defeq (\zeta(\eta)-\xi)^2.
    \]
    Then
    \[
    g(\eta_*)=0,
    \qquad
    g'(\eta)=2(\zeta(\eta)-\xi)\zeta'(\eta)\implies g'(\eta_*)=0,
    \]
    and
    \[
    g''(\eta)=2(\zeta'(\eta))^2+2(\zeta(\eta)-\xi)\zeta''(\eta)\implies g''(\eta_*)=2(\zeta'(\eta_*))^2>0.
    \]
    Thus $\eta_*$ is a nondegenerate minimum. A Taylor expansion gives
    \[
    g(\eta)=(\zeta'(\eta_*))^2(\eta-\eta_*)^2+o\big((\eta-\eta_*)^2\big),
    \]
    uniformly for $\xi\in\Xi$. Laplace's method yields
    \[
    \frac{1}{\sqrt{\pi\alpha}}\int_{-R(\alpha)}^{R(\alpha)}e^{-g(\eta)/\alpha}d\eta
    =\frac{1}{|\zeta'(\eta_*)|}+O(\alpha),
    \]
    uniformly for $\xi\in\Xi$. Hence
    \[
    S_{1,f}^{(h)}(t_k^+,\xi)=\frac{1}{|\partial_\eta\hat{\eta}_s(t_k^+,\eta_*(\xi))|}+O(\alpha),
    \]
    which coincides with $\Theta_{1,t_k^+}(\xi)+O(\alpha)$ by \eqref{eqn:Theta_density} since the preimage is unique on $(\xi_0,\xi_1)$. It remains to compute $|\partial_\eta\hat{\eta}_s(t_k^+,\eta_*)|$ in terms of $\xi$. Use the substitution
    \[
    u=a\exp\left(-\pi^2\sigma^2(\Delta^2-2\Delta(\eta-\xi_0))\right),
    \qquad
    \frac{du}{d\eta}=2\pi^2\sigma^2\Delta u.
    \]
    At $t_k^+$ we have
    \[
    \hat{\eta}_s(t_k^+,\eta)=\xi_0+\Delta\frac{u}{1+u}.
    \]
    Differentiating gives
    \[
    \partial_\eta\hat{\eta}_s(t_k^+,\eta)=\Delta\frac{1}{(1+u)^2}\frac{du}{d\eta}
    =2\pi^2\sigma^2\Delta^2\frac{u}{(1+u)^2}.
    \]
    At $\eta=\eta_*$ the identity $\xi=\xi_0+\Delta\frac{u}{1+u}$ implies
    \[
    u=\frac{\xi-\xi_0}{\xi_1-\xi},
    \qquad
    1+u=\frac{\Delta}{\xi_1-\xi},
    \]
    and hence
    \[
    \left|\partial_\eta\hat{\eta}_s(t_k^+,\eta_*)\right|
    =2\pi^2\sigma^2\Delta^2\frac{u}{(1+u)^2}
    =2\pi^2\sigma^2|\xi-\xi_0|\,|\xi-\xi_1|.
    \]
    Therefore,
    \[
    \Theta_{1,t_k^+}(\xi)=\frac{1}{2\pi^2\sigma^2}\frac{1}{|\xi-\xi_0|\,|\xi-\xi_1|},
    \qquad \xi\in(\xi_0,\xi_1),
    \]
    as claimed.
    
    If $\xi\notin(\xi_0,\xi_1)$, then $|\zeta(\eta)-\xi|\ge \min\{|\xi-\xi_0|,|\xi-\xi_1|\}$ for all $\eta$, so
    \begin{align*}
    \left|S_{1,f}^{(h)}(t_k^+,\xi)\right|
    &\le \frac{1}{\sqrt{\pi\alpha}}\int_{-R(\alpha)}^{R(\alpha)}
    \exp\left(-\frac{\min\{(\xi-\xi_0)^2,(\xi-\xi_1)^2\}}{\alpha}\right)d\eta \\
    &\le \frac{2R(\alpha)}{\sqrt{\pi\alpha}}
    \exp\left(-\frac{\min\{(\xi-\xi_0)^2,(\xi-\xi_1)^2\}}{\alpha}\right).
    \end{align*}
    The additional assumption $R(\alpha)=o(e^{c_\Xi/\alpha})$ guarantees that this is of the stated exponentially small form uniformly for $\xi$ away from $\xi_0,\xi_1$.
    
    The case $t=t_k^-$ is identical: now $\hat{\eta}_s(t_k^-,\mathbb{R})=(-\infty,\xi_0)\cup(\xi_1,\infty)$ and the unique-preimage Laplace argument applies on compact subsets of that support, while on $(\xi_0,\xi_1)$ the same bound yields exponential smallness.
\end{proof}

Next, consider the case when $G=V$; that is, the standard SST. We have the following analysis.

\begin{theorem}[Asymptotics for $S_{V,f}^{(h)}(t,\xi)$]\label{thm:asym_laplace_sst}
    Assume Assumption \ref{ass:g_alpha} holds. Take $S_{V,f}^{(h)}$ as defined above. For the constructive and destructive times, the following piecewise asymptotics hold,
    \[
    S_{V,f}^{(h)}(t_k^+,\xi)=
    \begin{cases}
    \Theta_{V,t_k^+}(\xi)+O(\alpha), & \xi\in(\xi_0,\xi_1),\\
    O\left(\exp\left(-\frac{M(\xi)}{\alpha}\right)\right), & \xi\in(-\infty,\xi_0)\cup(\xi_1,\infty),
    \end{cases}
    \]
    \[
    S_{V,f}^{(h)}(t_k^-,\xi)=
    \begin{cases}
    O\left(\exp\left(-\frac{M(\xi)}{\alpha}\right)\right), & \xi\in(\xi_0,\xi_1),\\
    \Theta_{V,t_k^-}(\xi)+O(\alpha), & \xi\in(-\infty,\xi_0)\cup(\xi_1,\infty),
    \end{cases}
    \]
    where $\Theta_{V,t}$ is defined by \eqref{eqn:Theta_density} with $G_f^{(h)}(t,\eta)=V_f^{(h)}(t,\eta)$.
    Writing $u_*(\xi)=\frac{\xi-\xi_0}{\xi-\xi_1}$, we have
    \[
    \Theta_{V,t_k^+}(\xi)=\frac{e^{2\pi i\xi_0 t_k^+}}{2\pi^2\sigma^2}\frac{A_+(\xi)}{(\xi-\xi_0)(\xi-\xi_1)}
    \quad\text{for }\xi\in(\xi_0,\xi_1),
    \]
    \[
    \Theta_{V,t_k^-}(\xi)=-\frac{e^{2\pi i\xi_0 t_k^-}}{2\pi^2\sigma^2}\frac{A_-(\xi)}{(\xi-\xi_0)(\xi-\xi_1)}
    \quad\text{for }\xi\in(-\infty,\xi_0)\cup(\xi_1,\infty),
    \]
    and $\Theta_{V,t_k^\pm}(\xi)=0$ outside the specified intervals, where
    \[
    A_+(\xi) = \left(1-u_*(\xi)\right)\exp\left(-\frac{\pi^2 \sigma^2 \Delta^2}{4}\right)\left(-\frac{a}{u_*(\xi)}\right)^{1/2} \exp\left(-\frac{\left[\ln\left(-\frac{u_*(\xi)}{a}\right)\right]^2}{4\pi^2\sigma^2\Delta^2}\right),
    \]
    \[
    A_-(\xi) = \left|1-u_*(\xi)\right|\exp\left(-\frac{\pi^2 \sigma^2 \Delta^2}{4}\right)\left(\frac{a}{u_*(\xi)}\right)^{1/2} \exp\left(-\frac{\left[\ln\left(\frac{u_*(\xi)}{a}\right)\right]^2}{4\pi^2\sigma^2\Delta^2}\right),
    \]
     and $M(\xi)=\min\{(\xi-\xi_0)^2,(\xi-\xi_1)^2\}$.
\end{theorem}
\begin{proof}[Proof of Theorem \ref{thm:asym_laplace_sst}]
    Set $\tau=2\pi\Delta t$. By definition,
    \[
    S_{V,f}^{(h)}(t,\xi) = \frac{1}{\sqrt{\pi\alpha}}\int_{-\infty}^{\infty} V_f^{(h)}(t,\eta)\exp\left(-\frac{|\hat{\eta}_s(t,\eta)-\xi|^2}{\alpha}\right)d\eta. 
    \]
    We can apply the change of variables $u=a\exp\left(-\pi^2\sigma^2\left(\Delta^2-2\Delta(\eta-\xi_0)\right)\right)$ with $\frac{du}{d\eta} = 2\pi^2\sigma^2\Delta u$ which maps $\eta\in\mathbb{R}$ to $u\in (0,\infty)$ and gives
    \[
    V_f^{(h)}(t,\eta) = e^{2\pi i\xi_0 t} \exp\left(-\frac{\pi^2\sigma^2\Delta^2}{4}\right)\left(\frac{a}{u}\right)^{1/2} \exp\left(-\frac{\left[\ln(u/a)\right]^2}{4\pi^2\sigma^2\Delta^2}\right)(1+ue^{i\tau})
    \]
    and
    \[
    \hat{\eta}_s(t,\eta)=\xi_0+\Delta\frac{ue^{i\tau}}{1+ue^{i\tau}}.
    \]
    Notice also that for $t=t_k^+$ we have $e^{i\tau}=1$ and the map $\eta\mapsto\hat{\eta}_s(t,\eta)$ is strictly increasing from $\xi_0$ to $\xi_1$. So, for each $\xi\in (\xi_0,\xi_1)$, there is a unique $\eta_*$ with $\hat{\eta}_s(t_k^+,\eta_*)=\xi$. At $\eta_*$, $u_*=\frac{\xi-\xi_0}{\xi_1-\xi}=-u_*(\xi)$ and 
    \[
    \left|\partial_\eta \hat{\eta}_s(t_k^+, \eta_*)\right| = -2\pi^2\sigma^2 (\xi-\xi_0)(\xi-\xi_1).
    \]
    Applying Laplace's method in the quadratic form at the nondegenerate minimum, we can set
    \[
    F(\eta)=V_f^{(h)}(t_k^+,\eta),\qquad g(\eta)=|\hat{\eta}_s(t_k^+,\eta)-\xi|^2,
    \]
    so 
    \[
    g(\eta_*)=g'(\eta_*)=0\quad \text{and} \quad g''(\eta_*)=2\left(\partial_\eta\hat{\eta}_s(t_k^+,\eta_*)\right)^2>0.
    \]
    Taylor expansion gives
    \[
    g(\eta)=\left(\partial_\eta\hat{\eta}_s(t_k^+,\eta_*)\right)^2(\eta-\eta_*)^2+o\left((\eta-\eta_*)^2\right)
    \]
    so that
    \begin{align*}
        S_{V,f}^{(h)}(t_k^+,\xi)
        &= \frac{1}{\sqrt{\pi\alpha}}\int_{-\infty}^\infty 
        F(\eta)\exp\left(-\frac{g(\eta)}{\alpha}\right)d\eta\\
        &= \frac{1}{\sqrt{\pi\alpha}}\int_{-\infty}^\infty\left(F(\eta_*)+O(|\eta-\eta_*|)\right)\\
        &\qquad\qquad\quad\times\exp\left(-\frac{\left(\partial_\eta\hat{\eta}_s(t_k^+,\eta_*)\right)^2(\eta-\eta_*)^2}{\alpha}+o\left(\frac{(\eta-\eta_*)^2}{\alpha}\right)\right)d\eta\\
        &= \frac{F(\eta_*)}{\sqrt{\pi\alpha}}\int_{-\infty}^\infty\exp\left(-\frac{\left(\partial_\eta\hat{\eta}_s(t_k^+,\eta_*)\right)^2(\eta-\eta_*)^2}{\alpha}\right)d\eta + O(\alpha)\\
        &= \frac{F(\eta_*)}{\left|\partial_\eta\hat{\eta}_s(t_k^+,\eta_*)\right|}+O(\alpha).
    \end{align*}
    Substituting the above forms gives the resulting approximation. In the complementary $\xi$-regions, there is no interior preimage, so the integral is exponentially small in $\alpha$. Specifically, if $\xi$ lies outside the region where a unique interior preimage exists (that is, $\xi\notin(\xi_0,\xi_1)$ at $t_k^+$ or $\xi\in(\xi_0,\xi_1)$ at $t_k^-$), then $g(\eta)=|\hat{\eta}_s(t,\eta)-\xi|^2$ satisfies $\inf_\eta g(\eta)=c(\xi)>0$, and
    \[
    \frac{1}{\sqrt{\pi\alpha}}\int_{\mathbb{R}}|V_f^{(h)}(t,\eta)|e^{-g(\eta)/\alpha}d\eta
    \le \frac{\|V_f^{(h)}(t,\cdot)\|_\infty}{\sqrt{\pi\alpha}}\int_{\mathbb{R}}e^{-c(\xi)/\alpha}d\eta
    = O\left(e^{-c(\xi)/\alpha}\right).
    \]
    The destructive time case follows the same application of Laplace's method, with the roles of the intervals being reversed.
\end{proof}
See Appendix \ref{app:animations} for an animation of Theorems \ref{thm:asym_laplace_G} and \ref{thm:asym_laplace_sst} as $\Delta$ changes.

\begin{remark}
At constructive times, the formal choice $G\equiv 1$ leads to the pushforward density
\[
\Theta(\xi)=\frac{1}{2\pi^2\sigma^2}\frac{1}{|\xi-\xi_0|\,|\xi-\xi_1|}
\quad\text{for }\xi\in(\xi_0,\xi_1),
\]
and
\[
\Theta(\xi)=0\quad\text{for }\xi\in(-\infty,\xi_0)\cup(\xi_1,\infty).
\]
This density is not integrable near $\xi_0$ and $\xi_1$, so the pushed-forward measure has infinite total mass and convolving with $g_\alpha$ produces $+\infty$. Truncating to $G=\mathbbm{1}_{[-R,R]}$ is a separate technical step whose role is only to make the regularization finite; the structural behavior is encoded in $\Theta$.
\end{remark}
Notice the difference in the asymptotic forms for $S_{\mathbbm{1}_{\{\eta\in [-R,R]\}},f}^{(h)}$ and $S_{V,f}^{(h)}(t,\xi)$ described by Theorems \ref{thm:asym_laplace_G} and \ref{thm:asym_laplace_sst}. In particular, for the constructive time $t_k^+$, the asymptotic form for $S_{\mathbbm{1}_{\{\eta\in [-R,R]\}},f}^{(h)}$ always gives two peaks at the two IFs $\xi_0$ and $\xi_1$, while the standard SST $S_{V,f}^{(h)}(t,\xi)$ gives one peak at constructive times when $\Delta$ is small. This can be numerically noticed, along with the asymptotic forms, in Figure \ref{fig:asymptotics} and an animation included in Appendix \ref{app:animations}.

\begin{figure}
    \centering
    \includegraphics[width=\linewidth]{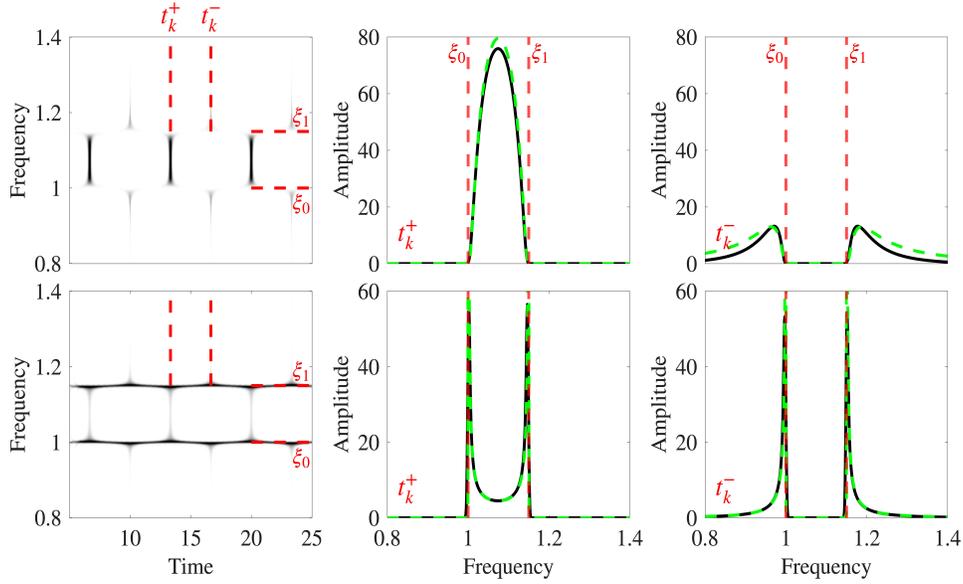}
    \caption{The same signal as described in Figures \ref{fig:Mobius_mapping} and \ref{fig:phase_vs_sst_tfrs} (with $\sigma=\sqrt{2}, a=1$, $\xi_0=1$, and $\xi_1=1.15$) where the top row is (left) $|S_{V,f}|$; (middle) its cross section at the constructive time $t_k^+$; (right) its cross section at the destructive time $t_k^-$; and (bottom row) the same for $|S_{\mathbbm{1}_{\{\eta\in [-R,R]\}},f}^{(h)}|$ with asymptotic approximations from Theorems \ref{thm:asym_laplace_G} and \ref{thm:asym_laplace_sst} in green.}
    \label{fig:asymptotics}
\end{figure}

\begin{remark}
    The same Laplace method above applies for any fixed time $t$, but the critical point structure depends on $\tau=2\pi \Delta t$. For example, writing
    \[
    S_{V,f}^{(h)}(t,\xi)=\frac{1}{\sqrt{\pi\alpha}}\int_{\mathbb{R}} V_f^{(h)}(t,\eta)\exp\left(-\frac{|\hat{\eta}_s(t,\eta)-\xi|^2}{\alpha}\right)d\eta,
    \]
    an interior contribution arises when $\hat{\eta}_s(t,\eta)=\xi$ has a real solution $\eta_*$. With $u(\eta)=a\exp\left(-\pi^2\sigma^2\left(\Delta^2-2\Delta(\eta-\xi_0)\right)\right)>0$, this gives $u_*=\frac{\xi-\xi_0}{\xi_1-\xi}e^{-i\tau}$. Since $u(\eta)\in(0,\infty)$, a real preimage exists if and only if $u_*\in (0,\infty)$. In those cases,
    \[
    S_{V,f}^{(h)}(t,\xi)=\frac{V_f^{(h)}(t,\eta_*)}{\left|\partial_\eta\hat{\eta}_s(t,\eta_*)\right|}+O(\alpha),
    \]
    and the amplitudes follow by evaluating $V_f^{(h)}(t,\eta_*)$ and $\partial_\eta\hat{\eta}_s(t,\eta_*)$ at the corresponding $u_*$. For generic $t$ and $\xi$ pairs that do not satisfy the condition $u_*\in(0,\infty)$, there is no real critical point so $\min_{\eta\in\mathbb{R}}|\hat{\eta}_s(t,\eta)-\xi|^2>0$ and
    \[
    S_{V,f}^{(h)}(t,\xi)=O\left(e^{-c(t,\xi)/\alpha}\right).
    \]
    A full asymptotic approximation can still be obtained by analytic continuation and steepest descent, but this introduces branch choices for the $\log$ term and case splits. We therefore restrict to $t_k^+$ and $t_k^-$ to keep the presentation simple.
\end{remark}

\subsection{Asymptotic behavior of standard SST}

The previous asymptotics hold for small $\alpha$. In this subsection, we compute the SST in a closed form when $\sigma\Delta$ is either large or small, and what the resulting ridge implications are.
Write the SST explicitly as
\begin{align*}
    S_{f}^{(h)}(t,\xi)&=\int_{-\infty}^\infty V_f^{(h)}(t,\eta) g_\alpha \left(\hat{\eta}_{s}(t,\eta)-\xi\right)d\eta\\
    &= \frac{1}{\sqrt{\pi \alpha}} e^{2\pi i \xi_0 t}\int_{-\infty}^\infty\left(\hat h(\eta-\xi_0)+ae^{2\pi i \Delta t}\hat h(\eta-\xi_1)\right) e^{-|\hat{\eta}_s(t,\eta)-\xi|^2/\alpha} d\eta.
\end{align*}
Since $\hat{\eta}_{s,f_0}(t,\eta)=\xi_0$ and $\hat{\eta}_{s,f_1}(t,\eta)=\xi_1$, we also have
\begin{equation*}
S_{f_0}^{(h)}(t,\xi)=\frac{1}{\pi\sigma\sqrt{\alpha}} e^{2\pi i\xi_0 t}e^{-\frac{1}{\alpha}(\xi_0-\xi)^2} \quad\text{and}\quad S_{f_1}^{(h)}(t,\xi)=\frac{a}{\pi\sigma\sqrt{\alpha}} e^{2\pi i\xi_1 t}e^{-\frac{1}{\alpha}(\xi_1-\xi)^2},
\end{equation*}
where the constant comes from integrating over the Gaussians $\hat{h}$.

\subsubsection{When $\sigma \Delta$ is "large"}
When $\sigma\Delta$ is large, as expected, SST is able to separate components well. This observation is made explicit in the following proposition.
\begin{proposition}\label{prop:sst_large_Delta}
    When $\sigma\Delta>0$ is sufficiently large, we have 
    \[
    S_{f}^{(h)}(t,\xi)= S_{f_0}^{(h)}(t,\xi)+S_{f_1}^{(h)}(t,\xi)+O(e^{-\pi^2\sigma^2\Delta^2/4})\,.
    \]
\end{proposition}

The proof of the above proposition is provided in Appendix \ref{sec:app_proofs}. Even when $\sigma\Delta$ is not large, we should still expect that when the amplitudes of the two components are highly unbalanced, then SST should approximate the SST of the component with larger relative amplitude. This observation can also be made precise by the following proposition.

\begin{proposition}\label{prop:sst_large_a}
    When $a\geq0$ is sufficiently small, we have
    \[
    S_{f}^{(h)}(t,\xi)= S_{f_0}^{(h)}(t,\xi)+O(a),
    \]
    and when $a>0$ is sufficiently large, we have
    \[
    S_{f}^{(h)}(t,\xi)= S_{f_1}^{(h)}(t,\xi)+O(1/a).
    \]
\end{proposition}
The proof of the above proposition is provided in Appendix \ref{sec:app_proofs}. 

\subsubsection{When $\sigma\Delta$ is "small"}

The region where $\sigma\Delta$ is small is a particularly interesting region since spectral interference appears in the STFT, and as one may expect intuitively, spectral interference also appears in the SST. While Lemma \ref{lem:stft_const} provides an explicit critical gap for how small $\Delta$ may be (relative to $\sigma$) for the STFT to separate two components, finding an analogous bound for SST is not straightforward due to the nonlinearity of both reassignment and SST. However, the measure mapping perspective introduced in Section \ref{subsec:51} provides some insight into regions where the synchrosqueezed reassignment map pushes forward nontrivial mass. Using this insight, we can explicitly map out contributions to the SST integral that produce meaningful contribution.

\begin{theorem}\label{thm:sst_delta_critical}
    Assume $C=\frac{\Delta}{4\sqrt{\alpha}}$. Based on Theorem \ref{thm:sst_const_asymp}, define 
    \begin{align*}
        H(\xi)&=\alpha^{-1/2}[\erf\left(\pi\sigma(\gamma_1-\xi_0)\right)-\erf\left(\pi\sigma(\gamma_2-\xi_0)\right)\\
        &\quad +a\erf\left(\pi\sigma(\gamma_1-\xi_1)\right)-a\erf\left(\pi\sigma(\gamma_2-\xi_1)\right)]\,.
    \end{align*}
    When $\alpha$ is sufficiently small, up to an $O\left(\alpha^{-1/2} e^{-C^2}\right)$ correction,  there exists a pair $(\Delta,r)$ with $r>0$ such that, with
    \[
    s_1=-\frac{\Delta r}{a+r},\qquad s_2=-\frac{\Delta}{a r+1},
    \]
    \[
    z_1=\frac{\pi\sigma\Delta}{2}-\frac{\ln r}{2\pi\sigma\Delta},\qquad
    z_2=\frac{\pi\sigma\Delta}{2}+\frac{\ln r}{2\pi\sigma\Delta},
    \]
    and
    \[
    \gamma_1'=-\frac{1}{2\pi^2\sigma^2 s_1(s_1+\Delta)},\quad
    \gamma_2'=-\frac{1}{2\pi^2\sigma^2 s_2(s_2+\Delta)},
    \]
    \[
    \gamma_1''=\frac{2s_1+\Delta}{2\pi^2\sigma^2 s_1^2(s_1+\Delta)^2},\quad
    \gamma_2''=\frac{2s_2+\Delta}{2\pi^2\sigma^2 s_2^2(s_2+\Delta)^2},
    \]
    the following two equations hold:
    \be
    r=\frac{\gamma_2'-a\gamma_1'}{\gamma_1'-a\gamma_2'},\label{eqn:R_eqn}
    \ee
    \be
    (r+a)\gamma_1''-(1+a r)\gamma_2''+2\pi\sigma\left[(-z_1 r+a z_2)(\gamma_1')^2+(z_2-a r z_1)(\gamma_2')^2\right]=0.\label{eqn:K_eqn}
    \ee
    Denote $\Delta_{\text{critical, SST}}$ to be the solution $\Delta$.
    For any solution $(\Delta,r)$ of \eqref{eqn:R_eqn} and \eqref{eqn:K_eqn}, the corresponding critical point of $H$ is $\xi_c=\xi_0+\frac{3\Delta}{4}+s_1$ so that $\tilde H'(\xi_c)=\tilde H''(\xi_c)=0$.

     At times $t=t_{k}^{+}$, $|S_f^{(h)}(t_{k}^{+},\eta)|$ has one local maximum when $\Delta<\Delta_{\text{critical, SST}}$ and two local maxima when $\Delta>\Delta_{\text{critical, SST}}$. 
\end{theorem}

See Appendix \ref{subsect:app_proof_theorem_sst_crit} for details on the proof of this theorem. The aforemendtioned section in the appendix also includes intermediary theorems on the closed form expressions for SST at the two interference times, specifically, Theorems \ref{thm:sst_const_asymp} and \ref{thm:sst_dest_asymp}, which are analagous versions of Propositions \ref{prop:sst_large_Delta} and \ref{prop:sst_large_a} for the region where $\sigma\Delta$ is small and $a=O(1)$. Furthermore, Appendix \ref{app:animations} includes animations of the aforementioned theorems as $\Delta$ varies. With this theorem, we can now compare SST and STFT and see that when $a=1$, SST is able to separate two components with smaller frequency gap by a direct comparison between $\Delta_{\text{critical, STFT}}$ and $\Delta_{\text{critical, SST}}$. See Figure \ref{fig:compare_criticals} for an illustration.
When $a=1$, \eqref{eqn:R_eqn} yields $r=\frac{1-\gamma_1'/\gamma_2'}{\gamma_1'/\gamma_2'-1}=\frac{1}{3}$ at the double root (equivalently $z_1^2-z_2^2=\ln 3$). Letting $z_1=\frac{\pi\sigma\Delta}{2}+\frac{\ln 3}{2\pi\sigma\Delta}$ and $z_2=\frac{\pi\sigma\Delta}{2}-\frac{\ln 3}{2\pi\sigma\Delta}$, equation \eqref{eqn:K_eqn} simplifies to
\[
\pi\sigma\Delta-\frac{2}{3}\frac{\ln 3}{\pi\sigma\Delta}=0,
\]
so $(\pi\sigma\Delta)^2=\frac{2}{3}\ln 3$ and thus
\[
\Delta_{\text{critical, SST}}(a=1)=\frac{1}{\pi\sigma}\sqrt{\frac{2\ln 3}{3}}.
\]

On the other hand, when $a=1$ Remark \ref{rem:critical_delta_a1} tells us that for $\Delta_{\text{critical, STFT}}(a=1)=\frac{\sqrt{2}}{\pi\sigma}$.
As a result, in the balanced amplitudes case, SST is able to separate components with smaller frequency gap, and the critical value for the separation between components is reduced by a factor of $\sqrt{\frac{\ln 3}{3}}\approx 0.6$ when applying SST.

\begin{figure}[hbt!]
    \centering
    \includegraphics[width=\linewidth]{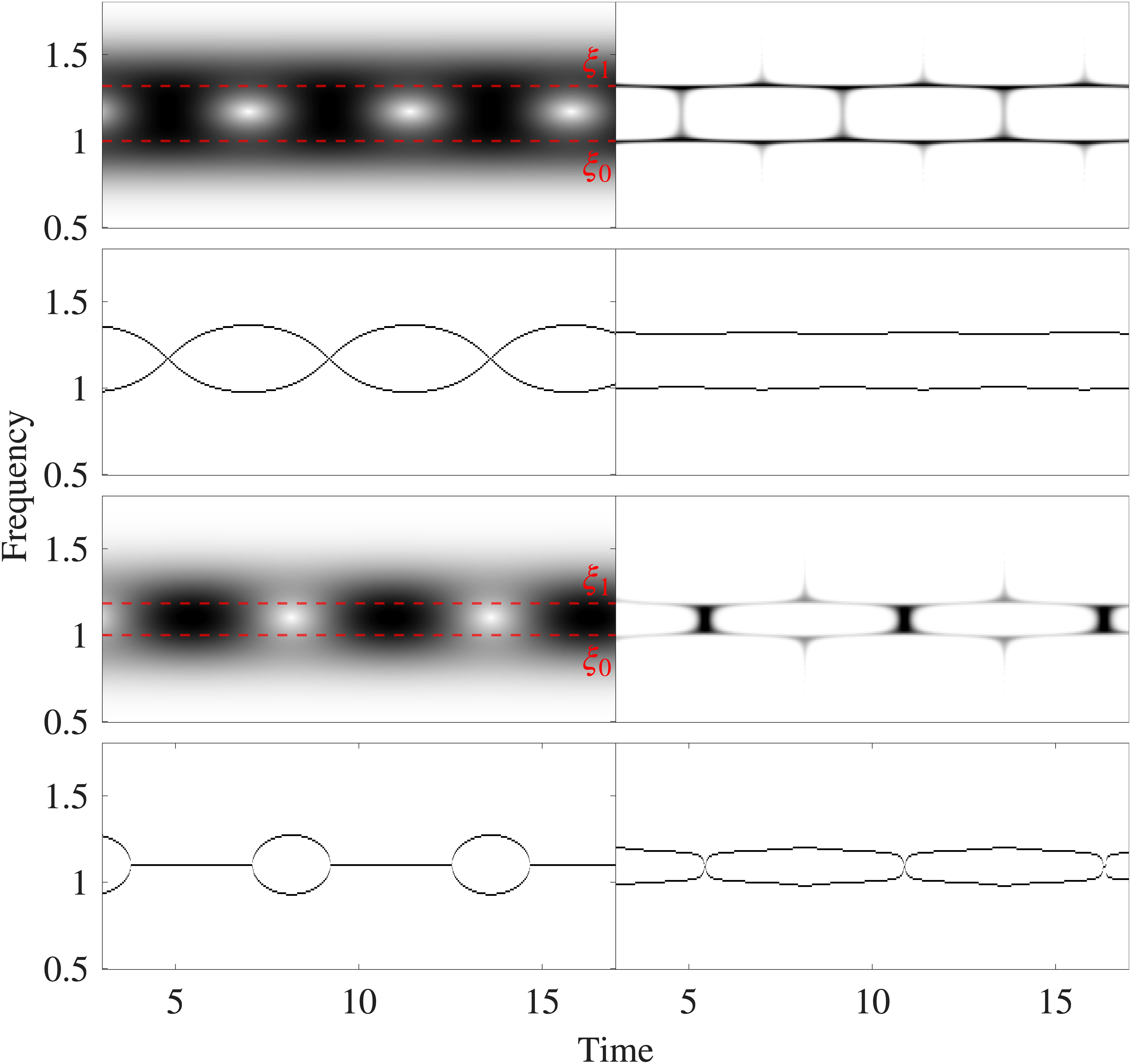}
    \caption{All with $\sigma=\sqrt{2}, a=1$ and $\xi_0=1$; (Left column) STFT and (Right column) SST; (Top four) time frequency representations on top and ridges on bottom with $\Delta=\Delta_{\text{cricitcal, STFT}}=\frac{\sqrt{2}}{\pi\sigma}$; (Bottom four) same as above with $\Delta=\Delta_{\text{cricitcal, SST}}=\frac{1}{\pi\sigma}\sqrt{\frac{2\ln{3}}{3}}$.}
    \label{fig:compare_criticals}
\end{figure}

\section{Discussion and future direction}

The standard SST is defined through the synchrosqueezing reassignment, but a natural question arises: is it optimal? Numerous studies have demonstrated that incorporating higher-order derivatives of the phase function can yield improved results, leading to the so-called second- \cite{oberlin2015second} and higher-order \cite{pham2017high} SSTs. Previous analyses have further suggested that the nonlinear relationship between the magnitude and phase of the STFT contributes to spectral interference patterns observed in the resulting TFR. In this paper we consider the differences between the synchrosqueezed reassignment and the phase-derived reassignment introduced earlier, which can be viewed as the ``zero''-th order SST compared with the higher order SST \cite{oberlin2015second,pham2017high}. For comparison, we examine the SST described above alongside an alternative version that employs the phase-derived frequency reassignment $\hat{\eta}_p(t,\eta)$ in place of ${\hat{\eta}_s}(t,\eta)$. Doing so, we can compare the effects of both types of reassignment (using the same parameters as in Figure \ref{fig:Mobius_mapping}) in Figure \ref{fig:phase_vs_sst_tfrs}. In particular, one may notice that using phase reassignment visually introduces a certain blurring at destructive times while also suffering from stronger spectral interference, which can be seen as stronger blurring at constructive times. From the ridge plots, one may also notice that the length of the intervals of the constructive time sets is smaller when using synchrosqueezing reassignment than for phase reassignment. We postpone the exploration of how purely imaginary term in \eqref{eqn:extra_term} provides benefit in the synchrosqueezed reassignment, along with the exploration of spectral interference of high-order SST, to future work. 

In practice, signals are discrete. To relate our continuous analysis to this setting, we adopt a discretization by sampling the interval $[-\sqrt{n}/2, \sqrt{n}/2]$ on a uniform grid of size $n$, with sampling period $1/\sqrt{n}$. This corresponds to a sampling rate $\sqrt{n}$ Hz and frequency resolution $1/\sqrt{n}$. Under this scheme, the STFT, reassignment, and SST are implemented via Riemann-sum approximations, and the resulting TFR converges to its continuous counterpart as $n\to \infty$. See \cite{chen2014non} for an example. While beyond the scope of this paper, a major challenge arises when the input is a random process. Recent results show that \cite{wu2025uncertainty}, under suitable moment and dependence conditions, the discretized TFR can be approximated by a Gaussian random field. Understanding how noise affects spectral interference in this regime is an interesting direction for future work.

In this paper, we focus on signals with two components. In practice, we need to handle signals with multiple components. For the multiple harmonic case $f^\dagger(t)=\sum_{k=1}^n a_k e^{2\pi i \xi_k t}$,    where $n>2$, we can evaluate the reassignment operator as $\hat{\eta}_{s,f^\dagger}(t,\eta)=M^\dagger\left(q_1(t,\eta),\ldots q_n(t,\eta)\right)$, where $q_k(t,\eta)=a_k e^{2\pi i \xi_k t} \hat{h}(\eta-\xi_k)$ and $M^\dagger:\mathbb{C}^n\to\mathbb{C}$ is defined by
$M^\dagger(z_1,\ldots,z_n)=\frac{\sum_{k=1}^n \xi_k z_k}{\sum_{k=1}^n z_k}$.    Observe that for each $k=1,\ldots, n$ we have
$M^\dagger(0,\ldots,0,z_k,0,\ldots,0)=\xi_k$ for all $z_k\in \mathbb{C}$ and $M^\dagger(z_1,\ldots, z_{k-1}, 1, z_{k+1}, \ldots z_n)\approx \xi_k$
if $\max_{j\neq k} |z_j|\ll 1$. Notice that in this setting $M^\dagger$ looks like a higher dimensional version of the linear fractional transformation we previously studied. When multiple consecutive frequencies are close, the analysis is more challenging than the two components case. For example, when $n=3$ and $\xi_{i+1}-\xi_i=\Delta_i$, $i=1,2$, are small, $a_1e^{2\pi i\xi_1t}$ does not only interfere with $a_2e^{2\pi i\xi_2t}$ but also with $a_3e^{2\pi i\xi_3t}$. Moreover, in practice one may use non-Gaussian windows in the STFT. Although we do not pursue this systematically here, different windows clearly induce different spectral interference patterns. We leave an exploration of the multiple components case and the effect of the window choice on spectral interference to future work.

\begin{figure}[hbt!]
    \centering
    \includegraphics[width=\linewidth]{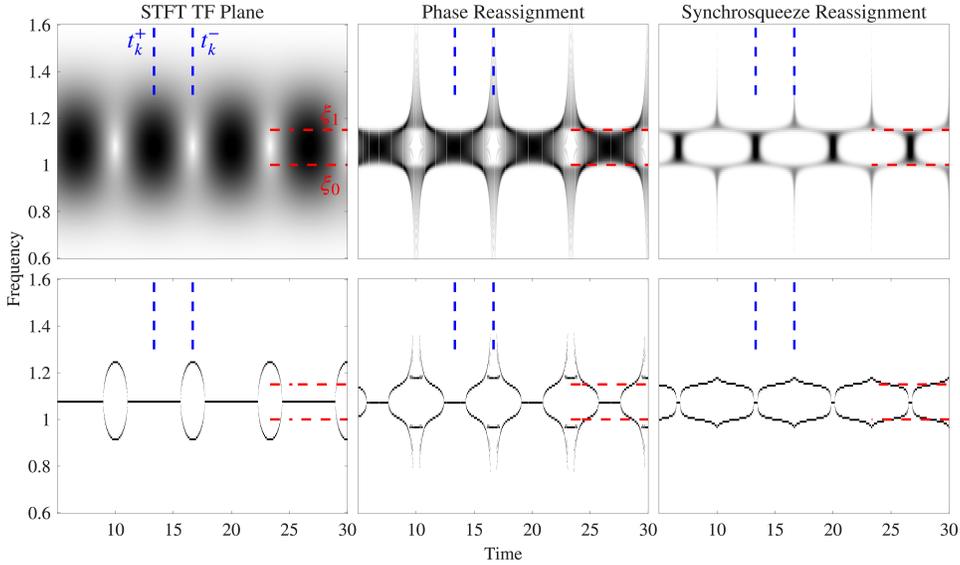}
    \caption{The same signal as described in Figure \ref{fig:Mobius_mapping} (with $\sigma=\sqrt{2}, a=1$, $\xi_0=1$, and $\xi_1=1.15$) where the top row is (left) STFT; (middle) synchrosqueezing using phase reassignment $\hat{\eta}$; (right) synchrosqueezing using the synchrosqueeze reassignment ${\hat{\eta}_s}$; and their respective ridges on the bottom row.}
    \label{fig:phase_vs_sst_tfrs}
\end{figure}

\section*{Acknowledgement}

The work of SC was funded in part by the US National Science Foundation via award DMS-2038056.

\printbibliography

\newpage

\appendix
\markboth{CHAND, NOLEN, AND WU}{APPENDIX}
\section{Proofs} \label{sec:app_proofs}
\begin{proof}[Proof of Proposition \ref{prop:stft_large_Delta}]
We define the positive functions
\[
A(\eta)\defeq\hat h(\eta-\xi_0),\qquad B(\eta)\defeq a\hat h(\eta-\xi_1).
\]
Notice that $A(\eta)=|V_{f_0}^{(h)}(t,\eta)|$ and $B(\eta)=|V_{f_1}^{(h)}(t,\eta)|$. Then \eqref{eqn:stft_simplified} reads $|V_f^{(h)}(t,\eta)|=|A(\eta)+e^{i\psi(t)}B(\eta)|$ with $\psi(t):=2\pi\Delta t$. The lower bound is the triangle inequality. The upper bound follows from
\[
(A+B)-|A+e^{i\psi}B| \leq (A+B)-|A-B|=2\min\{A,B\}.
\]
For the Gaussian, $A(\eta)=e^{-\pi^2\sigma^2(\eta-\xi_0)^2}$ and $B(\eta)=ae^{-\pi^2\sigma^2(\eta-\xi_1)^2}$, the function $\eta\mapsto\min\{A(\eta),B(\eta)\}$ attains its maximum where $A=B$, namely at $\eta=(\xi_0+\xi_1)/2=\bar{\xi}$, giving the stated uniform bound. 
\end{proof}

\begin{proof}[Proof of Lemma \ref{lem:stft_const}]
    Recall from \eqref{eqn:stft_simplified} that can write
    \[
        |V_f^{(h)}(t,\eta)|=\left|e^{-\pi^2\sigma^2(\eta-\xi_0)^2}+ae^{2\pi i \Delta t} e^{-\pi^2\sigma^2(\eta-\xi_1)^2}\right|.
    \]
    When $t=t_{k}^{+}$, this is 
    \[
        |V_f^{(h)}(t_{k}^{+},\eta)|=e^{-\pi^2\sigma^2(\eta-\xi_0)^2}+ae^{-\pi^2\sigma^2(\eta-\xi_1)^2}.
    \]
    We can notice that when the supports of each Gaussian overlap significantly, then at time $t_{k}^{+}$, the Gaussians interfere additively when $\Delta$ is small enough such that the supports of each Gaussian overlap significantly. In this case, we observe a single peak in $|V_f^{(h)}(t_{k}^{+},\eta)|$, and otherwise when $\Delta$ is large enough such that the Gaussians are well separated, we have two separate peaks. In order to find such a critical condition on $\Delta$, we can take the first and second derivatives of $|V_f^{(h)}(t_{k}^{+},\eta)|$ and set them to 0 to get a system of equations for the critical condition. Applying the change of variables $x=\eta-\bar{\xi}$ and letting $C=\pi^2\sigma^2$, we have
    \[
    g(x)\defeq |V_f^{(h)}(t_{k}^{+},\eta)|=e^{-C(x+\Delta/2)^{2}}+ae^{-C(x-\Delta/2)^{2}}.
    \]
    Now, the first derivative 
    \[
    g'(x)=-2C\left[\left(x+\frac{\Delta}{2}\right)e^{-C(x+\Delta/2)^{2}}+a\left(x-\frac{\Delta}{2}\right)e^{-C(x-\Delta/2)^{2}}\right]=0
    \]
    gives us the equation
    \[
    (x+\frac{\Delta}{2})e^{-C(x+\Delta/2)^{2}}+a(x-\frac{\Delta}{2})e^{-C(x-\Delta/2)^{2}}=0.
    \]
    Introducing some convenient notation, 
    \begin{align*}
    u\defeq x+\frac{\Delta}{2},\qquad v\defeq x-\frac{\Delta}{2}=u-\Delta, \\
    r\defeq e^{-C(v^{2}-u^{2})}=e^{-C\Delta^{2}+2Cu\Delta}>0,
    \end{align*}
    our first derivative equation becomes
    \[
    u+avr=0.
    \]
    Doing the same for the second derivative
    \begin{align*}
        g''(x)&=2Ce^{-C(x+\Delta/2)^{2}}\left[1-2C\left(x+\frac{\Delta}{2}\right)^{2}\right]\\
        &\quad+2Cae^{-C(x-\Delta/2)^{2}}\left[1-2C\left(x-\frac{\Delta}{2}\right)^{2}\right]=0,
    \end{align*}
    we get $$1-2Cu^{2}+ar(1-2Cv^{2})=0.$$ Now, if we define $s\defeq ar>0$, we can write the first derivative equation as
    \[
    u+s(u-\Delta)=0\implies u=\frac{s\Delta}{1+s},v=-\frac{\Delta}{1+s}
    \]
    and plugging the first equation into the second gives us
    \[
    \Delta_{\text{critical, STFT}}=\frac{1+s}{\sqrt{2Cs}}.
    \]
    At first glance, this equation for $\Delta_{\text{critical, STFT}}$ seems like it depends on $\eta$ since $s$ is a function of $\eta$. However, $s$ can be written solely as a function of $a$ by the following relation. First we have
    \[
    v^2-u^2=\left(\Delta\frac{1-s}{1+s}\right)^2\implies r=\exp{\left(-C\Delta^2 \frac{1-s}{1+s}\right)}
    \]
    so,
    \begin{align*}
        \ln r&=-\frac{(1+s)^2}{2s}\frac{1-s}{1+s}\\
        &=\frac{s^2-1}{2s}\\
        &=\frac{1}{2}\left(s-\frac{1}{s}\right).
    \end{align*}
    This gives the relation of $s$ as a function of $a$ as 
    \begin{equation}
        \ln\frac{s}{a}=\frac{1}{2}\left(s-\frac{1}{s}\right).\notag
    \end{equation}
\end{proof}

\begin{proof}[Proof of Lemma \ref{lem:stft_const_worst_case}]
    Set $x=\eta-\bar{\xi}$. If we define $F(t,x)\defeq|V_f^{(h)}(t,x)|^2$ we have
    \be\label{eqn:F}
    F(t,x)=e^{-2C(x+\Delta/2)^2}+a^2e^{-2C(x-\Delta/2)^2}+2ace^{-C(2x^2+\Delta^2/2)}
    \ee
    where $c=\cos(2\pi\Delta t)\in[-1,1]$. For any fixed $t$ (or equivalently fixed $c$), a change in the number of local maxima of $\eta\mapsto |V_f^{(h)}(t,\eta)|$ can occur only at a degenerate critical point of $x\mapsto F(t,x)$, i.e. at $(x,c)$ where simultaneously
    \[
    \partial_x F(t,x)=0\quad\text{and}\quad \partial_{xx}F(t,x)=0.
    \]
    We can compute these two equations explicitly, eliminate $c$, and obtain a $c$–free algebraic condition that characterizes all such degeneracies. First, differentiating \eqref{eqn:F},
    \begin{align*}        
    \partial_x F(t,x)&=-4C\bigg[\left(x+\frac{\Delta}{2}\right)e^{-2C(x+\Delta/2)^2}+\left(x-\frac{\Delta}{2}\right)a^2e^{-2C(x-\Delta/2)^2}\\
    &\qquad\qquad+2acxe^{-C(2x^2+\Delta^2/2)}\bigg].
    \end{align*}
    Now we can divide both sides of $\partial_xF=0$ by the common term $-4Ce^{-2C(x+\Delta/2)^2}$ and use
    \[
    \frac{e^{-2C(x-\Delta/2)^2}}{e^{-2C(x+\Delta/2)^2}}=e^{4C\Delta x},\qquad \frac{e^{-C(2x^2+\Delta^2/2)}}{e^{-2C(x+\Delta/2)^2}}=e^{2C\Delta x},
    \]
    which gives the first degeneracy equation
    \be\label{eqn:F_crit}
    \left(x+\frac{\Delta}{2}\right)+a^2\left(x-\frac{\Delta}{2}\right)e^{4C\Delta x}+2acxe^{2C\Delta x}=0.
    \ee
    Now, we compute $\partial_{xx}F$,
    \begin{align*}
    \partial_{xx}F(t,x)=&\left(-4C+16C^2\left(x+\frac{\Delta}{2}\right)^2\right)e^{-2C(x+\Delta/2)^2}\\
    &+a^2\left(-4C+16C^2\left(x-\frac{\Delta}{2}\right)^2\right)e^{-2C(x-\Delta/2)^2}\\
    &+2ac\left(-4C+16C^2x^2\right)e^{-C(2x^2+\Delta^2/2)}.
    \end{align*}
    Dividing $\partial_{xx}F=0$ by the common term $4Ce^{-2C(x+\Delta/2)^2}$ and using the same ratios as above gives
    \begin{align}\label{eqn:F_crit2}
    0&=\left(-1+4C\left(x+\frac{\Delta}{2}\right)^2\right)+a^2\left(-1+4C\left(x-\frac{\Delta}{2}\right)^2\right)e^{4C\Delta x}\\
    &\quad+2ac(-1+4Cx^2)e^{2C\Delta x}.\notag
    \end{align}
    Now from \eqref{eqn:F_crit},
    \be\label{eqn:F_c_critical}
    2ace^{2C\Delta x}=-\frac{\left(x+\frac{\Delta}{2}\right)+a^2\left(x-\frac{\Delta}{2}\right)e^{4C\Delta x}}{x}.
    \ee
    Inserting \eqref{eqn:F_c_critical} into \eqref{eqn:F_crit2}, multiplying by $x$, and expanding,
    \begin{align*}
    0=&x\left(-1+4C\left(x+\frac{\Delta}{2}\right)^2\right)+a^2x\left(-1+4C\left(x-\frac{\Delta}{2}\right)^2\right)e^{4C\Delta x}\\
    &-\left(-1+4Cx^2\right)\left[\left(x+\frac{\Delta}{2}\right)+a^2\left(x-\frac{\Delta}{2}\right)e^{4C\Delta x}\right]\\
    &=\left[x\left(-1+4C\left(x+\frac{\Delta}{2}\right)^2\right)-\left(-1+4Cx^2\right)\left(x+\frac{\Delta}{2}\right)\right]\\
    &+a^2e^{4C\Delta x}\left[x\left(-1+4C\left(x-\frac{\Delta}{2}\right)^2\right)-\left(-1+4Cx^2\right)\left(x-\frac{\Delta}{2}\right)\right].
    \end{align*}
    Now we can expand both brackets from the last line explicitly. For the first bracket,
    \begin{align*}
    x&\left(-1+4C\left(x^2+\Delta x+\frac{\Delta^2}{4}\right)\right)-\left(-1+4Cx^2\right)\left(x+\frac{\Delta}{2}\right)\\
    &=\left(-x+4Cx^3+4C\Delta x^2+C\Delta^2x\right)+\left(x+\frac{\Delta}{2}-4Cx^3-2C\Delta x^2\right)\\
    &=\frac{\Delta}{2}+2C\Delta x^2+C\Delta^2x.
    \end{align*}
    For the second bracket, with $v=x-\Delta/2$,
    \begin{align*}
    x&\left(-1+4C\left(x^2-\Delta x+\frac{\Delta^2}{4}\right)\right)-\left(-1+4Cx^2\right)v\\
    &=\left(-x+4Cx^3-4C\Delta x^2+C\Delta^2x\right)+\left(x-\frac{\Delta}{2}-4Cx^3+2C\Delta x^2\right)\\
    &=-\frac{\Delta}{2}-2C\Delta x^2+C\Delta^2x.
    \end{align*}
    Therefore the condition becomes
    \begin{align}
    0&=\left(\frac{\Delta}{2}+2C\Delta x^2+C\Delta^2x\right)+a^2e^{4C\Delta x}\left(-\frac{\Delta}{2}-2C\Delta x^2+C\Delta^2x\right)\notag\\
    &=\frac{\Delta}{2}\left[1+4Cx^2+2C\Delta x+a^2e^{4C\Delta x}\left(-1+2C\Delta x-4Cx^2\right)\right].\label{eqn:deg_step}
    \end{align}
    Since $\Delta>0$, \eqref{eqn:deg_step} is equivalent to
    \be\label{eqn:deg_step_simp}
    a^2e^{4C\Delta x}\left(2C\Delta x-4Cx^2-1\right)+\left(2C\Delta x+4Cx^2+1\right)=0.
    \ee
    Equations \eqref{eqn:F_crit} and \eqref{eqn:deg_step_simp} characterize all $(x,c)$ at which the there can be a change in the number of frequency axis maxima. Moreover, from \eqref{eqn:F_c_critical} we can write
    \be\label{eqn:c}
    c=-\frac{\left(x+\frac{\Delta}{2}\right)+a^2\left(x-\frac{\Delta}{2}\right)e^{4C\Delta x}}{2axe^{2C\Delta x}}.
    \ee
    We can now relate \eqref{eqn:deg_step_simp} and \eqref{eqn:c} to the constructive time threshold in Lemma \ref{lem:stft_const} and prove that no degeneracy (or bifurcation) with $c\in[-1,1]$ can occur when $\Delta>\Delta_{\text{critical, STFT}}$. First denote
    \[
    s\defeq ae^{2C\Delta x}>0,\qquad L\defeq\ln\frac{s}{a}=2C\Delta x,
    \]
    so that $e^{4C\Delta x}=(s/a)^2$ and $x=L/(2C\Delta)$. In these variables, \eqref{eqn:deg_step_simp} becomes, after multiplying both sides by $C\Delta^2$,
    \begin{align}
    0&=s^2\left(L-\frac{L^2}{C\Delta^2}-1\right)+\left(L+\frac{L^2}{C\Delta^2}+1\right)\notag\\
    &=\left(1-s^2\right)\left(L^2+C\Delta^2\right)+\left(s^2+1\right)C\Delta^2L.\label{eqn:deg_simplified}
    \end{align}
    Furthermore, \eqref{eqn:c} becomes
    \be\label{eqn:c_simp}
    c=-\frac{\left(1+s^2\right)L+\left(1-s^2\right)C\Delta^2}{2sL}.
    \ee
    At the constructive times ($c=1$), this reduces exactly to the pair of equations solved in Lemma \ref{lem:stft_const}. We now show that as $\Delta$ increases past the constructive threshold, any local branch of solutions of the degeneracy condition \eqref{eqn:deg_simplified} necessarily has $c>1$, so it cannot intersect the true range $c\in[-1,1]$. Consider solutions to \eqref{eqn:deg_simplified} in the variables $(L,\Delta)$ with parameter $s>0$ (recall $s=a e^L$). By the implicit function theorem, near any solution point $(L,\Delta)$ with $L\neq 0$ the set $\{(L,\Delta):(L,\Delta)\text{ solves }\eqref{eqn:deg_simplified}\}$ is a smooth curve. Along this curve note that $c$ given by \eqref{eqn:c_simp} is a smooth function of $\Delta$. Along the degeneracy curve we have $L=L(\Delta)$ ($s$ varies with $L$). Taking the total derivative of \eqref{eqn:deg_simplified}
    \begin{align}
    0&=\frac{d}{d\Delta}\left[\left(1-s^2\right)\left(L^2+C\Delta^2\right)+\left(s^2+1\right)C\Delta^2L\right]\notag\\
    &=\partial_L\left[\left(1-s^2\right)\left(L^2+C\Delta^2\right)+\left(s^2+1\right)C\Delta^2L\right]\frac{dL}{d\Delta}\notag\\
    &\;\;\;\;+\partial_\Delta\left[\left(1-s^2\right)\left(L^2+C\Delta^2\right)+\left(s^2+1\right)C\Delta^2L\right],\label{eqn:total_deriv}
    \end{align}
    where the partial derivatives are
    \begin{align*}
    \partial_L(\cdots)
    &=-2s^2(L^2+C\Delta^2)+2L(1-s^2)+2s^2 C\Delta^2 L+(s^2+1)C\Delta^2,\\
    \partial_\Delta(\cdots)
    &=2C\Delta\big[(1-s^2)+(s^2+1)L\big].
    \end{align*}
    Combined, this gives
    \begin{align}
    \left[-2s^2(L^2+C\Delta^2)+2L(1-s^2)+2s^2C\Delta^2L+(s^2+1)C\Delta^2\right]\frac{dL}{d\Delta}\notag\\
    +2C\Delta\left[(1-s^2)+(s^2+1)L\right]&=0.\label{eqn:degen_curve}
    \end{align}
    Next differentiate \eqref{eqn:c_simp} along the degeneracy curve (so $L=L(\Delta)$ satisfies \eqref{eqn:degen_curve}). 
    Writing $A(L,\Delta)=(1+s^2)L+(1-s^2)C\Delta^2$ and $B(L)=2sL$, we have $c=-A/B$. 
    Differentiating the identity
    \[
    A(L,\Delta)+c(L,\Delta)B(L)=0
    \]
    along the degeneracy curve $L=L(\Delta)$ gives
    \begin{equation}
    0=\frac{d}{d\Delta}\left(A+cB\right)=A_\Delta+A_L\frac{dL}{d\Delta}+B\frac{dc}{d\Delta}+cB_L\frac{dL}{d\Delta}.\label{eqn:ABc_diff}
    \end{equation}
    Solving \eqref{eqn:ABc_diff} for the total derivative $\frac{dc}{d\Delta}$ (not to be notationally confused with the derivative for fixed $t$),
    \begin{equation}
    \frac{dc}{d\Delta}=-\frac{A_\Delta+\left(A_L+cB_L\right)\frac{dL}{d\Delta}}{B}.\label{eqn:c_delta_diff}
    \end{equation}
    We now evaluate \eqref{eqn:c_delta_diff} at the constructive threshold from Lemma \ref{lem:stft_const}, where
    \[
    c=1,\qquad L=\frac12\left(s-\frac{1}{s}\right),\qquad C\Delta^2=\frac{(1+s)^2}{2s}.
    \]
    Computing the derivatives ($\frac{ds}{dL}=s$ since $s=a e^{L}$),
    \[
    A_\Delta=2C\Delta(1-s^2),\qquad
    A_L=(1+s^2)+2s^2L-2s^2C\Delta^2,\qquad
    B_L=2s(1+L),
    \]
    and substituting $L=\frac12(s-\frac1s)$ and $C\Delta^2=\frac{(1+s)^2}{2s}$ into $A_L,B,B_L$:
    \[
    A_L=-(s^2+2s-1),\qquad B=s^2-1,\qquad B_L=s^2+2s-1.
    \]
    Hence at the threshold we have the cancellation
    \[
    A_L+cB_L=A_L+B_L=0.
    \]
    Therefore the total derivative \eqref{eqn:c_delta_diff} reduces to
    \[
    \frac{dc}{d\Delta}=-\frac{A_\Delta}{B}
    =-\frac{2C\Delta(1-s^2)}{s^2-1}=2C\Delta>0.
    \]
    Thus, along the degeneracy branch passing through $(c,\Delta)=(1,\Delta_{\text{critical, STFT}})$, we have $c(\Delta)>1$ for all $\Delta>\Delta_{\text{critical, STFT}}(a)$ sufficiently close to the threshold. Since $c=\cos(2\pi\Delta t)\in[-1,1]$, no degeneracy can occur there. By continuity of the solution set of \eqref{eqn:F_crit} and \eqref{eqn:F_crit2}, any potential degeneracy at some larger $\Delta$ with $c\in[-1,1]$ would force another constructive-time degeneracy at an intermediate $\Delta$, contradicting Lemma \ref{lem:stft_const}.
\end{proof}

\begin{proof}[Proof of Lemma \ref{lem:stft_dest}]
    Set $C=\pi^2\sigma^2$ and
    \[
    \gamma(\eta)=e^{-C(\eta-\xi_0)^2}-ae^{-C(\eta-\xi_1)^2}.
    \]
    Then $|V_f^{(h)}(t_k^{-},\eta)|=|\gamma(\eta)|$ and $(|\gamma|^2)'=2\gamma\gamma'$, so stationary points of $|V_f^{(h)}(t_k^{-},\eta)|$ are the real solutions of $\gamma=0$ or $\gamma'=0$. First, $\gamma(\eta)=0$ means
    \begin{align}
    e^{-C(\eta-\xi_0)^2}&=ae^{-C(\eta-\xi_1)^2}\notag\\
    (\eta-\xi_0)^2-(\eta-\xi_1)^2&=\frac{\ln a}{C}\nonumber\\
    2\Delta(\eta-\xi_0)-\Delta^2&=\frac{\ln a}{C},\label{eqn:gamma_zeros}
    \end{align}
    which gives the unique solution
    \[
    \eta=\eta_{\text{AVG}}=\bar{\xi}-\frac{\ln a}{2C\Delta}.
    \]
    At this point $|\gamma|$ attains the strict global minimum. Differentiating $\gamma$,
    \[
    \gamma'(\eta)=-2C(\eta-\xi_0)e^{-C(\eta-\xi_0)^2}+2Ca(\eta-\xi_1)e^{-C(\eta-\xi_1)^2},
    \]
    and using \eqref{eqn:gamma_zeros} we have
    \[
    \gamma'(\eta_{\text{AVG}})=2C\big[(\eta_{\text{AVG}}-\xi_1)-(\eta_{\text{AVG}}-\xi_0)\big]e^{-C(\eta_{\text{AVG}}-\xi_0)^2}
    =-2C\Delta e^{-C(\eta_{\text{AVG}}-\xi_0)^2}<0,
    \]
    so $\gamma>0$ on $(-\infty,\eta_{\text{AVG}})$ and $\gamma<0$ on $(\eta_{\text{AVG}},\infty)$. Now, writing $x=\eta-\xi_0$. The condition $\gamma'(\eta)=0$ is equivalent to
    \begin{align}
    -xe^{-Cx^2}+a(x-\Delta)e^{-C(x-\Delta)^2}&=0\nonumber\\
    x&=a(x-\Delta)e^{2C\Delta x-C\Delta^2}.\label{eqn:crit_gamma_x}
    \end{align}
    Because $a>0$, the two sides of \eqref{eqn:crit_gamma_x} have the same sign, hence any solution satisfies either $x<0$ or $x>\Delta$ and there are no solutions with $x\in(0,\Delta)$. For $x>\Delta$, we can take logarithms of \eqref{eqn:crit_gamma_x} to obtain
    \begin{align}
    \ln x-\ln(x-\Delta)-2C\Delta x-\ln a+C\Delta^2&=0.\label{eqn:log_gamma_crit_x}
    \end{align}
    Define $h(x)=\ln x-\ln(x-\Delta)-2C\Delta x-\ln a+C\Delta^2$ for $x\in(\Delta,\infty)$. Then
    \begin{align*}
        h'(x)&=\frac{1}{x}-\frac{1}{x-\Delta}-2C\Delta\\
        &=-\frac{\Delta}{x(x-\Delta)}-2C\Delta\\
        &<0,
    \end{align*}
    so $h$ is strictly decreasing, with $\lim_{x\to\Delta^-}h(x)=\infty$ and $\lim_{x\to\infty}h(x)=-\infty$. Thus \eqref{eqn:log_gamma_crit_x} has a unique solution $x_+>\Delta$, hence a unique $\eta_+=\xi_0+x_+>\xi_1$ with $\gamma'(\eta_+)=0$.
    For $x<0$, \eqref{eqn:crit_gamma_x} gives
    \begin{align}
    -x&=a(\Delta-x)e^{2C\Delta x-C\Delta^2}\nonumber\\
    \ln(-x)-\ln(\Delta-x)-2C\Delta x-\ln a+C\Delta^2&=0.\label{eqn:log_gamma_crit_x2}
    \end{align}
    Define $H(x)=\ln(-x)-\ln(\Delta-x)-2C\Delta x-\ln a+C\Delta^2$ for $x\in(-\infty,0)$. Then
    \[
    H''(x)=-\frac{1}{x^2}-\frac{1}{(\Delta-x)^2}<0,
    \]
    so $H$ is strictly concave. Moreover $\lim_{x\to-\infty}H(x)=\infty$ and $\lim_{x\to 0^-}H(x)=-\infty$ so \eqref{eqn:log_gamma_crit_x2} has a unique solution $x_-<0$, giving a unique $\eta_-=\xi_0+x_-<\xi_0$ with $\gamma'(\eta_-)=0$. To classify $\eta_{\pm}$ as maxima, notice that
    \[
    \gamma'(\eta)=2Ce^{-Cx^2}\left(a(x-\Delta)e^{2C\Delta x-C\Delta^2}-x\right),
    \]
    so $\text{sign}\left(\gamma'(\eta)\right)=\text{sign}\left(a(x-\Delta)e^{2C\Delta x-C\Delta^2}-x\right)$. For $x>\Delta$, set
    \[
    r(x)=\ln\left(\frac{a(x-\Delta)e^{2C\Delta x-C\Delta^2}}{x}\right)=-h(x),
    \]
    which is strictly increasing on $(\Delta,\infty)$ with $r(x_+)=0$. So, $x<x_+\implies r(x)<0\implies \gamma'(\eta)<0$ and $x>x_+\implies r(x)>0\implies \gamma'(\eta)>0$. Since $\gamma(\eta)<0$ on $(\eta_{\text{AVG}},\infty)$, it follows that $(|\gamma|^2)'=2\gamma\gamma'$ changes from positive to negative at $\eta_+$, so $\eta_+$ is a strict local maximum of $|\gamma|$. For $x<0$, consider the same sign expression. As $x\to-\infty$, $e^{2C\Delta x}\to 0$ so $a(x-\Delta)e^{2C\Delta x-C\Delta^2}-x\to -x>0$ and as $x\to 0^-$,
    \[
    a(x-\Delta)e^{2C\Delta x-C\Delta^2}-x\to -a\Delta e^{-C\Delta^2}<0.
    \]
    By continuity and uniqueness of the zero $x_-$, $\gamma'(\eta)$ changes from positive to negative at $\eta_-$. Since $\gamma(\eta)>0$ on $(-\infty,\eta_{\text{AVG}})$, we get that $(|\gamma|^2)'$ changes from positive to negative at $\eta_-$, so $\eta_-$ is a strict local maximum of $|\gamma|$. Also, $\gamma(\eta)$ has exactly one zero (at $\eta_{\text{AVG}}$) and $\gamma'$ has exactly two zeros (at $\eta_-$ and $\eta_+$), so $|\gamma|$ has exactly three stationary points, the global minimum at $\eta_{\text{AVG}}$ and the two strict local maxima at $\eta_\pm$. For the left maximum set $x=-y$ with $y>0$:
    \begin{align*}
    y&=a(\Delta+y)e^{-C(\Delta+y)^2+Cy^2}=a(\Delta+y)e^{-C\Delta^2-2C\Delta y}\le a(\Delta+y)e^{-C\Delta^2}.
    \end{align*}
    If $a e^{-C\Delta^2}<1$ this yields
    \[
    y\leq \frac{ae^{-C\Delta^2}}{1+ae^{-C\Delta^2}} \Delta,
    \]
    and the same inequality is trivial when $a e^{-C\Delta^2}\ge 1$ because the right hand side is then at least $\Delta/2$. For the right maximum write $x=\Delta+z$ with $z>0$ so
    \begin{align*}
    \Delta+z&=a ze^{-C(\Delta+z)^2+C z^2}=a ze^{-C\Delta^2-2C\Delta z}\le a ze^{-C\Delta^2},
    \end{align*}
    which rearranges to the stated result. 
\end{proof}

\begin{proof}[Proof of AHM local error bound for STFT]
    We can first write
    \be\label{eqn:integral_ahm_approx}
    V_F(t,\eta)-a_0V_f(t-t_*,\eta) = \int_{-\infty}^\infty \left(F(x)-a_0f(x-t_*)\right) h(x-t) e^{-2\pi i \eta(x-t)} dx.
    \ee
    To estimate the error in this approximation, we can apply the given AHM bounds. Using $|A_j'(u)\leq \varepsilon|\phi_j'(u)|$ and Estimate 3.4 (applied to $\phi_j'$ at $t_*$) from \cite{sst},
    \[
    |\phi_j'(t_*+s)-\phi_j'(t_*)|\leq \varepsilon\left(|\phi_j'(t_*)||s|+\frac{1}{2}M''|s|^2\right),
    \]
    so for any $u$ between $t_*$ and $x$, we have (by the triangle inequality)
    \be\label{eqn:taylor_bound_phi}
    |\phi_j'(u)|\leq |\phi_j'(t_*)|+\varepsilon\left(|\phi_j'(t_*)||u-t_*|+\frac{1}{2}M''|u-t_*|^2\right).
    \ee
    Using this gives
    \begin{align*}
        |A_j(x)-A_j(t_*)| &= \left|\int_{t_*}^x A_j'(u) du\right|\\
        &\leq \int_{t_*}^x |A_j'(u)| du\\
        &\leq \varepsilon\int_{t_*}^x |\phi_j'(u)| du\\
        &\leq \varepsilon\int_{0}^{|x-t_*|} \left(|\phi_j'(t_*)|+M'' v\right) dv\\
        &=\varepsilon\left(|\phi_j'(t_*)||x-t_*|+\frac{1}{2}M''|x-t_*|^2\right)
    \end{align*}
    for $j=0,1$. We can also write
    \[
    \phi_j(x) = \phi_j(t_*)+\xi_j(x-t_*) + r_j(x), \qquad r_j(x)\defeq \int_{t_*}^x (x-u) \phi_j''(u) du.
    \]
    Using $|\phi_j''(u)|\leq \varepsilon|\phi_j'(u)|$ and \eqref{eqn:taylor_bound_phi},
    \begin{align*}
        |r_j(x)|&\leq \varepsilon\int_0^{|x-t_*|} \left(|x-t_*|-v\right)\left(|\phi_j'(t_*)|+M'' v\right) dv\\
        &=\varepsilon\left(\frac{|\xi_j|}{2}|x-t_*|^2+\frac{M''}{6}|x-t_*|^3\right),
    \end{align*}
    so,
    \[
    e^{2\pi i \phi_j(x)} = e^{2\pi i\left(\phi_j(t_*)+\xi_j(x-t_*)\right)}e^{2\pi i r_j(x)}=e^{2\pi i \xi_j (x-t_*)}\left(1+\delta_j(x)\right)
    \]
    with
    \[
    |\delta_j(x)|=|e^{2\pi i r_j(x)} - 1|\leq 2\pi |r_j(x)|\leq \varepsilon\left(\pi|\xi_j||x-t_*|^2+\frac{\pi M''}{3}|x-t_*|^3\right).
    \]
    Now, since $\phi_j(t_*)=0$, we can write
    \begin{align*}
        F(x)&= A_0(x)e^{2\pi i \xi_0(x-t_*)}\left(1+\delta_0(x)\right)+A_1(x)e^{2\pi i \xi_1(x-t_*)}\left(1+\delta_1(x)\right)\\
        a_0 f(x-t_*)&= a_0e^{2\pi i \xi_0(x-t_*)}+a_1e^{2\pi i \xi_1(x-t_*)},
    \end{align*}
    and so
    \begin{align*}
        F(x)-a_0(x-t_*)&= \underbrace{\left(A_0(x)-a_0\right)e^{2\pi i \xi_0(x-t_*)}+ \left(A_1(x)-a_1\right)e^{2\pi i \xi_1(x-t_*)}}_{\text{amplitude error}}\\
        &\quad+\underbrace{A_0(x)e^{2\pi i \xi_0(x-t_*)}\delta_0(x)+A_1(x)e^{2\pi i \xi_1(x-t_*)}\delta_1(x)}_{\text{phase error}}.
    \end{align*}
    Taking absolute values and applying the above bounds gives
    \begin{align*}
        |F(x)-a_0f(x-t_*)|&\leq \varepsilon\left(|\xi_0||x-t_*|+\frac{M''}{2}|x-t_*|^2\right)\\
        &\quad +\varepsilon a\left(|\xi_1||x-t_*|+\frac{M''}{2}|x-t_*|^2\right)\\
        &\quad + \varepsilon a_0\left(\pi|\xi_0||x-t_*|^2+\frac{\pi M''}{3}|x-t_*|^3\right)\\
        &\quad +\varepsilon a_1\left(\pi|\xi_1||x-t_*|^2+\frac{\pi M''}{3}|x-t_*|^3\right).
    \end{align*}
    Now, we can complete the proof by inserting the above bound into \eqref{eqn:integral_ahm_approx}. Specifically,
    \begin{align*}
        \left|V_F(t,\eta)-a_0V_f(t-t_*,\eta)\right| &= \left|\int_{-\infty}^\infty \left(F(x)-a_0f(x-t_*)\right) h(x-t) e^{-2\pi i \eta(x-t)} dx\right|\\
        &\leq \int_{-\infty}^\infty \left|F(x)-a_0f(x-t_*)\right| |h(x-t)| dx\\
        &\leq \varepsilon\int_{-\infty}^\infty\Bigg(|\xi_0||x-t_*|+\frac{M''}{2}|x-t_*|^2\\
        &\qquad\qquad\quad + a\left(|\xi_1||x-t_*|+\frac{M''}{2}|x-t_*|^2\right)\\
        &\qquad\qquad\quad +  a_0\left(\pi|\xi_0||x-t_*|^2+\frac{\pi M''}{3}|x-t_*|^3\right)\\
        &\qquad\qquad\quad + a_1\left(\pi|\xi_1||x-t_*|^2+\frac{\pi M''}{3}|x-t_*|^3\right)\Bigg)dx.
    \end{align*}
    To compute the above integral, we can notice that all terms are of the form $\int h(x-t) |x-t_*|^m$ for $m=1,2,3$. Using the Gaussian moments
    \[
    \int_{-\infty}^\infty h(x-t) dx = 1, \quad \int_{-\infty}^\infty |x-t| h(x-t) dx =\frac{\sigma}{\sqrt{\pi}}, \quad \int |x-t|^2 h(x-t) dx = \frac{\sigma^2}{2},
    \]
    \[
    \int_{-\infty}^\infty |x-t|^3 h(x-t) dx = \frac{\sigma^3}{\sqrt{\pi}}
    \]
    and the binomial bound $|a+b|^3\leq |a|^3+3a^2|b|+3|a|b^2+|b|^3$ with $a=t-t_*$ and $b=x-t$, we have
    \begin{align*}
        \int_{-\infty}^\infty |x-t_*| h(x-t) dx &\leq |t-t_*|+\frac{\sigma}{\sqrt{\pi}},\\
        \int_{-\infty}^\infty |x-t_*|^2 h(x-t) dx &\leq (t-t_*)^2+\frac{2\sigma}{\sqrt{\pi}}|t-t_*|+\frac{\sigma^2}{2},\\
        \int_{-\infty}^\infty |x-t_*|^3 h(x-t) dx &\leq |t-t_*|^3+\frac{3\sigma}{\sqrt{\pi}}|t-t_*|^2+\frac{3\sigma^2}{2}|t-t_*|+\frac{\sigma^2}{\sqrt{\pi}}.
    \end{align*}
    Collecting the terms and substituting yields the stated result.
\end{proof}

\begin{proof}[Proof of Proposition \ref{prop:winding}]
    Define the linear change of variables $z=t/\sigma -i\pi\sigma t\eta$. Recall \eqref{stft vs bt} and set
    \[
    G(z)\defeq e^{\frac{1}{2}\left[(t/\sigma)^2+(\pi\sigma\eta)^2\right]}e^{i\pi t\eta} V_f^{(h)}(t,\eta)
    \]
    which is holomorphic in $z$. Write $z_0=t_0/\sigma-i\pi\sigma\eta_0$ and choose $r>0$ so that the circle $\mathcal{C}_\rho\defeq \{z:|z-z_0|=\rho\}$ contains no zeros of $G$ on its boundary and encloses exactly the zero $z_0$ for all $0<\rho<r$. The ellipse $\mathcal{E}_\rho$ is the preimage of $\mathcal{C}_\rho$ under $(t,\eta)\mapsto z$. By the argument principle, we have 
    \begin{align*}
        \frac{1}{2\pi i}\oint_{\mathcal{C}_\rho} \frac{G'(z)}{G(z)}dz
        = m\,,
    \end{align*}
    where $m$ is the multiplicity of the zero of $G$ at $z_0$ (equivalently, of $V_f^{(h)}$ at $(t_0,\eta_0)$). Lastly, it remains to show that $m=1$. Writing $A(t,\eta)\defeq e^{2\pi i\xi_0 t}e^{-\pi^2\sigma^2(\eta-\xi_0)^2}$ and $B(t,\eta)\defeq ae^{2\pi i\xi_1 t}e^{-\pi^2\sigma^2(\eta-\xi_1)^2}$ so $V_f^{(h)}=A+B$ and at $(t_0,\eta_0)$ we have $B=-A$. Differentiating, 
    \[
    \partial_t V_f^{(h)}(t_0,\eta_0)=2\pi i(\xi_0 A+\xi_1 B)|_{(t_0,\eta_0)}=-i2\pi\Delta A(t_0,\eta_0)
    \]
    and
    \[
    \partial_\eta V_f^{(h)}(t_0,\eta_0)=-2\pi^2\sigma^2\left((\eta_0-\xi_0)A+(\eta_0-\xi_1)B\right)|_{(t_0,\eta_0)}=-2\pi^2\sigma^2\Delta A(t_0,\eta_0).
    \]
    With
    \[
    \partial_z = \frac{1}{2}\left(\sigma \partial_t + \frac{i}{\pi\sigma}\partial_\eta\right),
    \]
    we have
    \begin{align*}
      \partial_z G(z)|_{z=z_0} &= \left[\frac{1}{2}\left(\sigma \partial_t + \frac{i}{\pi\sigma}\partial_\eta\right)\left(e^{\frac{1}{2}\left[(t/\sigma)^2+(\pi\sigma\eta)^2\right]}e^{i\pi t\eta} V_f^{(h)}(t,\eta)\right)\right]_{t=t_0,\eta=\eta_0}\\
      &= \frac{1}{2}e^{\frac{1}{2}[(t_0/\sigma)^2+(\pi\sigma\eta_0)^2]}e^{i\pi t_0\eta_0}\\
      &\qquad\times\left(\sigma(-2\pi i\Delta A(t_0,\eta_0))+\frac{i}{\pi\sigma}(-2\pi^2\sigma^2\Delta A(t_0,\eta_0))\right)\\
      &= -2\pi i\sigma\Delta e^{\frac{1}{2}[(t_0/\sigma)^2+(\pi\sigma\eta_0)^2]}e^{i\pi t_0\eta_0}A(t_0,\eta_0)\\
      &\neq0,
    \end{align*}
    so the zero is simple ($m=1$) and the total phase change is $2\pi$ giving a phase winding number of $1$ around the zeros.
\end{proof}

\begin{proof}[Proof of Proposition \ref{prop:AHM_omega}]
    Multiplying a signal by a nonzero constant does not change the reassignment rule, since both $V^{(h)}$ and $V^{(Dh)}$ are scaled by the same factor. In particular, replacing $F$ by $a_0^{-1}F$ and $f$ by $a_0^{-1}f$ leaves the ratio $V^{(Dh)}/V^{(h)}$ invariant, so ${\hat{\eta}}_s^F$ and ${\hat{\eta}}_s^f$ are unchanged. We therefore keep $a_0$ explicit but note that it can be set to $0$ without loss of generality. From Section \ref{Section STFT approximation of AHM}, evaluated on $|t-t_*|\leq T$, we have
    \[
    \left|V_F^{(h)}(t,\eta)-a_0 V_f^{(h)}(t-t_*,\eta)\right|\leq C_h \varepsilon
    \]
    for all $|t-t_*|\leq T$ and all $\eta$. Applying the same argument with the window $Dh$ in place of $h$ similarly gives
    \[
    \left| V_F^{(Dh)}(t,\eta)-a_0 V_f^{(Dh)}(t-t_*,\eta)\right| \leq C_{Dh} \varepsilon
    \]
    for all $|t-t_*|\leq T$ and all $\eta$. We can set 
    \[
    B(t,\eta)\defeq V_f^{(h)}(t-t_*,\eta),\qquad
    A(t,\eta)\defeq V_f^{(Dh)}(t-t_*,\eta),
    \]
    and the errors
    \[
    \delta B(t,\eta)\defeq V_F^{(h)}(t,\eta)-a_0 B(t,\eta),\qquad
    \delta A(t,\eta)\defeq V_F^{(Dh)}(t,\eta)-a_0 A(t,\eta).
    \]
    Then on $|t-t_*|\leq T$, we have 
    \[
    |\delta B(t,\eta)|\leq C_h \varepsilon,\qquad
    |\delta A(t,\eta)|\leq C_{Dh} \varepsilon.
    \]
    For the clean harmonic model $f$, the STFT with window $Dh$ can be written explicitly and we can get
    \[
    A(t,\eta)=V_f^{(Dh)}(t-t_*,\eta) = e^{2\pi i \xi_0 (t-t_*)}\widehat{Dh}(\eta-\xi_0) + a e^{2\pi i \xi_1 (t-t_*)}\widehat{Dh}(\eta-\xi_1),
    \]
    and therefore,
    \[
    |A(t,\eta)|\leq (1+|a|)\|\widehat{Dh}\|_\infty.
    \]
    For Gaussian $h$, that is $\widehat{h}(\eta)=e^{-\pi^2 \sigma^2 \eta^2}$, and $\widehat{Dh}(\eta)=2\pi i \eta e^{-\pi^2 \sigma^2 \eta^2}$, so
    \[
    \|\widehat{Dh}\|_\infty =\sup_{\eta\in\mathbb{R}} 2\pi |\eta| e^{-\pi^2 \sigma^2 \eta^2} =\frac{\sqrt{2}}{\sigma} e^{-1/2},
    \]
    and thus, we can use 
    \[
    \sup_{t,\eta} |A(t,\eta)| \leq (1+|a|)\frac{\sqrt{2}}{\sigma} e^{-1/2}.
    \]
    Now, fixing $0<\beta\leq \frac{1}{2}$ and only considering $t,\eta$ such that $|t-t_*|\leq T$ and where
    \[
    |B(t,\eta)|=\left|V_f^{(h)}(t-t_*,\eta)\right|\geq \varepsilon^\beta,
    \]
    assuming that
    \[
    0<\varepsilon\leq \varepsilon_0\defeq\min\left\{1,\left(\frac{|a_0|}{2 C_h}\right)^{\frac{1}{1-\beta}}\right\},
    \]
    we have
    $\varepsilon^{1-\beta} \leq \varepsilon_0^{1-\beta} \leq \frac{|a_0|}{2 C_h}$ and multiplying by $C_h \varepsilon^\beta$ gives $C_h \varepsilon^{1-\beta}\varepsilon^\beta=C_h \varepsilon  \leq \frac{|a_0|}{2}\varepsilon^\beta$. Using the triangle inequality and the bound $|\delta B(t,\eta)|\leq C_h \varepsilon$, we get
    \begin{align*}
        |a_0 B(t,\eta)+\delta B(t,\eta)| &\geq |a_0||B(t,\eta)| - |\delta B(t,\eta)|\\
        &\geq |a_0|\varepsilon^\beta - C_h \varepsilon\\
        &\geq \frac{|a_0|}{2}\varepsilon^\beta.
    \end{align*}
    Combining this with $|B(t,\eta)|\geq \varepsilon^\beta$ gives
    \[
    |B(t,\eta)||a_0 B(t,\eta)+\delta B(t,\eta)| \geq \frac{|a_0|}{2}\varepsilon^{2\beta}.
    \]
    Now, the reassignment rules are
    \[
    {\hat{\eta}}_s^F(t,\eta)= \eta - \frac{1}{2\pi i}\frac{V_F^{(Dh)}(t,\eta)}{V_F^{(h)}(t,\eta)},\qquad{\hat{\eta}}_s^f(t-t_*,\eta) = \eta - \frac{1}{2\pi i}\frac{A(t,\eta)}{B(t,\eta)}.
    \]
    We can further define $A_F(t,\eta)\defeq V_F^{(Dh)}(t,\eta)=a_0 A(t,\eta)+\delta A(t,\eta)$ and $B_F(t,\eta)\defeq V_F^{(h)}(t,\eta)=a_0 B(t,\eta)+\delta B(t,\eta)$ so that
    \[
    {\hat{\eta}}_s^F(t,\eta)-{\hat{\eta}}_s^f(t-t_*,\eta) = -\frac{1}{2\pi i}R(t,\eta).
    \]
    where
    \begin{align*}
        R(t,\eta)&\defeq \frac{A_F(t,\eta)}{B_F(t,\eta)} - \frac{A(t,\eta)}{B(t,\eta)}\\
        &= \frac{A_F B - A B_F}{B B_F}\\
        &= \frac{(a_0 A + \delta A)B - A(a_0 B + \delta B)}{B(a_0 B + \delta B)}\\
        &= \frac{\delta A(t,\eta) B(t,\eta) - A(t,\eta)\delta B(t,\eta)}{B(t,\eta)\left(a_0 B(t,\eta)+\delta B(t,\eta)\right)}.
    \end{align*}
    For the numerator, the triangle inequality and the above bounds give
    \begin{align*}
        |\delta A(t,\eta) B(t,\eta) - A(t,\eta)\delta B(t,\eta)| &\leq |\delta A(t,\eta)|\,|B(t,\eta)| + |A(t,\eta)||\delta B(t,\eta)|\\
        &\leq C_{Dh}\varepsilon + C_h \varepsilon\sup_{t,\eta} |A(t,\eta)|\\
        &\leq C_{Dh}\varepsilon + C_h \varepsilon(1+|a|)\frac{\sqrt{2}}{\sigma} e^{-1/2}.
    \end{align*}
    Combining this with the above bound for the denominator, we have
    \begin{align*}
        |R(t,\eta)| &\leq \frac{C_{Dh}\varepsilon + C_h \varepsilon(1+|a|)\frac{\sqrt{2}}{\sigma} e^{-1/2}}{\frac{|a_0|}{2}\varepsilon^{2\beta}}\\
        &= \frac{2}{|a_0|}\left(C_{Dh} +  C_h (1+|a|)\frac{\sqrt{2}}{\sigma} e^{-1/2}\right)\varepsilon^{1-2\beta}.
    \end{align*}
    Finally, 
    \[
    \left|{\hat{\eta}}_s^F(t,\eta)-{\hat{\eta}}_s^f(t-t_*,\eta)\right| = \frac{1}{2\pi}|R(t,\eta)| \leq \frac{C_{Dh} + (1+|a|)\frac{\sqrt{2}}{\sigma} e^{-1/2} C_h}{\pi |a_0|}\varepsilon^{1-2\beta}
    \]
    as stated.
\end{proof}

\begin{proof}[Proof of Proposition \ref{prop:sst_large_Delta}]
    Applying the change of coordinate $x=\eta-\xi_0$ and denoting
    \[
    t(x)\defeq-\frac{1}{\alpha}\left(\xi_0-\xi+\Delta\frac{a}{a-\exp\left(\pi^2 \sigma^2 \Delta (\Delta-2x)-2\pi i\Delta t\right)}\right)^2,
    \]
    we can write
    \be\label{eqn:sst_tx}
    S_f^{(h)}(t,\xi) =\frac{1}{\sqrt{\pi\alpha}}e^{2\pi i\xi_0 t}\int_{-\infty}^{\infty} \left(e^{-\pi^2\sigma^2 x^2}+ae^{2\pi i\Delta t}e^{-\pi^2\sigma^2(x-\Delta)^2}\right) e^{t(x)}dx
    \ee
    To evaluate this integral in simpler components, we can write
    \begin{align*}
        \sqrt{\pi\alpha}e^{-2\pi i \xi_0 t}S_f^{(h)}(t,\xi)&= I_1+ae^{2\pi i\Delta t}I_2
    \end{align*}
    where $I_1=\int_{-\infty}^{\infty} e^{-\pi^2\sigma^2 x^2}e^{t(x)}dx$ and $I_2=\int_{-\infty}^{\infty} e^{-\pi^2\sigma^2(x-\Delta)^2} e^{t(x)}dx$.
    
    Focusing on $I_1$, for any $\epsilon>0$,
    \begin{align*}
        I_1&=\int_{-\infty}^{\infty} e^{-\pi^2\sigma^2 x^2}e^{t(x)}dx\\
        &=\underbrace{\int_{|x|<\epsilon} e^{-\pi^2\sigma^2 x^2}e^{t(x)}dx}_{\defeq I_{1A}}+\underbrace{\int_{|x|\geq \epsilon} e^{-\pi^2\sigma^2 x^2}e^{t(x)}dx}_{\defeq I_{1B}}.
    \end{align*}
    
    Notice that for $I_{1B}$, $e^{-\pi^2 \sigma^2 x^2}$ is exponentially small if $\frac{1}{\sigma}\lesssim\epsilon$. Furthermore, when the exponential $e^{-\pi^2\sigma^2 x^2}$ is small, we can use the trivial bound $e^{t(x)}\leq 1$. Assuming this is true (which we will soon make precise), we can focus on $I_{1A}$. We have,
    \begin{align*}
        I_{1A}&=\int_{|x|<\epsilon} e^{-\pi^2\sigma^2 x^2}e^{t(x)}dx.
    \end{align*}
    
    Defining $E\defeq \exp\left(\pi^2 \sigma^2 (\Delta^2-2\Delta x)\right)$ and $c\defeq\cos(2\pi\Delta t)$, and by expanding $t(x)$ and using that $|x|<\epsilon \ll \Delta$ we get, 
    \begin{align*}
        -\alpha t(x)&=(\xi_0-\xi)^2+\frac{2\Delta(\xi_0-\xi)\left[1+a^{-1}\exp\left(\pi^2 \sigma^2 (\Delta^2-2\Delta x)\right)\cos(2\pi\Delta t)\right]+\Delta^2}{1+2a^{-1}\exp\left(\pi^2 \sigma^2 (\Delta^2-2\Delta x)\right)\cos(2\pi\Delta t)+a^{-2}\exp\left(2\pi^2 \sigma^2 (\Delta^2-2\Delta x)\right)}\\
        &=(\xi_0-\xi)^2+\frac{2\Delta(\xi_0-\xi)\left[a^2+aEc\right]+a^2\Delta^2}{a^2+2aEc+E^2}\\
        &=(\xi_0-\xi)^2+\frac{2a^2\Delta(\xi_0-\xi)}{a^2+2aEc+E^2}
        +\frac{2a\Delta(\xi_0-\xi)Ec}{a^2+2aEc+E^2}
        +\frac{a^2\Delta^2}{a^2+2aEc+E^2}\\
        &=(\xi_0-\xi)^2+\left[2a^2\Delta(\xi_0-\xi)+2a\Delta(\xi_0-\xi)Ec+a^2\Delta^2\right]\frac{1}{E^2}\frac{1}{1+2acE^{-1}+a^2E^{-2}}\\
        &=(\xi_0-\xi)^2+\left[2a^2\Delta(\xi_0-\xi)+2a\Delta(\xi_0-\xi)Ec+a^2\Delta^2\right]E^{-2}\left[1-2acE^{-1}+O(E^{-2})\right]\\
        &=(\xi_0-\xi)^2+2a\Delta(\xi_0-\xi)cE^{-1}+O\left(\Delta E^{-2}\right)+O\left(\Delta^2E^{-2}\right)\\
        &=(\xi_0-\xi)^2+2a\Delta(\xi_0-\xi)\cos(2\pi\Delta t)\exp\left(-\pi^2\sigma^2(\Delta^2-2\Delta x)\right)\\
        &\qquad+O\left(\Delta\exp\left(-2\pi^2\sigma^2(\Delta^2-2\Delta x)\right)\right)+O\left(\Delta^2\exp\left(-2\pi^2\sigma^2(\Delta^2-2\Delta x)\right)\right)\\
        &=(\xi_0-\xi)^2+O\left(\exp\left(-\pi^2\sigma^2\Delta^2\right)\right).
    \end{align*}
    Using this, we have
    \begin{align*}
        \exp\left(t(x)\right)&= \exp\left(-\alpha^{-1}\left[(\xi_0-\xi)^2+O\left(\exp\left(-\pi^2\sigma^2\Delta^2\right)\right)\right]\right)\\
        &= \exp\left(-\alpha^{-1}(\xi_0-\xi)^2\right)\left[1+O\left(\exp\left(-\pi^2\sigma^2\Delta^2\right)\right)\right]\\
        &=\exp\left(-\alpha^{-1}(\xi_0-\xi)^2\right)+O\left(\exp\left(-\pi^2\sigma^2\Delta^2\right)\right)
    \end{align*}
    and putting this all together, we have 
    \begin{align*}
        I_{1A}&=\left[\exp\left(-\alpha^{-1}(\xi_0-\xi)^2\right)+O\left(\exp\left(-\pi^2\sigma^2\Delta^2\right)\right)\right] \int_{|x|<\epsilon} e^{-\pi^2\sigma^2 x^2}dx\\
        &=\frac{1}{\sqrt{\pi}\sigma} \erf(\pi\sigma\epsilon)\exp\left(-\alpha^{-1}(\xi_0-\xi)^2\right)+O\left(\exp\left(-\pi^2\sigma^2\Delta^2\right)\right).
    \end{align*}
    For $I_{1B}$, since $|x|>\epsilon$, the Gaussian integrand $e^{-\pi^2 \sigma^2 x^2}$ is exponentially small and so
    \begin{align*}
        I_{1B}&\lesssim \erfc(\pi\sigma\epsilon).
    \end{align*}  
    We of course need $0\ll\epsilon\leq\frac{\Delta}{2}$ for our approximations to be valid. Using this and an asymptotic expansion of $\erf(x)$ as $x\to\infty$, $\erf(x)=1-\erfc(x)\sim \frac{e^{-x^2}}{\sqrt{\pi}x}$, we have 
    \begin{align*}
        I_1&=I_{1A}+I_{1B}\\
        &= \frac{1}{\sqrt{\pi}\sigma} \erf(\pi\sigma\epsilon)\exp\left(-\alpha^{-1}(\xi_0-\xi)^2\right)\\
        &\qquad+O\left(\exp\left(-\pi^2\sigma^2\Delta^2\right)\right) + O\left(\exp\left(-\frac{\pi^2\sigma^2\Delta^2}{4}\right)\right)\\
        &=\frac{1}{\sqrt{\pi}\sigma} \erf(\pi\sigma\epsilon)\exp\left(-\alpha^{-1}(\xi_0-\xi)^2\right)+ O\left(\exp\left(-\frac{\pi^2\sigma^2\Delta^2}{4}\right)\right)
    \end{align*}
    if we pick $\epsilon=\frac{\Delta}{2}$. Now, we can look at
    \begin{align*}
        I_2&= \int_{-\infty}^{\infty} e^{-\pi^2\sigma^2(x-\Delta)^2} e^{t(x)}dx\\
        &= \int_{-\infty}^{\infty} e^{-\pi^2\sigma^2 x^2} e^{u(x)}dx
    \end{align*}
    where 
    \[
    u(x)=-\frac{1}{\alpha}\left(\xi_0-\xi+\Delta\frac{a}{a-\exp\left(-\pi^2 \sigma^2 \Delta (\Delta+2x)-2\pi i\Delta t\right)}\right)^2.
    \]
    This integral can be approximated in the same way as the $I_1$ integral with the small adjustment of the approximation for $u(x)$. In any case by expanding $u(x)$ we get a similar result as for the $I_1$ integral, 
    \begin{align*}
        -\alpha u(x)&=(\xi_1-\xi)^2+O\left(\exp\left(-\pi^2\sigma^2\Delta^2\right)\right).
    \end{align*}
    Following the same steps as before, we get
    \begin{align*}
        I_2&=\frac{1}{\sqrt{\pi}\sigma} \erf(\pi\sigma\epsilon)\exp\left(-\alpha^{-1}(\xi_1-\xi)^2\right)+ O\left(\exp\left(-\frac{\pi^2\sigma^2\Delta^2}{4}\right)\right)
    \end{align*}
    and we can combine this altogether to get
    \begin{align*}
        S_f^{(h)}(t,\xi)&=\frac{1}{\sqrt{\pi\alpha}} e^{2\pi i\xi_0 t}\left(I_1+ae^{2\pi i \Delta t}I_2\right)\\
        &=\frac{1}{\pi\sigma\sqrt{\alpha}} e^{2\pi i\xi_0 t}\left[e^{-\frac{1}{\alpha}(\xi_0-\xi)^2}+O\left(e^{-\pi^2\sigma^2 \frac{\Delta^2}{4}}\right)+ae^{2\pi i \Delta t}\left(e^{-\frac{1}{\alpha}(\xi_1-\xi)^2}+O\left( e^{-\pi^2\sigma^2 \frac{\Delta^2}{4}}\right)\right)\right]\\
        &=\frac{1}{\pi\sigma\sqrt{\alpha}} e^{2\pi i\xi_0 t}\left[e^{-\frac{1}{\alpha}(\xi_0-\xi)^2}+ae^{2\pi i \Delta t}\left(e^{-\frac{1}{\alpha}(\xi_1-\xi)^2}\right)\right]+O\left(e^{-\pi^2\sigma^2 \frac{\Delta^2}{4}}\right).
    \end{align*}
    Noticing that we can explicitly find $$S_{f_0}^{(h)}(t,\xi)=\frac{1}{\pi\sigma\sqrt{\alpha}} e^{2\pi i\xi_0 t}e^{-\frac{1}{\alpha}(\xi_0-\xi)^2}$$ and $$S_{f_1}^{(h)}(t,\xi)=\frac{a}{\pi\sigma\sqrt{\alpha}} e^{2\pi i(\xi_1) t}e^{-\frac{1}{\alpha}(\xi_1-\xi)^2}$$ finishes the proof.
\end{proof}

\begin{proof}[Proof of Proposition \ref{prop:sst_large_a}]
    Let us recall from \eqref{eqn:sst_tx} that we can write
    \be
    S_f^{(h)}(t,\xi) =\frac{1}{\sqrt{\pi\alpha}}e^{2\pi i\xi_0 t}\int_{-\infty}^{\infty} \left(e^{-\pi^2\sigma^2 x^2}+ae^{2\pi i\Delta t}e^{-\pi^2\sigma^2(x-\Delta)^2}\right) e^{t(x)}dx
    \ee
    with
    \[
    t(x)\defeq-\frac{1}{\alpha}\left(\xi_0-\xi+\Delta\frac{a}{a-\exp\left(\pi^2 \sigma^2 \Delta (\Delta-2x)-2\pi i\Delta t\right)}\right)^2.
    \]
    Our interest is in the behavior as $a\to 0$, so we can first expand $e^{t(x)}$. Denote
    \[
    C\defeq \frac{1}{D}\defeq e^{\pi^2 \sigma^2 \Delta(\Delta-2x)-2\pi i \Delta t}.
    \]
    Then,
    \begin{align*}
        \frac{a}{a+e^{\pi^2 \sigma^2 \Delta(\Delta-2x)-2\pi i \Delta t}} &= \frac{a}{a+C}\\
        &= \frac{a}{C}\left(\frac{1}{1+\frac{a}{C}}\right)\\
        &\approx \frac{a}{C}\left(1-\frac{a}{C}+\frac{a^2}{C^2}-\ldots\right)\\
        &=aD-a^2D^2+a^3D^3-\ldots
    \end{align*}
    which converges for $|\frac{a}{C}|<1$. So, we now have
    \begin{align*}
        e^{t(x)}&\approx e^{-\frac{1}{\alpha}\left|\xi-\xi_1\left(aD-O(a^2)\right)\right|^2}\\
        &=e^{-\frac{1}{\alpha}\left|\xi-\xi_1 a D+O(a^2)\right|^2}\\
        &=e^{-\frac{1}{\alpha}\left[(\xi-\xi_0)^2-2(\xi-\xi_0)\Delta a \Re(D)+O(a^2)\right]}\\
        &=e^{-\frac{1}{\alpha}(\xi-\xi_0)^2}e^{\frac{2}{\alpha}(\xi-\xi_0)\Delta a \Re(D)+O(a^2)}\\
        &\approx e^{-\frac{1}{\alpha}(\xi-\xi_0)^2}\left(1+\frac{2}{\alpha}(\xi-\xi_0)\Delta a \Re(D)+O(a^2)\right).
    \end{align*}
    Using this,
    \begin{align*}
        S_f^{(h)}(t,\xi) &= \frac{1}{\sqrt{\pi\alpha }}e^{2\pi i\xi_0t}\int_{-\infty}^\infty e^{-\pi^2\sigma^2 x^2} e^{-\frac{1}{\alpha}(\xi-\xi_0)^2} \left(1+\frac{2}{\alpha}(\xi-\xi_0)\Delta a \Re(D)+O(a^2)\right)  dx\\
        &\quad+\frac{a}{\sqrt{\pi\alpha}}e^{2\pi i (\xi_1)t}\int_{-\infty}^\infty e^{-\pi^2\sigma^2 (x-\Delta^2)} e^{-\frac{1}{\alpha}(\xi-\xi_0)^2} \left(1+\frac{2}{\alpha}(\xi-\xi_0)\Delta a \Re(D)+O(a^2)\right)  dx\\
        &=\frac{1}{\pi\sigma \sqrt{\alpha}} e^{2\pi i\xi_0 t} e^{-\frac{1}{\alpha}(\xi-\xi_0)^2} + \frac{2a\Delta(\xi-\xi_0) \cos(2\pi\Delta t)}{\pi\sigma \alpha^{3/2}} e^{2\pi i\xi_0 t} e^{-\frac{1}{\alpha}(\xi-\xi_0)^2}\\
        &\quad+ \frac{a}{\pi\sigma \sqrt{\alpha}} e^{2\pi i(\xi_1)t} e^{-\frac{1}{\alpha}(\xi-\xi_0)^2} + O(a^2)\\
        &=\frac{1}{\pi\sigma \sqrt{\alpha}} e^{2\pi i\xi_0 t} e^{-\frac{1}{\alpha}(\xi-\xi_0)^2} \left[1+a\left(\frac{2\Delta (\xi-\xi_0) \cos(2\pi\Delta t)}{\alpha}+e^{2\pi i\Delta t}\right)+O(a^2)\right]
    \end{align*}
    as $a\to 0$.Denoting $f_0(t)=e^{2\pi i\xi_0 t},$ we have that $f(t)\to f_0(t)$ as $a\to 0$, so we might also expect that $S_f^{(h)}(t,\eta)\to S_{f_0}^{(h)}(t,\eta)$. Indeed, 
    \begin{align*}
        \left|S_f^{(h)}(t,\eta)- S_{f_0}^{(h)}(t,\eta)\right|&=\frac{a}{\pi\sigma\sqrt{\alpha}}e^{-\frac{1}{\alpha}(\xi-\xi_0)^2}\left|\frac{2\Delta(\xi-\xi_0)}{\alpha}\cos(2\pi\Delta t)+e^{2\pi i \Delta t}+O(a)\right|\\
        &\leq \frac{a}{\pi\sigma\sqrt{\alpha}}e^{-\frac{1}{\alpha}(\xi-\xi_0)^2} \left(1+\frac{2\Delta}{\alpha}|\xi-\xi_0| |\cos(2\pi\Delta t)|\right) + O(a^2)
    \end{align*}
    which implies that the SST converges linearly (in terms of $a$). We can do the $a\to\infty$ case similarly to the $a\to 0$ case. Let $b\defeq 1/a$ and look at $g(t)=bf(t)=\frac{1}{a}e^{2\pi i\xi_0 t} +e^{2\pi i(\xi_1)t}$. Now, $a\to\infty$ is equivalent to $b\to 0$ for $f(t)=\frac{1}{b}g(t)$. Also notice that $V_g^{(h)}(t,\eta)=bV_f^{(h)}(t,\eta)$, $\hat{\eta}_{s,g}(t,\eta)=\hat{\eta}_{s,f}(t,\eta)$, and $S_g^{(h)}(t,\eta)=bS_f^{(h)}(t,\eta)$. With this, we can recreate the steps for the $a\to 0$ case and get
    \[
    S_f^{(h)}(t,\eta) = \frac{1}{\pi\sigma \sqrt{\alpha}} e^{2\pi i(\xi_1) t} e^{-\frac{1}{\alpha}(\xi-\xi_1)^2} \left[a+\left(e^{-2\pi i\Delta t}-\frac{2\Delta (\xi-\xi_1) \cos(2\pi\Delta t)}{\alpha}\right)+O\left(\frac{1}{a}\right)\right]
    \]
    as $a\to\infty$. With $f_1(t)=ae^{2\pi i(\xi_1) t}$ we can see that $f(t)\to f_1(t)$ as $a\to \infty$, so we might also expect that $S_f^{(h)}(t,\eta)\to S_{f_1}^{(h)}(t,\eta)$. Indeed, 
    \begin{align*}
        \left|S_f^{(h)}(t,\eta)- S_{f_1}^{(h)}(t,\eta)\right|&=\frac{1}{\pi\sigma\sqrt{\alpha}}e^{-\frac{1}{\alpha}(\xi-\xi_1)^2}\left|e^{-2\pi i \Delta t}-\frac{2\Delta(\xi-\xi_1)}{\alpha}\cos(2\pi\Delta t)+O\left(\frac{1}{a}\right)\right|\\
        &\leq \frac{1}{\pi\sigma\sqrt{\alpha}}e^{-\frac{1}{\alpha}(\xi-\xi_1)^2} \left(1+\frac{2\Delta}{\alpha}|\xi-\xi_1| |\cos(2\pi\Delta t)|\right) + O\left(\frac{1}{a}\right)
    \end{align*}
    which again implies that the SST converges linearly (when scaled by $a$ of course).
\end{proof}

\subsection{Proof of Theorem \ref{thm:sst_delta_critical}}\label{subsect:app_proof_theorem_sst_crit}
In order to prove Theorem \ref{thm:sst_delta_critical}, we focus our exploration on $t_k^+$, but prove analagous results for $t_k^-$. Since the standard deviation of the Gaussian $g_\alpha$ is $\sqrt{\alpha}$, the only $\eta$ that will contribute to the integral \eqref{eqn:sst_def} will be $\eta$ such that 
\begin{equation}\label{Definition of C in SST analysis}
|\hat{\eta}_s(t,\eta)-\xi|<C\sqrt{\alpha}\,.
\end{equation} 
If we choose $C=1,2,3,\ldots$, we are looking at $\eta$ such that ${\hat{\eta}_s}$ maps $(t,\eta)$ within $1,2,3,\ldots$ standard deviations away from $\xi$, respectively. Hence, we need to determine the preimage of ${\hat{\eta}_s}$. Specifically we are interested in the set
\be
H = \{\eta:|\hat{\eta}_s(t,\eta)-\xi|<C\sqrt{\alpha}\}\label{preimage}
\ee
for fixed $(t,\xi)$. In what follows, define $I_1,\ldots, I_7$ as  shown in Figure \ref{segments}.
\begin{figure}
\centering
\begin{tikzpicture}[>=stealth,x=70cm,y=1cm]
    \def\tickHt{0.10}
    \def\labelHt{1.5}
    \def\ivalY{-1.10}
    \def\ivalLblY{-1.30}
    \def\ivalTickHt{0.06}
    \def\xiZeroVal{1}
    \def\DeltaVal{0.1}
    \def\Asmall{0.007}

    \pgfmathsetmacro{\pA}{-\Asmall}
    \pgfmathsetmacro{\pB}{\Asmall}
    \pgfmathsetmacro{\pC}{0.5*\DeltaVal-\Asmall}
    \pgfmathsetmacro{\pD}{0.5*\DeltaVal+\Asmall}
    \pgfmathsetmacro{\pE}{\DeltaVal-\Asmall}
    \pgfmathsetmacro{\pF}{\DeltaVal+\Asmall}

    \def\tickHt{0.10}
    \def\labelHt{1.5}
    \def\ivalY {-1.50}
    \def\ivalLblY{-100.78}

    \pgfmathsetmacro{\leftEdge}{\pA-0.02}
    \pgfmathsetmacro{\rightEdge}{\pF+0.02}

    \draw[<->] (\leftEdge,0) -- (\rightEdge,0);
    \node[below] at (\leftEdge,0) {$-\infty$};
    \node[below] at (\rightEdge,0) {$\infty$};

    \draw[blue,dashed] (0,0.3) -- (0,-0.3)
          node[below,blue] {$\xi_0$};
    
    \draw[blue,dashed] (\DeltaVal/2,0.3) -- (\DeltaVal/2,-0.3)
          node[below,blue] {$\bar{\xi}$};
    
    \draw[blue,dashed] (\DeltaVal,0.3) -- (\DeltaVal,-0.3)
          node[below,blue] {$\xi_1$};

    \foreach \x/\lbl in {
      \pA/$\xi_0 - C\sqrt{\alpha}$,
      \pB/$\xi_0 + C\sqrt{\alpha}$,
      \pC/$\bar{\xi} - C\sqrt{\alpha}$,
      \pD/$\bar{\xi} + C\sqrt{\alpha}$,
      \pE/$\xi_1 - C\sqrt{\alpha}$,
      \pF/$\xi_1 + C\sqrt{\alpha}$}
    {\draw (\x,-\tickHt) -- (\x,\tickHt);
    \node[rotate=90] at (\x,\labelHt) {\lbl};}

    \coordinate (L) at (\leftEdge,-1.7);
    \coordinate (A) at (\pA,-1.7);
    \coordinate (B) at (\pB,-1.7);
    \coordinate (C) at (\pC,-1.7);
    \coordinate (D) at (\pD,-1.7);
    \coordinate (E) at (\pE,-1.7);
    \coordinate (F) at (\pF,-1.7);
    \coordinate (R) at (\rightEdge,-1.7);
    
    \foreach \p/\q/\name in {
        L/A/I_{1},
        A/B/I_{2},
        B/C/I_{3},
        C/D/I_{4},
        D/E/I_{5},
        E/F/I_{6},
        F/R/I_{7}}
    {\draw[red] (\p|-0,\ivalY) -- (\q|-0,\ivalY);
      \node[red] at ($(\p)!0.5!(\q)$ |- 0,\ivalLblY-) {$\name$};
    }
    \foreach \x in {A,B,C,D,E,F}
    {
      \draw[red] (\x|-0,\ivalY+\ivalTickHt) -- (\x|-0,\ivalY-\ivalTickHt);
    }
\end{tikzpicture}
\caption{The $\xi$-axis with labeled sub-intervals for Lemmas \ref{lem:preimage_const}, \ref{lem:preimage_dest}, and \ref{lem:preimage_mid}.}
\label{segments}
\end{figure}
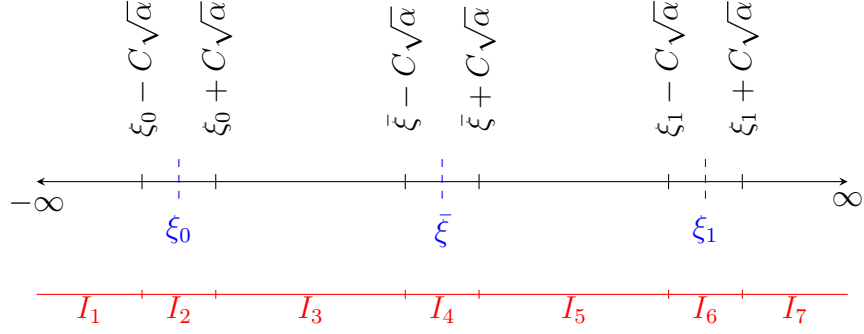

\begin{lemma}\label{lem:preimage_const}
    Let $t=t_{k}^{+}$ and assume $0<C\leq\frac{\Delta}{4\sqrt{\alpha}}$ in \eqref{Definition of C in SST analysis}. We define the preimage as in \eqref{preimage}.
    Define 
    \[
    C_R(\xi)\defeq\frac{1}{2\pi^2\sigma^2\Delta}\ln\left(-1-\frac{\Delta}{\xi-\xi_1+C\sqrt{\alpha}}\right)
    \]
    and
    \[
    C_L(\xi)\defeq\frac{1}{2\pi^2\sigma^2\Delta}\ln\left(-1-\frac{\Delta}{\xi-\xi_1-C\sqrt{\alpha}}\right)\,,
    \]
    which are both smooth real functions when $\xi\in(\xi_0-C\sqrt{\alpha}, \xi_1-C\sqrt\alpha)$ and $\xi\in(\xi_0+C\sqrt{\alpha}, \xi_1+C\sqrt\alpha)$ respectively. Then
    \begin{itemize}
        \item for $\xi\in I_1\cup I_7$, the pre-image $H=\emptyset$,
        \item for $\xi\in I_2$, the pre-image $H=(-\infty, \eta_{\text{AVG}}+C_R)$ and in this case $C_R<0$,
        \item for $\xi\in I_3\cup I_4\cup I_5$, the pre-image $H=[\eta_{\text{AVG}}+C_L, \eta_{\text{AVG}}+C_R]$,
        \begin{itemize}
            \item For $I_3$, $C_L<C_R<0$,
            \item For $I_4$, $C_L<0<C_R$,
            \item For $I_3$, $0<C_L<C_R$,
        \end{itemize}
        \item and for $\xi\in I_6$, the pre-image $H=(\eta_{\text{AVG}}+C_L,\infty)$ and in this case $C_L>0$.
    \end{itemize}
\end{lemma}
\begin{proof}[Proof of Lemma \ref{lem:preimage_const}]
    We can write 
    \begin{align*}
        \hat{\eta}_s(t_{k}^{+},\eta)&=\frac{\xi_0+(\xi_1)q(t_{k}^{+},\eta)}{1+q(t_{k}^{+},\eta)}\\
        &= \xi_0+\Delta\frac{q(t_{k}^{+},\eta)}{1+q(t_{k}^{+},\eta)}.
    \end{align*}
    Now, letting $d\defeq \xi-\xi_1$, we have
    \begin{align*}
        \left|\hat{\eta}_s(t_{k}^{+},\eta)-\xi\right|&= \left|\Delta\frac{q(t_{k}^{+},\eta)}{1+q(t_{k}^{+},\eta)}-\Delta-d\right|\\
        &=\left|\Delta\left(\frac{q(t_{k}^{+},\eta)}{1+q(t_{k}^{+},\eta)}-1\right)-d\right|\\
        &=\left|-\Delta\frac{1}{1+q(t_{k}^{+},\eta)}-d\right|\\
        &=\left|\Delta\frac{1}{1+q(t_{k}^{+},\eta)}+d\right|.
    \end{align*}
    We want to know when $\left|\hat{\eta}_s(t_{k}^{+},\eta)-\xi\right|<C\sqrt{\alpha}$ which occurs when $$\left|\Delta\frac{1}{1+q(t_{k}^{+},\eta)}+d\right|<C\sqrt{\alpha}$$ or equivalently when $$\underbrace{\frac{-C\sqrt\alpha-d}{\Delta}}_{\defeq\epsilon_1}<\frac{1}{1+q(t_{k}^{+},\eta)}<\underbrace{\frac{C\sqrt\alpha-d}{\Delta}}_{\defeq\epsilon_2}.$$ Notice that $0<\frac{1}{1+q(t_{k}^{+},\eta)}=\frac{1}{1+a\exp\{-\pi^2\sigma^2\Delta(\Delta-2(\eta-\xi_0))}\}<1$ for $\eta\in\mathbb{R}$. In particular, since $q(t_{k}^{+},\eta)>0$ an immediate consequence is that $H=\{\emptyset\}$ for $\xi\in I_1\cup I_7$ (since the above inequality is satisfied for no $\eta\in\mathbb{R}$ when $d>C\sqrt{\alpha}$ or $d<-\Delta-C\sqrt{\alpha}$). Now, suppose that $-\Delta-C\sqrt\alpha<d<C\sqrt{\alpha}$. This implies that $-\frac{2C\sqrt{\alpha}}{\Delta}<\epsilon_1<1$ and that $0<\epsilon_2<1+\frac{2C\sqrt{\alpha}}{\Delta}$. Now, we can split into two cases where $\epsilon_1>0$ ($-\Delta-C\sqrt{\alpha}<d<-C\sqrt{\alpha}$) and where $\epsilon_1<0$ ($-C\sqrt{\alpha}<d<C\sqrt{\alpha}$). For the first case, we have
    \begin{align*}
        \epsilon_1&<\frac{1}{1+q(t_{k}^{+},\eta)}<\epsilon_2\\
        &\implies \frac{1}{\epsilon_2}-1<q(t_{k}^{+},\eta)<\frac{1}{\epsilon_1}-1\\
        &\implies -\frac{\Delta+C\sqrt{\alpha}+d}{C\sqrt{\alpha}+d}<\exp\{-\pi^2\sigma^2\Delta[\Delta-2(\eta-\xi_0)]+\ln a\}<\frac{\Delta-C\sqrt{\alpha}+d}{C\sqrt{\alpha}-d}\\
        &\implies \eta\in (\eta_{\text{AVG}}+C_2, \eta_{\text{AVG}}+C_1)
    \end{align*}
    where $C_1,C_2,$ and $\eta_{\text{AVG}}$ are defined as above. This case gives the results for $\xi\in I_2\cup I_3\cup I_4\cup I_5$. Also notice that $C_2\to-\infty$ as $\xi\to \xi_0-C\sqrt{\alpha}$ from the right.

    For the second case, we have 
    \begin{align*}
        \epsilon_1<\frac{1}{1+q(t_{k}^{+},\eta)}<\epsilon_2 &\implies0<\frac{1}{1+q(t_{k}^{+},\eta)}<\epsilon_2\\
        &\implies \frac{1}{\epsilon_2}-1<q(t_{k}^{+},\eta)<\infty\\
        &\implies \frac{\Delta-C\sqrt{\alpha}+d}{C\sqrt{\alpha}-d}<\exp\{-\pi^2\sigma^2\Delta[\Delta-2(\eta-\xi_0)]+\ln a\}<\infty\\
        &\implies \eta\in (\eta_{\text{AVG}}+C_2, \infty)
    \end{align*}
    which gives the result for $\xi\in I_6$.
\end{proof}
Now, we will proof results analogous to \ref{lem:preimage_const} for $t_k^-$ and the time in between the constructive and destructive times.
\begin{lemma}\label{lem:preimage_dest}
    Let $t=t_{k}^{-}$ and assume $0<C\leq\frac{\Delta}{4\sqrt{\alpha}}$ in \eqref{Definition of C in SST analysis}. We define the preimage as in \eqref{preimage}. Define
    \[
    C_L\defeq\frac{1}{2\pi^2\sigma^2\Delta}\ln\left(1+\frac{\Delta}{\xi-\xi_1+C\sqrt{\alpha}}\right)
    \]
    and
    \[
    C_R\defeq\frac{1}{2\pi^2\sigma^2\Delta}\ln\left(1+\frac{\Delta}{\xi-\xi_1-C\sqrt{\alpha}}\right)
    \]
    which are both smooth real functions when $\xi\in(-\infty,\xi_0-C\sqrt{\alpha}) \cup (\xi_1-C\sqrt\alpha,\infty)$ and $\xi\in(-\infty,\xi_0+C\sqrt{\alpha}) \cup (\xi_1+C\sqrt\alpha,\infty)$ respectively. Then
    \begin{itemize}
        \item for $\xi\in I_1\cup I_7$, the pre-image $H=[\eta_{\text{AVG}}+C_L, \eta_{\text{AVG}}+C_R]$,
        \begin{itemize}
            \item For $I_1$, $C_L<C_R<0$,
            \item For $I_7$, $0<C_L<C_R$
        \end{itemize}
        \item for $\xi\in I_2$, the pre-image $H=(-\infty, \eta_{\text{AVG}}+C_R)$ and in this case $C_R<0$,
        \item for $\xi\in I_3\cup I_4\cup I_5$, the pre-image $H=\emptyset$,
        \item and for $\xi\in I_6$, the pre-image $H=(\eta_{\text{AVG}}+C_L,\infty)$ and in this case $C_L>0$.
    \end{itemize}
\end{lemma}
\begin{proof}[Proof of Lemma \ref{lem:preimage_dest}]
    Similarly to the proof of the previous Lemma \ref{lem:preimage_const}, we can write
    \begin{align*}
        \hat{\eta}_s(t_{k}^{-},\eta)&= \xi_0+\Delta\frac{q(t_{k}^{-},\eta)}{1+q(t_{k}^{-},\eta)}.
    \end{align*}
    Again, letting $d\defeq \xi-\xi_1$, we have
    \begin{align*}
        \left|\hat{\eta}_s(t_{k+-},\eta)-\xi\right|&=\left|\Delta\frac{1}{1+q(t_{k}^{-},\eta)}+d\right|.
    \end{align*}
    We want to know when $\left|\hat{\eta}_s(t_{k}^{-},\eta)-\xi\right|<C\sqrt{\alpha}$ which occurs when $$\underbrace{\frac{-C\sqrt\alpha-d}{\Delta}}_{\defeq\epsilon_1}<\frac{1}{1+q(t_{k}^{-},\eta)}<\underbrace{\frac{C\sqrt\alpha-d}{\Delta}}_{\defeq\epsilon_2}.$$ Notice that $\frac{1}{1+q(t_{k}^{-},\eta)}=\frac{1}{1-a\exp\{-\pi^2\sigma^2\Delta(\Delta-2(\eta-\xi_0))\}}\in (-\infty,0)\cup (1,\infty)$ for $\eta\in\mathbb{R}$. An immediate consequence is that $H=\{\emptyset\}$ for $\xi\in I_3$ (since the above inequality is satisfied for no $\eta\in\mathbb{R}$ when $-\Delta+C\sqrt\alpha<-C\sqrt{\alpha}$).The remaining cases of $\xi$ follow from the same type of analysis done in the proof of Lemma \ref{lem:preimage_const}.
\end{proof}
\begin{remark}
    The assumption for $0<C\leq\frac{\Delta}{4\sqrt{\alpha}}$ in Lemmas \ref{lem:preimage_const} and \ref{lem:preimage_dest} is needed for the disjoint segmentation described in Figure \ref{segments}. Removing this assumption can be done, however segmenting in such a way makes the behavior over different intervals more interpretable. Setting $C$ as large as possible provides the best approximation.
\end{remark}

\begin{lemma}\label{lem:preimage_mid}
    Let $t=t_{k}^{I}\defeq \frac{k+\frac{1}{4}}{\Delta}$ and assume $0<C\leq\frac{\Delta}{2\sqrt{\alpha}}$ in \eqref{Definition of C in SST analysis}. 
    Define
    \[
    C_*(\xi) \defeq\ \frac{1}{2\pi^2\sigma^2\Delta} \ln\left(\sqrt{\frac{(\xi-\xi_0)^2-C^2\alpha}{C^2\alpha-(\xi-\xi_1)^2}}\right)\,,
    \] 
 which is real-valued and smooth on $I_2\cup I_6$ since the quantity inside the logarithm is strictly positive for $\xi\in I_2\cup I_6$, hence $C_*(\xi)$.
    Denoting the pre-image as in \eqref{preimage} we have,
    \begin{itemize}
        \item for $\xi\in I_1\cup I_3\cup I_4\cup I_5\cup I_7$, the pre-image $H=\emptyset$,
        \item for $\xi\in I_2$, the pre-image $H=(-\infty, \eta_{\text{AVG}}+C_*(\xi))$ and in this case $C_*(\xi)<0,$
        \item and for $\xi\in I_6$, the pre-image $H=(\eta_{\text{AVG}}+C_*(\xi), \infty)$ and in this case $C_*(\xi)>0.$
    \end{itemize}
\end{lemma}
\begin{proof}[Proof of Lemma \ref{lem:preimage_mid}]
    First, we can denote 
    \[
    E\defeq e^{2\pi^2\sigma^2 \Delta(\eta-\bar{\xi})}
    \]
    and write
    \begin{align*}
        q(t_k^I, \eta)&= ae^{2\pi i\Delta t_k^I}e^{2\pi^2\sigma^2 \Delta(\eta-\bar{\xi}}\\
        &= ae^{\frac{\pi i}{2}}E\\
        &= iaE.
    \end{align*}
    Once again letting $d=\xi_1-\Delta$, we have
    \begin{align*}
        \left| \hat{\eta}_s(t_k^I,\eta)-\xi \right|&= \left| \Delta \frac{1}{1+q(t_k^I, \eta)}+d\right|\\
        &=\left| \Delta \frac{1}{1+iaE}+d\right|\\
        &= \left| d+\Delta \frac{1}{1+a^2E^2}-i\frac{\Delta aE}{1+a^2 E^2}\right|\\
        &= \sqrt{d^2+\frac{2d\Delta+\Delta^2}{1+a^2E^2}}\\
        &=\sqrt{\frac{d^2(1+a^2E^2) + 2d\Delta +\Delta^2}{1+a^2E^2}}.
    \end{align*}
    We want to find which sets/intervals of $\eta$ and $\xi$ satisfy $| \hat{\eta}_s(t_k^I,\eta)-\xi |<C\sqrt{\alpha}$ which means we need 
    \[
    d^2(1+a^2E^2)+2d\Delta+\Delta^2<C^2\alpha(1+a^2E^2)
    \]
    or equivalently by completing the square, 
    \[
    \left(d+\frac{\Delta}{1+a^2 E^2}\right)^2< C^2\alpha - \frac{\Delta^2 a^2 E^2}{(1+a^2E^2)^2}.
    \]
    Set 
    \[
    y\defeq\frac{1}{1+a^2F^2}\in(0,1],
    \]
    which is strictly decreasing in $\eta$. Using $a^2E^2=e^{4\pi^2\sigma^2\Delta(\eta-\eta_{\text{AVG}})}$, we have the logistic form
    \[
    y(\eta)=\frac{1}{1+e^{4\pi^2\sigma^2\Delta(\eta-\eta_{\text{AVG}})}}.
    \]
    Expanding the previous inequality gives
    \begin{align*}
        (d+\Delta y)^2+\Delta^2y(1-y)&<C^2\alpha\\
        d^2+2d\Delta y+\Delta^2 y&<C^2\alpha\\
        \Delta(\Delta+2d)y&<C^2\alpha-d^2.
    \end{align*}
    Now we can split into cases based on the sign of $\Delta+2d$.

    \textbf{Case 1: $\Delta+2d>0$ (i.e. $\xi>\bar{\xi}$)}\\
    In this case 
    \[
    y<\frac{C^2\alpha-d^2}{\Delta(\Delta+2d)}.
    \]
    For a feasible solution we must have $y\in(0,1]$ satisfying 
    \[
    \frac{C^2\alpha-d^2}{\Delta(\Delta+2d)}>0\quad\text{and}\quad \frac{C^2\alpha-d^2}{\Delta(\Delta+2d)}\leq 1.
    \]
    Since $\Delta(\Delta+2d)>0$ in Case 1, these are equivalent to
    \[
    |d|<C\sqrt{\alpha}\qquad\text{and}\qquad(d+\Delta)^2>C^2\alpha.
    \]
    The first inequality implies that 
    \[
    \xi\in(\xi_1-C\sqrt{\alpha}, \xi_1+C\sqrt{\alpha})=I_6,
    \]
    while the second ineuqality is automatically satisfied for $\xi\in I_6$ because
    \begin{align*}
        (d+\Delta)^2&=|\xi-\xi_0|^2\\
        &\geq (\Delta-|d|)^2\\
        &> (\Delta-C\sqrt{\alpha})^2\\
        &> C^2\alpha
    \end{align*}
    where the last inequality uses $C<\Delta/(2\sqrt{\alpha})$. 
    Conversely, for any $\xi>\bar{\xi}$ outside $I_6$ (namely $\xi\in I_5$ or $I_7$, and the portion of $I_4$ with $\xi>\bar{\xi}$ when it does not overlap $I_6$), we have $|d|\geq C\sqrt{\alpha}$ so the right hand side is $\leq 0$ and there is no feasible $y\in(0,1]$. Therefore, within Case 1 the only valid $\xi$ are those in $I_6$, and the inequality $y<(\cdots)$ is equivalent to 
    \[
    \eta>\eta_{\text{AVG}}+\frac{1}{2\pi^2\sigma^2\Delta}\ln\left(\sqrt{\frac{(d+\Delta)^2-C^2\alpha}{C^2\alpha-d^2}}\right)=\eta_{\text{AVG}}+C_*(\xi).
    \]
    Thus $H=(\eta_{\text{AVG}}+C_*(\xi),\infty)$ with $C_*(\xi)>0$ on $I_6$.

    \noindent\textbf{Case 2: $\Delta+2d<0$ (i.e. $\xi<\xi_0+\frac{\Delta}{2}$)}\\
    In this case 
    \[
    y>\frac{C^2\alpha-d^2}{\Delta(\Delta+2d)}.
    \]
    For feasible solutions we need
    \[
    0<\frac{C^2\alpha-d^2}{\Delta(\Delta+2d)}<1.
    \]
    Since now $\Delta(\Delta+2d)<0$, these inequalities are equivalent to
    \[
    |d|>C\sqrt{\alpha}\qquad\text{and}\qquad(d+\Delta)^2<C^2\alpha.
    \]
    The second inequality says $|\xi-\xi_0|<C\sqrt{\alpha}$, so
    \[
    \xi\in(\xi_0-C\sqrt{\alpha}, \xi_0+C\sqrt{\alpha})=I_2.
    \]
    The first inequality is automatic on $I_2$ since $|d|=|\xi-\xi_1|\geq \Delta-|\xi-\xi_0|>\Delta-C\sqrt{\alpha}>C\sqrt{\alpha}$. 
    Conversely, if $\xi<\bar{\xi}$ but $\xi\notin I_2$ (i.e.\ $\xi\in I_1$ or $I_3$, or the portion of $I_4$ with $\xi<\bar{\xi}$ when it does not overlap $I_2$), then $|\xi-\xi_0|\geq C\sqrt{\alpha}$ and so $(d+\Delta)^2\geq C^2\alpha$, making the right hand side $\leq 0$ (or $\geq 1$ after sign flip), hence no valid $y\in(0,1]$. Therefore, within Case 2, the only feasible $\xi$ are those in $I_2$, and the inequality $y>(\cdots)$ is equivalent to 
    \[
    \eta<\eta_{\text{AVG}}+\frac{1}{2\pi^2\sigma^2\Delta}\ln\left(\sqrt{\frac{(d+\Delta)^2-C^2\alpha}{C^2\alpha-d^2}}\right)= \eta_{\text{AVG}}+C_*(\xi).
    \]
    Thus $H=(-\infty,\ \eta_{\text{AVG}}+C_*(\xi))$ with $C_*(\xi)<0$ on $I_2$.
\end{proof}

Recall from Lemmas \ref{lem:stft_const} and \ref{lem:stft_dest} that at the constructive and destructive times, $V_f^{(h)}(t,\eta)$ is a real valued function up to a complex constant, 
\[
V_f^{(h)}(t_{k}^{+},\eta)=e^{2\pi i\frac{\xi_0}{\Delta} k} \left(e^{-\pi^2 \sigma^2 (\eta-\xi_0)^2}+e^{-\pi^2 \sigma^2 (\eta-\xi_1)^2}\right),
\]
and 
\[
V_f^{(h)}(t_{k}^{-},\eta)=e^{2\pi i\frac{\xi_0}{\Delta} (k+\frac{1}{2})} \left(e^{-\pi^2 \sigma^2 (\eta-\xi_0)^2}-e^{-\pi^2 \sigma^2 (\eta-\xi_1)^2}\right).
\]
Using this and lemmas \ref{lem:preimage_const} and \ref{lem:preimage_dest}, we can compute the value of $S_f^{(h)}(t,\xi)$ for $t=t_{k}^{+}$ or $t_{k}^{-}$ up to an $O\left(\alpha^{-1/2} e^{-C^2}\right)$ correction. Note that when $C=\frac{\Delta}{4\sqrt{\alpha}}$ in \eqref{Definition of C in SST analysis} and $\alpha$ is sufficiently small, $\alpha^{-1/2}e^{-C^2}$ is negligible. This is made explicit in the following two theorems.

\begin{theorem}\label{thm:sst_const_asymp}
    Let $t=t_{k}^{+}$. Define
    \[
    \gamma_1(\xi) =\bar{\xi}+\frac{1}{2\pi^2\sigma^2\Delta}\ln\left(-a^{-1}\left(1+\frac{\Delta}{\xi-\xi_1+C\sqrt{\alpha}}\right)\right)
    \]
    and
    \[
    \gamma_2(\xi) =\bar{\xi}+\frac{1}{2\pi^2\sigma^2\Delta}\ln\left(-a^{-1}\left(1+\frac{\Delta}{\xi-\xi_1-C\sqrt{\alpha}}\right)\right),
    \]
    which are both smooth real functions when $\xi\in(\xi_0-C\sqrt{\alpha}, \xi_1-C\sqrt\alpha)$ and $\xi\in(\xi_0+C\sqrt{\alpha}, \xi_1+C\sqrt\alpha)$ respectively. Consider $I_1,\ldots,I_7$ defined in Figure \ref{segments}. Suppose $0<C\leq\frac{\Delta}{4\sqrt{\alpha}}$ in \eqref{Definition of C in SST analysis}. When $\alpha$ is sufficiently small,
    up to an $O\left(\alpha^{-1/2} e^{-C^2}\right)$ correction,\\
    $\left|S_f^{(h)}(t_{k}^{+},\xi)\right| \\
    \sim\begin{cases}
    0 & \text{if } \xi \in I_1\cup I_7,\\

    \alpha^{-1/2}\left[1+a+\erf\left(\pi\sigma\left(\gamma_1(\xi)-\xi_{0}\right)\right)+a\erf\left(\pi\sigma\left(\gamma_1(\xi)-{\xi_1}\right)\right)\right] & \text{if } \xi\in I_2,\\

    \alpha^{-1/2}\left[1+a-\erf\left(\pi\sigma\left(\gamma_2(\xi)-\xi_{0}\right)\right)-a\erf\left(\pi\sigma\left(\gamma_2(\xi)-{\xi_1}\right)\right)\right] & \text{if } \xi\in I_6,\\

    \alpha^{-1/2}[\erf\left(\pi\sigma\left(\gamma_1(\xi)-\xi_{0}\right)\right)-\erf\left(\pi\sigma\left(\gamma_2(\xi)-\xi_{0}\right)\right)\\
    \;\;\;\;\;\;\;\;\;\;+a\erf\left(\pi\sigma\left(\gamma_1(\xi)-{\xi_1}\right)\right)-a\erf\left(\pi\sigma\left(\gamma_2(\xi)-{\xi_1}\right)\right)]& \text{if } \xi\in I_3\cup I_4\cup I_5\,.
    \end{cases}$ 
\end{theorem}
\begin{proof}[Proof of Theorem \ref{thm:sst_const_asymp}]
    First, denote $\Psi(\xi)\defeq \sqrt{\pi\alpha} S_f^{(h)}(t_{k}^{+},\xi)$. For $\xi$ in any interval, we can write
    \begin{align*}
        \Psi(\xi)&=\int_{\mathbb{R}} V_f^{(h)}(t_{k}^{+},\eta) g_\alpha\left(\hat{\eta}_s(t_{k}^{+},\eta)-\xi\right) d\eta\\
        &=\underbrace{\int_{H} V_f^{(h)}(t_{k}^{+},\eta) g_\alpha\left(\hat{\eta}_s(t_{k}^{+},\eta)-\xi\right) d\eta}_{\defeq \Psi_1}+\underbrace{\int_{H^C} V_f^{(h)}(t_{k}^{+},\eta) g_\alpha\left(\hat{\eta}_s(t_{k}^{+},\eta)-\xi\right) d\eta}_{\defeq \Psi_2}.
    \end{align*}
    Starting with $I_1$ (the $I_7$ case follows the same), Lemma \ref{lem:preimage_const} suggests we should use $H=\emptyset$ and $H^C=\mathbb{R}$, so we have $\Psi_1=0$. Furthermore, for $\Psi_2$ the $g_\alpha$ term can be bounded by $e^{-C^2}$, so for $\xi\in I_1$ we have 
    \begin{align*}
        \Psi(\xi)&=\int_{H^C} V_f^{(h)}(t_{k}^{+},\eta) g_\alpha\left(\hat{\eta}_s(t_{k}^{+},\eta)-\xi\right) d\eta\\
        &=\int_{\mathbb{R}} V_f^{(h)}(t_{k}^{+},\eta) g_\alpha\left(\hat{\eta}_s(t_{k}^{+},\eta)-\xi\right) d\eta\\
        &\leq e^{-C^2} \int_{\mathbb{R}} V_f^{(h)}(t_{k}^{+},\eta) d\eta\\
        &=e^{-C^2} e^{2\pi i \xi_0 t_{k}^{+}}\left[\int_{\mathbb{R}} e^{-\pi^2\sigma^2 (\eta-\xi_0)^2} d\eta+a\int_{\mathbb{R}} e^{-\pi^2\sigma^2 (\eta-\xi_1)^2} d\eta\right]\\
        &= \frac{e^{-C^2} e^{2\pi i \xi_0 t_{k}^{+}}}{\pi\sigma\sqrt{\alpha}}(1+a)\\
        &=O\left(\alpha^{-1/2} e^{-C^2}\right).
    \end{align*}
    For $I_2$ ($I_6$ follows the same analysis with minor adjustment), Lemma \ref{lem:preimage_const} suggests having $\Psi_1$ be over $H=(-\infty,\eta_{\text{AVG}}+C_1)$ and $\Psi_2$ be over $H^C=[\eta_{\text{AVG}}+C_1,\infty)$ so that for $\Psi_2$, the $g_\alpha$ term can be bounded by $e^{-C^2}$. Lemma \ref{lem:stft_const} also shows that in $\Psi_1$, $V$ is exponentially small, so for $\xi\in I_2$ we have
    \begin{align*}
        \Psi(\xi)&=\int_{-\infty}^{\eta_{\text{AVG}}+C_1} V_f^{(h)}(t_{k}^{+},\eta) g_\alpha\left(\hat{\eta}_s(t_{k}^{+},\eta)-\xi\right) d\eta\\
        &\qquad+\int_{\eta_{\text{AVG}}+C_1}^{\infty} V_f^{(h)}(t_{k}^{+},\eta) g_\alpha\left(\hat{\eta}_s(t_{k}^{+},\eta)-\xi\right) d\eta\\
        &\leq \int_{-\infty}^{\eta_{\text{AVG}}+C_1} V_f^{(h)}(t_{k}^{+},\eta) d\eta+e^{-C^2}\int_{\eta_{\text{AVG}}+C_1}^{\infty} V_f^{(h)}(t_{k}^{+},\eta)  d\eta\\
        &=e^{2\pi i \xi_0 t_{k}^{+}}\bigg[\int_{-\infty}^{\eta_{\text{AVG}}+C_1}e^{-\pi^2\sigma^2 (\eta-\xi_0)^2} d\eta+a\int_{-\infty}^{\eta_{\text{AVG}}+C_1}e^{-\pi^2\sigma^2 (\eta-\xi_1)^2} d\eta\\
        &\qquad+e^{-C^2}\int_{\eta_{\text{AVG}}+C_1}^{\infty}e^{-\pi^2\sigma^2 (\eta-\xi_0)^2} d\eta+ae^{-C^2}\int_{\eta_{\text{AVG}}+C_1}^{\infty}e^{-\pi^2\sigma^2 (\eta-\xi_1)^2} d\eta\bigg]\\
        &=\frac{e^{2\pi i \xi_0 t_{k}^{+}}}{2\sigma\sqrt{\pi}}\bigg[\erf(\pi\sigma(\gamma_1(\xi)-\xi_0))+1+a\left(\erf(\pi\sigma(\gamma_1(\xi)-\xi_1))+1\right)\\
        &\qquad+e^{-C^2}\Big(\erfc(\pi\sigma(\gamma_1(\xi)-\xi_0))+1+a\left(\erfc(\pi\sigma(\gamma_1(\xi)-\xi_1))+1\right)\Big)\bigg]\\
        &=\frac{e^{2\pi i \xi_0 t_{k}^{+}}}{2\pi\sigma\sqrt{\alpha}}\bigg[\left(1-e^{-C^2}\right)\left(\erf(\pi\sigma(\gamma_1(\xi)-\xi_0))+a\erf(\pi\sigma(\gamma_1(\xi)-\xi_1))\right)\\
        &\qquad+\left(1+e^{-C^2}\right)(1+a)\bigg].
    \end{align*}
    Finally for $\xi\in I_3$ (and similarly for $I_4$ or $I_5$), we can use Lemmas \ref{lem:stft_const} and \ref{lem:preimage_const} similarly as above to find 
    \begin{align*}
        \Psi(\xi)&=\int_{\eta_{\text{AVG}}+C_2}^{\eta_{\text{AVG}}+C_1} V_f^{(h)}(t_{k}^{+},\eta) g_\alpha\left(\hat{\eta}_s(t_{k}^{+},\eta)-\xi\right) d\eta\\
        &\qquad+\int_{-\infty}^{\eta_{\text{AVG}}+C_2} V_f^{(h)}(t_{k}^{+},\eta) g_\alpha\left(\hat{\eta}_s(t_{k}^{+},\eta)-\xi\right) d\eta\\
        &\qquad+\int_{\eta_{\text{AVG}}+C_1}^{\infty} V_f^{(h)}(t_{k}^{+},\eta) g_\alpha\left(\hat{\eta}_s(t_{k}^{+},\eta)-\xi\right) d\eta\\
        &=\int_{\eta_{\text{AVG}}+C_2}^{\eta_{\text{AVG}}+C_1} V_f^{(h)}(t_{k}^{+},\eta) g_\alpha\left(\hat{\eta}_s(t_{k}^{+},\eta)-\xi\right) d\eta+O\left(e^{-C^2}\right)
    \end{align*}
    since $g_\alpha$ is small for the second two integrals above. Now, we can write
    \begin{align*}
        \Psi(\xi)-\int_{\eta_{\text{AVG}}+C_2}^{\eta_{\text{AVG}}+C_1} V_f^{(h)}(t_{k}^{+},\eta) d\eta &= \int_{\eta_{\text{AVG}}+C_2}^{\eta_{\text{AVG}}+C_1} V_f^{(h)}(t_{k}^{+},\eta) \left(g_\alpha\left(\hat{\eta}_s(t_{k}^{+},\eta)-\xi\right)-1\right) d\eta\\
        &\qquad\qquad\qquad+O\left(e^{-C^2}\right)\\
        &=O\left(e^{-C^2}\right)
    \end{align*}
    since $g_\alpha= 1-O\left(e^{-C^2}\right)$ in this interval. Thus,
    \begin{align*}
        \Psi(\xi)&=\int_{\eta_{\text{AVG}}+C_2}^{\eta_{\text{AVG}}+C_1} V_f^{(h)}(t_{k}^{+},\eta) d\eta+O\left(e^{-C^2}\right)\\
        &=\frac{e^{2\pi i \xi_0 t_{k}^{+}}}{2\pi\sigma\sqrt{\alpha}}\bigg[\erf\Big(\pi\sigma\left(\gamma_1(\xi)-\xi_{0}\right)\Big)-\erf\Big(\pi\sigma\left(\gamma_2(\xi)-\xi_{0}\right)\Big)\\
    &\qquad+a\erf\Big(\pi\sigma\left(\gamma_1(\xi)-\xi_{0}-\Delta\right)\Big)-a\erf\Big(\pi\sigma\left(\gamma_2(\xi)-\xi_{0}-\Delta\right)\Big)\bigg]\\
    &\qquad+O\left(e^{-C^2}\right)
    \end{align*}
\end{proof}
See Appendix \ref{app:animations} for an animation of the approximation in the above theorem as $\Delta$ varies with $C=\frac{\Delta}{4\sqrt{\alpha}}$.

\begin{theorem}\label{thm:sst_dest_asymp}
    Let $t=t_{k}^{-}$. Define
    \[
    \gamma_1(\xi) =\bar{\xi}+\frac{1}{2\pi^2\sigma^2\Delta}\ln\left(a^{-1}\left(1+\frac{\Delta}{\xi-\xi_1+C\sqrt{\alpha}}\right)\right)
    \]
    and
    \[
    \gamma_2(\xi) =\bar{\xi}+\frac{1}{2\pi^2\sigma^2\Delta}\ln\left(a^{-1}\left(1+\frac{\Delta}{\xi-\xi_1-C\sqrt{\alpha}}\right)\right)\,,
    \] 
    which are both smooth real functions when $\xi\in(-\infty,\xi_0-C\sqrt{\alpha}) \cup (\xi_1-C\sqrt\alpha,\infty)$ and $\xi\in(-\infty,\xi_0+C\sqrt{\alpha}) \cup (\xi_1+C\sqrt\alpha,\infty)$ respectively. Consider $I_1,\ldots,I_7$ defined in Figure \ref{segments}. Suppose $C<\frac{\Delta}{4\sqrt{\alpha}}$ in \eqref{Definition of C in SST analysis}. When $\alpha$ is sufficiently small,
    up to an $O\left(\alpha^{-1/2} e^{-C^2}\right)$ correction,\\
    $\left|S_f^{(h)}(t_{k}^{+},\xi)\right|\\
    \sim \begin{cases}
    \alpha^{-1/2}[\erf\left(\pi\sigma\left(\gamma_1(\xi)-\xi_{0}\right)\right)-\erf\left(\pi\sigma\left(\gamma_2(\xi)-\xi_{0}\right)\right)\\
    \;\;\;\;\;\;\;\;\;\;+a\erf\left(\pi\sigma\left(\gamma_1(\xi)-{\xi_1}\right)\right)-a\erf\left(\pi\sigma\left(\gamma_2(\xi)-{\xi_1}\right)\right)]& \text{if } \xi\in I_1\cup I_7,\\

    \alpha^{-1/2}\left[1+a+\erf\left(\pi\sigma\left(\gamma_2(\xi)-\xi_{0}\right)\right)+a\erf\left(\pi\sigma\left(\gamma_2(\xi)-{\xi_1}\right)\right)\right] & \text{if } \xi\in I_2,\\

    0 & \text{if } \xi \in I_3\cup I_4 \cup I_5,\\

    \alpha^{-1/2}\left[1+a-\erf\left(\pi\sigma\left(\gamma_1(\xi)-\xi_{0}\right)\right)-a\erf\left(\pi\sigma\left(\gamma_1(\xi)-{\xi_1}\right)\right)\right] & \text{if } \xi\in I_6,\\

    \end{cases}$
    
\end{theorem}
\begin{proof}[Proof of Theorem \ref{thm:sst_dest_asymp}]
    The proof follows from the same analysis done for the Theorem \ref{thm:sst_const_asymp}.
\end{proof}
See Appendix \ref{app:animations} for an animation of the approximation in the above theorem as $\Delta$ varies with $C=\frac{\Delta}{4\sqrt{\alpha}}$. Clearly, at the destructive time, even if two frequencies are close, SST is capable of identifying two components, while the estimated frequencies are biased. The bias is larger when $\Delta$ is smaller, and this is quantified by $\gamma_i$ in the above approximation. As a consequence of the above theorems, we have the following proof of the main theorem regarding spectral interference of the SST.

\begin{proof}[Proof of Theorem \ref{thm:sst_delta_critical}]
    Writing $y=\xi-\xi_0$ with $C=\frac{\Delta}{4\sqrt{\alpha}}$,
    \[
    y-\Delta \pm C\sqrt{\alpha}.
    \]
    Setting $s_1=y-\frac{3\Delta}{4}$ and $s_2=s_1-\frac{\Delta}{2}$ so that
    \begin{gather*}
        \gamma_1(\xi)=\bar{\xi}+\frac{1}{2\pi^2\sigma^2\Delta}\ln\left(-a^{-1}\left(1+\frac{\Delta}{s_1}\right)\right),\\
        \gamma_2(\xi)=\bar{\xi}+\frac{1}{2\pi^2\sigma^2\Delta}\ln\left(-a^{-1}\left(1+\frac{\Delta}{s_2}\right)\right).
    \end{gather*}
    Define $\ell_1=\ln\left(-a^{-1}\left(1+\Delta/s_1\right)\right)$ and $\ell_2=\ln\left(-a^{-1}\left(1+\Delta/s_2\right)\right),$
    and
    \[
    z_1=\pi\sigma\left(\frac{\Delta}{2}+\frac{\ell_1}{2\pi^2\sigma^2\Delta}\right),\qquad
    z_2=\pi\sigma\left(\frac{\Delta}{2}+\frac{\ell_2}{2\pi^2\sigma^2\Delta}\right),
    \]
    \[
    z_3=z_1-\pi\sigma\Delta,\quad z_4=z_2-\pi\sigma\Delta.
    \]
    Now, differentiating $\ln(-a^{-1}(1+\Delta/s))$ gives
    \[
    \frac{d}{d\xi}\ln\left(-a^{-1}\left(1+\frac{\Delta}{s}\right)\right)=-\frac{\Delta}{s(s+\Delta)}\frac{ds}{d\xi},
    \]
    and since $\frac{ds_1}{d\xi}=\frac{ds_2}{d\xi}=1$ we have
    \[
    \gamma_1'=-\frac{1}{2\pi^2\sigma^2}\frac{1}{s_1(s_1+\Delta)},\qquad
    \gamma_2'=-\frac{1}{2\pi^2\sigma^2}\frac{1}{s_2(s_2+\Delta)}.
    \]
    Differentiating again and using $\frac{d}{d\xi}(s(s+\Delta))=2s+\Delta$ gives
    \[
    \gamma_1''=\frac{2s_1+\Delta}{2\pi^2\sigma^2s_1^2(s_1+\Delta)^2},\qquad
    \gamma_2''=\frac{2s_2+\Delta}{2\pi^2\sigma^2s_2^2(s_2+\Delta)^2}.
    \]
    Since the critical point is contained in the interval $I_3\cup I_4 \cup I_5$ from Theorem \ref{thm:sst_const_asymp}, we can look for bifurcation points in $H(\xi).$ Since $\frac{d}{du}\erf(u)=\frac{2}{\sqrt{\pi}}e^{-u^2}u'$ and $\frac{d^2}{du^2}\erf(u)=\frac{2}{\sqrt{\pi}}e^{-u^2}(u''-2u(u')^2)$, we have
    \begin{align}
    \label{eqn:fprime}
    H'(\xi)&=\frac{2\pi\sigma}{\sqrt{\pi\alpha}}\left[\gamma_1'\left(e^{-z_1^2}+ae^{-z_3^2}\right)-\gamma_2'\left(e^{-z_2^2}+ae^{-z_4^2}\right)\right],\\
    \label{eqn:fsecond}
    \begin{split}
    H''(\xi)&=\frac{2}{\sqrt{\pi\alpha}}\Big[e^{-z_1^2}\left(\pi\sigma\gamma_1''-2z_1(\pi\sigma\gamma_1')^2\right)-e^{-z_2^2}\left(\pi\sigma\gamma_2''-2z_2(\pi\sigma\gamma_2')^2\right)\\
    &\qquad+ae^{-z_3^2}\left(\pi\sigma\gamma_1''-2z_3(\pi\sigma\gamma_1')^2\right)-ae^{-z_4^2}\left(\pi\sigma\gamma_2''-2z_4(\pi\sigma\gamma_2')^2\right)\Big].
    \end{split}
    \end{align}
    We first show existence of a critical separation. Fix $\Delta>0$ and consider the interval $I_\Delta=(\xi_0+\Delta/4, \xi_0+3\Delta/4)$.  We can compute the one-sided limits of $H'(\xi)$ at the endpoints of $I_\Delta$. As $\xi\downarrow \xi_0+\frac{\Delta}{4}$ we have $s_2\to -\Delta^+$ and $s_2+\Delta\to 0^+$, so $s_2(s_2+\Delta)\to 0^-$ and
    \[
    \gamma_2'(\xi)=-\frac{1}{2\pi^2\sigma^2s_2(s_2+\Delta)}\to\infty,
    \]
    while $\gamma_1'(\xi)$ stays finite. So, we get $\lim_{\xi\downarrow \xi_0+\frac{\Delta}{4}} f'(\xi)=-\infty.$ Similarly, as $\xi\uparrow \xi_0+\frac{3\Delta}{4}$ we have $s_1\to 0^-$ and $s_1+\Delta\to \Delta>0$, so $s_1(s_1+\Delta)\to 0^-$ and $\gamma_1'(\xi)\to\infty$ while $\gamma_2'(\xi)$ stays finite. So, $\lim_{\xi\uparrow \xi_0+\frac{3\Delta}{4}} H'(\xi)=+\infty.$ By an abuse of notation, we will write $H=H(\xi)=H(\xi,\Delta)$ to make the dependence on $\Delta$ precise when necessary. By continuity of $H'$ on $I_\Delta$, the intermediate value theorem gives at least one interior zero $\xi_*(\Delta)\in I_\Delta$ with $H'(\xi_*(\Delta),\Delta)=0$. Furthermore, $H$ is $C^2$ in $(\xi,\Delta)$ on $I_\Delta\times(0,\infty)$, so $H''(\xi,\Delta)$ is continuous there. For sufficiently small $\Delta$, a direct evaluation at $\xi=\bar{\xi}$ shows $H''(\xi,\Delta)<0$, so by continuity $H''(\xi_*(\Delta),\Delta)\neq 0$ for all small enough $\Delta$. The implicit function theorem applied to $G(\xi,\Delta)\defeq H'(\xi,\Delta)$ at $(\xi_*(\Delta),\Delta)$ with $\partial_\xi G(\xi_*(\Delta),\Delta)=H''(\xi_*(\Delta),\Delta)\neq 0$ then gives a unique $C^1$ branch $\Delta\mapsto \xi_*(\Delta)$ of interior stationary points for $\Delta$ in a neighborhood of that value. Finally, for large $\Delta$ the two modes of $f$ are well separated and the stationary point $\xi_*(\Delta)$ along this branch has $H''(\xi_*(\Delta),\Delta)>0$ (it is a local minimum). By continuity in $\Delta$, there exists $\Delta_{\text{critical}}>0$ with $H''(\xi_*(\Delta_{\text{critical}}),\Delta_{\text{critical}})=0$, and at $\xi_c=\xi_*(\Delta_{\text{critical}})$ we have $H'(\xi_c)=H''(\xi_c)=0$ as claimed. Now, let $\Delta>0$ and $\xi_c$ satisfy $H'(\xi_c)=H''(\xi_c)=0$. Denote $E_k=e^{-z_k(\xi_c)^2}$ for $k=1,2,3,4$. From \eqref{eqn:fprime},
    \[
    \gamma_1'(E_1+aE_3)=\gamma_2'(E_2+aE_4).
    \]
    From \eqref{eqn:fsecond}, after rearranging,
    \begin{align*}
    & E_1\left(\pi\sigma\gamma_1''-2z_1(\pi\sigma\gamma_1')^2\right)-E_2\left(\pi\sigma\gamma_2''-2z_2(\pi\sigma\gamma_2')^2\right)\\
    &\qquad +a\left[E_3\left(\pi\sigma\gamma_1''-2z_3(\pi\sigma\gamma_1')^2\right)-E_4\left(\pi\sigma\gamma_2''-2z_4(\pi\sigma\gamma_2')^2\right)\right]=0.
    \end{align*}
    A direct calculation shows that $E_3=E_2$ and $E_4=E_1$ are equivalent to $z_3^2=z_2^2$ and $z_4^2=z_1^2$, which are equivalent to $z_1+z_2=\pi\sigma\Delta$. So, $z_3=z_1-\pi\sigma\Delta$ gives $z_3^2=z_2^2$ if and only if $(z_1-z_2)(z_1+z_2-\pi\sigma\Delta)=0$, and $z_1\neq z_2$ in the interior, hence $z_1+z_2=\pi\sigma\Delta$. The same reasoning applies to $z_4^2=z_1^2$. If $E_3\neq E_2$ or $E_4\neq E_1$, the two linear equations above in the pair $(\gamma_1'',\gamma_2'')$ have coefficients that are not compatible with the signs given by $z_1,z_2,z_3,z_4$, and no solution exists. Therefore at a double root we must have $E_3=E_2$ and $E_4=E_1$, or equivalently $z_1+z_2=\pi\sigma\Delta$. Define $r\defeq E_1/E_2=e^{-(z_1^2-z_2^2)}>0$. Using $z_1+z_2=\pi\sigma\Delta$ we have
    \[
    z_1-z_2=\frac{\ell_1-\ell_2}{2\pi\sigma\Delta},\qquad
    z_1^2-z_2^2=(z_1+z_2)(z_1-z_2)=\frac{\ell_1-\ell_2}{2},
    \]
    so $r=e^{-(\ell_1-\ell_2)/2}$. Writing $u\defeq 1+\Delta/s_1$ and $v\defeq 1+\Delta/s_2$ gives $\ell_1=\ln(-a^{-1}u)$ and $\ell_2=\ln(-a^{-1}v)$, so $r=\sqrt{v/u}$. Also, $z_1+z_2=\pi\sigma\Delta$ is equivalent to $\ell_1+\ell_2=0$, meaning $uv=a^2$. Solving $uv=a^2$ and $v/u=r^2$ with the constraint $u,v<0$ (forced by the arguments of the logarithms) gives
    \[
    u=-\frac{a}{r},\qquad v=-a r.
    \]
    Therefore
    \[
    1+\frac{\Delta}{s_1}=-\frac{a}{r} \implies s_1=-\frac{\Delta r}{a+r},\qquad
    1+\frac{\Delta}{s_2}=-a r \implies s_2=-\frac{\Delta}{a r+1}.
    \]
    Substituting $E_1=rE_2$, $E_3=E_2$, $E_4=E_1$ into $H'(\xi_c)=0$ from \eqref{eqn:fprime} gives
    \[
    0=\gamma_1'\left(rE_2+aE_2\right)-\gamma_2'\left(E_2+a r E_2\right)
    =E_2\left(\gamma_1'(r+a)-\gamma_2'(1+a r)\right),
    \]
    so
    \[
    r=\frac{\gamma_2'-a\gamma_1'}{\gamma_1'-a\gamma_2'},
    \]
    which is \eqref{eqn:R_eqn}. Likewise, substituting $E_1=rE_2$, $E_3=E_2$, $E_4=E_1$ into \eqref{eqn:fsecond}, dividing by the positive factor $E_2$, and collecting terms gives
    \begin{align*}
    0&=r\left(\pi\sigma\gamma_1''-2z_1(\pi\sigma\gamma_1')^2\right)-\left(\pi\sigma\gamma_2''-2z_2(\pi\sigma\gamma_2')^2\right)\\
    &\qquad +a\left(\pi\sigma\gamma_1''+2z_2(\pi\sigma\gamma_1')^2\right)-a r\left(\pi\sigma\gamma_2''+2z_1(\pi\sigma\gamma_2')^2\right)\\
    &=(r+a)\pi\sigma\gamma_1''-(1+a r)\pi\sigma\gamma_2''
    +2\left[(-z_1 r+a z_2)(\pi\sigma\gamma_1')^2+(z_2-a r z_1)(\pi\sigma\gamma_2')^2\right],
    \end{align*}
    which is equivalent to \eqref{eqn:K_eqn}. This proves that any double root produces a pair $(\Delta,r)$ satisfying \eqref{eqn:R_eqn}–\eqref{eqn:K_eqn}, and $\xi_c=\xi_0+\frac{3\Delta}{4}+s_1$ with $s_1=-\frac{\Delta r}{a+r}$. Conversely, suppose $(\Delta,r)$ satisfies \eqref{eqn:R_eqn}–\eqref{eqn:K_eqn}. Define $s_1,s_2,z_1,z_2$ as in the statement and set $\xi_c=\xi_0+\frac{3\Delta}{4}+s_1$. Then $u=1+\Delta/s_1=-a/r$ and $v=1+\Delta/s_2=-a r$ imply $\ell_1=-\ln r$, $\ell_2=\ln r$, so $z_1+z_2=\pi\sigma\Delta$, $z_3=-z_2$, $z_4=-z_1$, and $E_1/E_2=r$, $E_3=E_2$, $E_4=E_1$. Substituting these into \eqref{eqn:fprime} and \eqref{eqn:fsecond}, \eqref{eqn:R_eqn} gives $H'(\xi_c)=0$ and \eqref{eqn:K_eqn} gives $H''(\xi_c)=0$. Therefore $\xi_c$ is an interior double root and $\Delta$ is a critical separation. Since the formulas above for $s_1,s_2,z_1,z_2$ do not involve $\alpha$, the characterization is independent of $\alpha$ once $C=\Delta/(4\sqrt{\alpha})$ is imposed.
\end{proof}

\section{Animations}\label{app:animations}
See \href{https://github.com/shrikantchand/spectral_interference/blob/c54ce8e671c78dfbacbf872b686bc8c48c7015a6/asymptotics_SV_S1.gif}{here} for an animation of Theorems \ref{thm:asym_laplace_G} and \ref{thm:asym_laplace_sst} as $\Delta$ changes. Note that in this animation, the black curve is the approximation derived from the aforementioned theorems. To see how the approximations compare to the exact curve for the two forms of generalized synchrosqueezing transform, see Figure \ref{fig:asymptotics}.

See \href{https://www.desmos.com/calculator/7d3391ef3a}{here} for an animation of the approximation in Theorem \ref{thm:sst_const_asymp} as $\Delta$ varies with $C=\frac{\Delta}{4\sqrt{\alpha}}$. In the above link, the animation is of the approximation itself. Notice that since $C=\frac{\Delta}{4\sqrt{\alpha}}$, the intervals $I_2$ and $I_6$ do not exist. Additionally, the transition between one and two peaks as $\Delta$ varies can be seen clearly.

See \href{https://www.desmos.com/calculator/ruclp30ioh}{here} for an animation of the approximation in Theorem \ref{thm:sst_dest_asymp} as $\Delta$ varies with $C=\frac{\Delta}{4\sqrt{\alpha}}$.
\end{document}